\documentclass{amsart}

\usepackage{graphicx}
\usepackage[centertags]{amsmath}
\usepackage{amsfonts}
\usepackage{amssymb}
\usepackage{amsthm}
\usepackage{newlfont}
\usepackage{hyperref}
\usepackage{mathrsfs} 
\usepackage{yfonts} 

\newtheorem*{thm*}{Theorem}
\newtheorem{thm}{Theorem}[section]
\newtheorem{cor}[thm]{Corollary}
\newtheorem{lem}[thm]{Lemma}
\newtheorem{claim}{Claim}

\newtheorem{prop}[thm]{Proposition}

\theoremstyle{definition}
\newtheorem{defn}[thm]{Definition}

\theoremstyle{remark}

\numberwithin{equation}{section}

\DeclareSymbolFont{bbold}{U}{bbold}{m}{n}
\DeclareSymbolFontAlphabet{\mathbbold}{bbold}

\newcommand{\N}{\mathbb{N}}
\newcommand{\Z}{\mathbb{Z}}
\newcommand{\R}{\mathbb{R}}
\newcommand{\Q}{\mathbb{Q}}

\newcommand{\id}{\mathrm{id}}
\newcommand{\acts}{\curvearrowright}
\newcommand{\Stab}{\mathrm{Stab}}

\newcommand{\pmp}{p{$.$}m{$.$}p{$.$}}

\newcommand{\pv}{\bar{p}}
\newcommand{\qv}{\bar{q}}
\newcommand{\rv}{\bar{r}}

\newcommand{\pone}{\mathbbold{1}}
\newcommand{\cA}{\mathcal{A}}
\newcommand{\cF}{\mathcal{F}}

\newcommand{\cP}{\mathcal{P}}
\newcommand{\cQ}{\mathcal{Q}}

\newcommand{\E}{\mathscr{E}}
\newcommand{\sP}{\mathscr{P}}
\newcommand{\M}{\mathfrak{M}}

\newcommand{\dist}{\mathrm{dist}}

\newcommand{\rng}{\mathrm{rng}}
\newcommand{\symd}{\triangle}
\newcommand{\salg}{\sigma \text{-}\mathrm{alg}}

\newcommand{\sH}{\mathrm{H}}
\newcommand{\dB}{d^{\mathrm{alg}}} 
\newcommand{\dR}{d^{\mathrm{Rok}}} 
\newcommand{\dC}{d^*} 
\newcommand{\Prt}{\mathscr{P}} 
\newcommand{\rh}{h}
\newcommand{\ksh}{h^{\mathrm{KS}}}

\newcommand{\Borel}{\mathcal{B}}
\newcommand{\given}{\mathbin{|}}

\renewcommand{\:}{\,:\,}
\newcommand{\res}{\restriction}

\newcommand{\Prob}{\mathrm{Prob}}
\newcommand{\Nbr}{\mathcal{N}}

\newcommand{\BL}{\mathscr{L}}

\newcommand{\sinv}{\mathscr{I}}

\newcommand{\Def}[4]{\mathscr{D}_{#1}^{#2}(#3 \given #4)}
\newcommand{\Bp}[3]{\mathscr{B}_{#1}(#2 \mathbin{;} #3)}

\newcommand{\psd}{\delta}
\newcommand{\aprod}{\rtimes_{\mathrm{alg}}}
\newcommand{\Sub}{\mathrm{Sub}}

\newcommand{\sinai}{Sinai}

\begin{document}

\title[Positive entropy actions factor onto Bernoulli shifts]{Positive entropy actions of countable groups factor onto Bernoulli shifts}
\author{Brandon Seward}
\address{Courant Institute of Mathematical Sciences, New York University, 251 Mercer Street, New York, NY 10003, U.S.A.}
\email{b.m.seward@gmail.com}
\keywords{Sinai's factor theorem, Bernoulli shift, Rokhlin entropy, sofic entropy, non-amenable groups, completely positive entropy, Koopman representation}
\subjclass[2010]{37A35, 37A15}

\begin{abstract}
We prove that if a free ergodic action of a countably infinite group has positive Rokhlin entropy (or, less generally, positive sofic entropy) then it factors onto all Bernoulli shifts of lesser or equal entropy. This extends to all countably infinite groups the well-known {\sinai} factor theorem from classical entropy theory.
\end{abstract}
\maketitle

\section{Introduction}

For a countable group $G$ and a standard probability space $(L, \lambda)$, the \emph{Bernoulli shift} over $G$ with base space $(L, \lambda)$ is the product measure space $(L^G, \lambda^G)$ together with the left shift-action of $G$: for $g \in G$ and $x \in L^G$, $g \cdot x$ is defined by $(g \cdot x)(t) = x(g^{-1} t)$ for $t \in G$. For $n \in \N$ we write $n^G$ for $\{0, \ldots, n - 1\}^G$ and $u_n$ for the normalized counting measure on $\{0, \ldots, n-1\}$. The \emph{Shannon entropy} of $(L, \lambda)$, denoted $\sH(L, \lambda)$, is $\sum_{\ell \in L} - \lambda(\ell) \log \lambda(\ell)$ if $\lambda$ has countable support and is $\infty$ otherwise. We note that $\sH(n, u_n) = \log(n)$.

Kolmogorov--{\sinai} entropy, which we denote $\ksh_\Z$, was introduced in 1958 in order to provide a negative answer to the roughly 20 year-old question of von Neumann that asked whether the Bernoulli shifts $\Z \acts (2^\Z, u_2^\Z)$ and $\Z \acts (3^\Z, u_3^\Z)$ are isomorphic \cite{Ko58,Ko59,Si59}. Shortly after, {\sinai} proved the following famous theorem.

\begin{thm*}[{\sinai}'s factor theorem, 1962 \cite{Si62}]
Let $\Z \acts (X, \mu)$ be a free ergodic {\pmp} action. If $(L, \lambda)$ is a probability space with $\ksh_\Z(X, \mu) \geq \sH(L, \lambda)$, then $\Z \acts (X, \mu)$ factors onto the Bernoulli shift $\Z \acts (L^\Z, \lambda^\Z)$.
\end{thm*}

One reason for this theorem's significance is its applicability. Specifically, the conclusion is quite strong as it is extremely difficult to construct Bernoulli factors. In fact, the conclusion was not even previously known in the case where $\Z \acts (X, \mu)$ is itself a Bernoulli shift. On the other hand, the assumption of the theorem is by comparison much easier to verify, as many practical techniques are known for obtaining bounds on Kolmogorov--{\sinai} entropy.

This theorem is also important for aesthetic and philosophical reasons. Previously, entropy merely happened to be a tool capable of showing the non-isomorphism of some $\Z$-Bernoulli shifts. However, {\sinai}'s theorem indicates that there is a deeper connection at play, and that entropy and Bernoullicity are intricately linked to one another. His theorem reveals that in fact Bernoulli shifts are the source of all positive entropy (as the original action has $0$ relative entropy over a full-entropy Bernoulli factor).

{\sinai}'s theorem also played a significant role in the historical development of entropy theory and our understanding of Bernoulli shifts. Specifically, {\sinai}'s theorem inspired the profound work of Ornstein that followed. Building off of {\sinai}'s theorem and the methods of its proof, Ornstein proved his famous isomorphism theorem stating that $\Z$-Bernoulli shifts are completely classified up to isomorphism by their Kolmogorov--{\sinai} entropy (equivalently, the Shannon entropy of their base) \cite{Or70a,Or70b}. Ornstein and other researchers continued to cultivate these techniques, ultimately creating `Ornstein theory,' a collection of practical necessary-and-sufficient conditions for $\Z$-actions to be isomorphic to Bernoulli. This led to the surprising discovery that many natural examples of $\Z$-actions turn out to be isomorphic to Bernoulli, such as: factors of Bernoulli shifts \cite{Or71}, inverse limits of Bernoulli shifts \cite{Or71b}, ergodic automorphisms of compact metrizable groups \cite{L77,MT78}, mixing Markov chains \cite{FO70}, geodesic flows on surfaces of negative curvature \cite{OW73}, Anosov flows with a smooth measure \cite{R74}, and two-dimensional billiards with a convex scatterer \cite{GO74}. These discoveries are generally considered to be the deepest results in entropy theory.

Over time, {\sinai}'s theorem was generalized in a few different ways. First, it was generalized to the realm of relative entropy by Thouvenot.

\begin{thm*}[The relative factor theorem, Thouvenot 1975 \cite{Th75} (see also \cite{Or70b})]
Let $\Z \acts (X, \mu)$ be a free ergodic {\pmp} action, let $\Z \acts (Y, \nu)$ be a factor action given by the map $\phi : X \rightarrow Y$, and set $\cF = \phi^{-1}(\Borel(Y))$. If $(L, \lambda)$ is a probability space with $\ksh_\Z(X, \mu \given \cF) \geq \sH(L, \lambda)$, then there is a map $\psi : X \rightarrow L^\Z$ so that $\psi \times \phi$ is a factor map from $\Z \acts (X, \mu)$ onto the direct-product action $\Z \acts (L^\Z \times Y, \lambda^\Z \times \nu)$.
\end{thm*}

There was also a perturbative factor theorem established in the course of {\sinai}'s original proof which Thouvenot relativized in obtaining the above theorem. We refer the reader to Section \ref{sec:prelim} for the relevant definitions.

\begin{thm*}[The perturbative relative factor theorem, Thouvenot 1975 \cite{Th75} (see also \cite{Or70b})]
There exists a function $\psd: \N_+ \times \R_+ \rightarrow \R_+$ with the following property. Let $\Z \acts (X, \mu)$ be a free ergodic {\pmp} action, let $\cF$ be a $\Z$-invariant sub-$\sigma$-algebra, let $k \in \N$, and let $\epsilon > 0$. If $\xi$ is a $k$-piece ordered partition of $X$ and $\pv$ is a length-$k$ probability vector with $\sH(\pv) \leq \ksh_\Z(X, \mu \given \cF)$ and
$$|\dist(\xi) - \pv| + |\ksh_\Z(\xi \given \cF) - \sH(\pv)| < \psd(k, \epsilon),$$
then there is an ordered partition $\alpha$ with $\dist(\alpha) = \pv$ and $\dB_\mu(\alpha, \xi) < \epsilon$ such that the $\Z$-translates of $\alpha$ are mutually independent and $\salg_\Z(\alpha)$ is independent with $\cF$.
\end{thm*}

Later, the factor theorem was extended to non-ergodic actions by Kieffer and Rahe. Below, for a Borel action $G \acts X$ we let $\E_G(X)$ denote the set of invariant ergodic Borel probability measures.

\begin{thm*}[The non-ergodic factor theorem, Kieffer--Rahe 1981 \cite{KiRa}]
Let $\Z \acts X$ be a free Borel action on a standard Borel space $X$. If $\nu \mapsto \lambda_\nu$ is a Borel map associating to every $\nu \in \E_\Z(X)$ a probability measure $\lambda_\nu$ on $[0, 1]$ satisfying $\ksh_\Z(X, \nu) \geq \sH([0, 1], \lambda_\nu)$, then there is a $\Z$-equivariant factor map $\phi : X \rightarrow [0, 1]^\Z$ such that $\phi_*(\nu) = \lambda_\nu^\Z$ for every $\nu \in \E_\Z(X)$.
\end{thm*}

Finally, the most stream-lined proof came from the residual theorem of Burton, Keane, and Serafin.

\begin{thm*}[The residual factor theorem, Burton--Keane--Serafin 2000 \cite{BKS}]
Let $A$ be a finite set and let $\Z \acts (A^\Z, \mu)$ be a free ergodic {\pmp} action. If $(L, \lambda)$ is a finite probability space with $\ksh_\Z(A^\Z, \mu) = \sH(L, \lambda)$, then within the space of ergodic joinings of $\mu$ with $\lambda^\Z$, the set of joinings that come from graphs of factor maps, $\{(\id \times \phi)_*(\mu) : \phi : A^\Z \rightarrow L^\Z, \ \phi_*(\mu) = \lambda^\Z\}$, is a dense $G_\delta$ set in the weak$^*$-topology.
\end{thm*}

{\sinai}'s theorem was also extended to actions of other groups.  In 1972, Katznelson and Weiss extended the original factor theorem to free ergodic actions of $\Z^n$ \cite{KaW72}, and in 1987 Ornstein and Weiss generalized the original factor theorem to free ergodic actions of countable amenable groups \cite{OW87}. Later, Danilenko and Park extended Thouvenot's relative factor theorem to free ergodic actions of countable amenable groups in 2002 \cite{DP02}.

For countable non-amenable groups, the study of Bernoulli shifts and the development of an entropy theory were long considered unattainable. Today it is still open, for example, whether the Bernoulli shifts $G \acts (2^G, u_2^G)$ and $G \acts (3^G, u_3^G)$ are non-isomorphic for all countable groups $G$. However, dramatic advancements have recently occurred due to breakthrough work of Bowen in 2008.

Work of Bowen \cite{B10b}, combined with improvements by Kerr and Li \cite{KL11a,KL13}, created the notion of sofic entropy for {\pmp} actions of sofic groups. We remind the reader that the class of sofic groups contains the countable amenable groups, and it is an open problem whether every countable group is sofic. Sofic entropy is an extension of Kolmogorov--{\sinai} entropy, as when the acting group is amenable the two notions coincide \cite{B12,KL13}.  For sofic groups $G$, the Bernoulli shift $G \acts (L^G, \lambda^G)$ has sofic entropy $\sH(L, \lambda)$ as expected \cite{B10b,KL11b}, thus implying the non-isomorphism of many Bernoulli shifts.

In this paper, we will work with a different notion of entropy which was introduced by the author in 2014 \cite{S1} and is defined for actions of arbitrary (not necessarily sofic) countable groups. If $G$ is a countable group, $G \acts (X, \mu)$ is an ergodic {\pmp} action, and $\xi \subseteq \Borel(X)$ is a collection of sets, then we define the \emph{(outer) Rokhlin entropy of $\xi$} to be
$$\rh_G(\xi) = \inf \Big\{ \sH(\alpha) : \alpha \text{ a countable partition with } \xi \subseteq \salg_G(\alpha) \Big\},$$
where $\salg_G(\alpha)$ is the smallest $G$-invariant $\sigma$-algebra containing $\alpha$. When $\xi = \Borel(X)$ is the Borel $\sigma$-algebra of $X$, we simply write $\rh_G(X, \mu)$ for the \emph{Rokhlin entropy of $(X, \mu)$}. In this paper, $\rh_G$ will always denote Rokhlin entropy. We caution the reader that Rokhlin entropy is defined in a distinct way for non-ergodic actions; see Section \ref{sec:prelim}.

For free actions of amenable groups, Rokhlin entropy coincides with Kolmogorov--{\sinai} entropy \cite[Cor. 7.2]{AS}. Furthermore, for free actions of sofic groups, sofic entropy is a lower-bound to Rokhlin entropy \cite[Prop. 1.10]{AS}, and it is an important open question whether the two coincide (excluding cases where sofic entropy is minus infinity). The equality $\rh_G(L^G, \lambda^G) = \sH(L, \lambda)$ holds for sofic groups $G$, and it is conjectured to hold for all countably infinite groups \cite{S2}.

At face value there's not necessarily any good reason to expect the factor theorem to hold for non-amenable groups. Specifically, we know Bernoulli shifts behave in drastically different ways for non-amenable groups. For example, for amenable groups factors of Bernoulli shifts, inverse limits of Bernoulli shifts, and d-bar limits of Bernoulli shifts are all isomorphic to Bernoulli shifts, but all of these assertions are false for non-amenable groups. Furthermore, all prior proofs of {\sinai}'s factor theorem and its variations rely critically upon three properties: the Rokhlin lemma, the Shannon-McMillan-Breiman theorem, and the monotonicity of entropy under factor maps. Yet all three of these properties fail miserably for non-amenable groups. Finally, a concerning fact is that there are simple examples showing that the residual factor theorem is false for non-amenable groups (see Proposition \ref{prop:resfail}). Despite these warning signs, we generalize the factor theorem to all countably infinite groups. For simplicity, here we state our main theorem only in its most basic form. (For the strongest versions, see Section \ref{sec:sinai}).

\begin{thm} \label{intro:mysinai}
Let $G$ be a countably infinite group and let $G \acts (X, \mu)$ be a free ergodic {\pmp} action. If $(L, \lambda)$ is a probability space with $\rh_G(X, \mu) \geq \sH(L, \lambda)$ then $G \acts (X, \mu)$ factors onto the Bernoulli shift $G \acts (L^G, \lambda^G)$.
\end{thm}

Let us recast the above theorem in a more concrete form that highlights the simplicity of the assumption. Recall that a Borel partition $\alpha$ of $X$ is \emph{generating} if $\salg_G(\alpha) = \Borel(X)$ mod null sets. The classical finite generator theorem of Krieger extends to Rokhlin entropy \cite{S1}, allowing the following reformulation of Theorem \ref{intro:mysinai}.

\begin{cor}
Let $G$ be a countably infinite group and let $G \acts (X, \mu)$ be a free ergodic {\pmp} action. If $G \acts (X, \mu)$ does not admit any $n$-piece generating partition then it factors onto the Bernoulli shift $(n^G, u_n^G)$. If $G \acts (X, \mu)$ does not admit any finite generating partition, then it factors onto every $G$-Bernoulli shift.
\end{cor}

Since sofic entropy is a lower-bound to Rokhlin entropy \cite[Prop. 1.10]{AS}, as a bonus we automatically obtain {\sinai}'s factor theorem for sofic entropy as well.

\begin{cor} \label{intro:sofic}
Let $G$ be a sofic group and let $G \acts (X, \mu)$ be a free ergodic {\pmp} action. If $(L, \lambda)$ is a probability space and the sofic entropy of $G \acts (X, \mu)$ (with respect to some sofic approximation of $G$) is greater than or equal to $\sH(L, \lambda)$, then $G \acts (X, \mu)$ factors onto the Bernoulli shift $G \acts (L^G, \lambda^G)$.
\end{cor}

A recent striking result of Bowen states that when $G$ is non-amenable, all $G$-Bernoulli shifts factor onto one another \cite{B17}. This leads to the following corollary.

\begin{cor}
Let $G$ be a countable non-amenable group and let $G \acts (X, \mu)$ be a free ergodic {\pmp} action. If $\rh_G(X, \mu) > 0$ then $G \acts (X, \mu)$ factors onto every $G$-Bernoulli shift.
\end{cor}

In particular, the converse of Theorem \ref{intro:mysinai} fails for non-amenable groups (but holds for amenable groups by monotonicity of entropy).

As with the original factor theorem, we believe the conclusion of Theorem \ref{intro:mysinai} is strong. In fact, the conclusion was not previously known in the case where $G \acts (X, \mu)$ is itself a Bernoulli shift (though this was implied by independent results obtained later \cite{B17,S}). To point to other concrete examples, we mention that Hayes has computed the sofic entropy of Gaussian actions and principal algebraic actions of sofic groups \cite{H16,Hb}, many of which have positive sofic entropy and thus posses Bernoulli factors by our theorem. Let us describe more precisely this latter class of actions. Let $G$ be a sofic group and let $f \in \Z G$ be an element of the group ring. Let $X_f$ be the compact abelian group (using coordinate-wise addition) of $x \in (\R / \Z)^G$ such that the right-convolution $x * f$ is trivial mod $\Z^G$, and let $m$ be the Haar probability measure on $X_f$. Hayes proved that if $f$ is injective as a left-convolution operator on $\ell^2(G)$, then the sofic entropy of $G \acts (X_f, m)$ equals $\int_{[0, \infty)} \log(t) \ d \mu_{|f|}(t)$ (the logarithm of the Fuglede--Kadison determinant of $f$), where $\mu_{|f|}$ is the spectral measure of the operator $|f|$ acting on $\ell^2(G)$. Many explicit examples will give positive entropy values (see for instance \cite[Prop. 4.16.(iii) and Thm. 6.7]{KL13b} and \cite[Thm. 1.2]{H17}), but directly finding Bernoulli factors would appear to be extremely difficult without our theorem.

To give a brief but intriguing glimpse into the applicability of Theorem \ref{intro:mysinai}, we mention the following. This corollary relies upon the recent powerful result of Bowen \cite{B17} that all Bernoulli shifts over non-amenable groups satisfy the measurable von Neumann--Day conjecture.

\begin{cor}
Let $G$ be a countable non-amenable group and let $G \acts (X, \mu)$ be a free ergodic action. If $\rh_G(X, \mu) > 0$ then this action satisfies the measurable von Neumann--Day conjecture; specifically there is a free action of the rank $2$ free group $F_2 \acts (X, \mu)$ such that $\mu$-almost-every $F_2$-orbit is contained in a $G$-orbit.
\end{cor}

In analogy with the original factor theorem, we believe Theorem \ref{intro:mysinai} holds philosophical and aesthetic value. While sofic entropy and Rokhlin entropy are useful for distinguishing Bernoulli shifts, its not at all clear a priori what properties of the action are reflected in these entropies. Our theorem indicates the surprising fact that both sofic entropy and Rokhlin entropy continue to be intimately tied to Bernoullicity. Specifically, our theorem shows that all positive entropy phenomena comes from Bernoulli shifts. (However, the picture is not yet as clear as in the classical setting; see Section \ref{sec:discuss}).

Theorem \ref{intro:mysinai} is only the most basic form of our main theorem. Specifically, we prove relative, perturbative, and non-ergodic factor theorems. With the exception of the residual factor theorem which is known to be false for non-amenable groups, our work extends all prior versions of the factor theorem to general countably infinite groups. Furthermore, through the use of non-free variants of Bernoulli shifts, our theorems apply to non-free actions as well.

We believe Theorem \ref{intro:mysinai} and Corollary \ref{intro:sofic} mark an important step forward in the development of both sofic entropy theory and Rokhlin entropy theory. An ambitious long-term goal in this field is the development of an Ornstein theory for non-amenable groups. To a large extent, our understanding of Bernoulli shifts over non-amenable groups is still quite poor. One peculiar exception is the (positive direction of) the Ornstein isomorphism theorem itself. Work of Bowen \cite{B12b} and recent work of the author \cite{S} (occurring after the present results were obtained) show that for every countably infinite group, Bernoulli shifts with base spaces of equal Shannon entropy must be isomorphic (in particular, Bernoulli shifts over sofic groups are completely classified by their sofic entropy, i.e. the Shannon entropy of their base). This result is peculiar in the sense that the proof relies upon very restrictive and specialized tricks that do not use any entropy theory aside from the original Ornstein isomorphism theorem for $\Z$. The proof is an isolated achievement, seemingly incapable of generalizing or of revealing more about the structure of Bernoulli shifts. On the other hand, we believe the techniques and theorems of this paper may lay a new path to studying Bernoulli shifts, just as the original factor theorem historically served as the foundation for Ornstein theory.

One immediate consequence of our main theorem is that Bernoulli shifts $G \acts (L^G, \lambda^G)$ are finitely-determined whenever $\rh_G(L^G, \lambda^G) = \sH(L, \lambda)$. Recall that for a finite set $L$ and two $G$-invariant probability measures $\mu, \nu$ on $L^G$, their d-bar distance $\bar{d}(\mu, \nu)$ is defined to be the infimum of
$$\lambda(\{(x, y) \in L^G \times L^G : x(1_G) \neq y(1_G)\}),$$
where $\lambda$ ranges over all joinings of $\mu$ with $\nu$. The measure $\mu$ is \emph{finitely-determined} if for every $\epsilon > 0$ there is a weak$^*$-open neighborhood $U$ of $\mu$ and $\delta > 0$ such that whenever $\nu \in U$ satisfies $\Stab_*(\nu) = \Stab_*(\mu)$ and $|\rh_G(L^G, \nu) - \rh_G(L^G, \mu)| < \delta$ we have $\bar{d}(\mu, \nu) < \epsilon$. The notion of finitely-determined played a prominent role in Ornstein theory, and it was proven that for an amenable group $G$ a measure $\mu$ on $L^G$ is finitely-determined if and only if $G \acts (L^G, \mu)$ is isomorphic to a Bernoulli shift \cite{OW87}.

\begin{cor} \label{intro:fd}
Let $G$ be a countably infinite group and let $(L, \lambda)$ be a finite probability space. If $\rh_G(L^G, \lambda^G) = \sH(L, \lambda)$ then $\lambda^G$ is finitely-determined.
\end{cor}

The techniques we create to prove our main theorem also lead to new insights on the spectral structure of actions having positive Rokhlin entropy. This will be presented in upcoming work \cite{Sa} and will mirror and expand upon similar results obtained by Hayes in the setting of sofic entropy \cite{H, Ha}. In fact, we would like to explicitly mention that our main theorem directly combines with the work of Hayes \cite{H, Ha} to produce the following corollary (see \cite{Ha} for relevant notation). This corollary extends a similar result for actions of amenable groups obtained by Dooley--Golodets \cite{DoGo}.

\begin{cor}
Let $G$ be a sofic group. Let $G \acts (X, \mu)$ be a free ergodic {\pmp} action and let $G \acts (Y, \nu)$ and $G \acts (Y_0, \nu_0)$ be the (sofic entropy) Pinsker factor and the (sofic entropy) outer Pinsker factor, respectively. Then, as a representation of $\BL^\infty (Y) \aprod G$, we have that $\BL^2(X) \ominus \BL^2(Y)$ is isomorphic to $\BL^2(Y, \ell^2(G))^{\infty}$. Similarly, $\BL^2(X) \ominus \BL^2(Y_0)$ is isomorphic to $\BL^2(Y_0, \ell^2(G))^{\infty}$.
\end{cor}

In \cite{Sa} this corollary will be adapted to Rokhlin entropy and extended to actions of general countable groups.

\subsection*{Acknowledgments}
The author thanks Tim Austin for pointing out the consequence that Bernoulli measures are finitely determined, Lewis Bowen for permission to include the fact that the residual factor theorem fails for non-amenable groups, and Benjy Weiss for referring the author to the very useful non-ergodic factor theorem of Kieffer and Rahe. The author is grateful to
Mikl\'{o}s Ab\'{e}rt,
Tim Austin,
Lewis Bowen,
Damien Gaboriau,
Ben Hayes,
Mike Hochman,
Ralf Spatzier,
and Benjy Weiss
for valuable conversations and encouragement and thanks the referee for a careful reading. The author was partially supported by ERC grant 306494 and Simons Foundation grant 328027 (P.I. Tim Austin).

\section{Preliminaries} \label{sec:prelim}

Our probability spaces $(X, \mu)$ will always be assumed to be standard, meaning $X$ is a standard Borel space, with its Borel $\sigma$-algebra denoted $\Borel(X)$, and $\mu$ is a Borel probability measure. The space $\Prob(X)$ of all Borel probability measures on $X$ is itself a standard Borel space. Its Borel $\sigma$-algebra is the smallest $\sigma$-algebra making all of the maps $\mu \mapsto \mu(A)$ measurable for every $A \in \Borel(X)$.

Recall that for every sub-$\sigma$-algebra $\Sigma$ there is a probability space $(Y, \nu)$, unique up to isomorphism, and a measure-preserving map $\phi : X \rightarrow Y$ with $\phi^{-1}(\Borel(Y)) = \Sigma$. Furthermore, there is a $\nu$-almost-everywhere unique Borel map $y \in Y \mapsto \mu_y \in \Prob(X)$ satisfying $\mu_y(\phi^{-1}(y)) = 1$ and $\mu = \int_Y \mu_y \ d \nu(y)$. We refer to this as the \emph{disintegration of $\mu$ over $\nu$} or the \emph{disintegration of $\mu$ with respect to $\phi$}. If $\alpha$ is a countable Borel partition of $X$ then the \emph{Shannon entropy of $\alpha$} is
$$\sH(\alpha) = \sH_\mu(\alpha) = \sum_{A \in \alpha} - \mu(A) \log \mu(A),$$
and the \emph{conditional Shannon entropy of $\alpha$ relative to $\Sigma$} is
$$\sH(\alpha \given \Sigma) = \int_Y \sH_{\mu_y}(\alpha) \ d \nu(y).$$
If $\beta$ is another countable Borel partition of $X$ then we write $\sH(\alpha \given \beta)$ for $\sH(\alpha \given \salg(\beta))$. We will assume familiarity with the basic properties of conditional Shannon entropy; see \cite{Do11} for reference.

A collection of $\sigma$-algebras $(\cF_i)_{i \in I}$ on $(X, \mu)$ are said to be \emph{mutually independent} if for any disjoint subsets $J_1, J_2 \subseteq I$ and sets $A_i \in \bigvee_{k \in J_i} \cF_k$ we have $\mu(A_1 \cap A_2) = \mu(A_1) \mu(A_2)$. More generally, if $\Sigma$ is another $\sigma$-algebra, $\phi : (X, \mu) \rightarrow (Y, \nu)$ is the associated factor, and $\mu = \int_Y \mu_y \ d \nu(y)$ is the disintegration of $\mu$ over $\nu$, then we say that $(\cF_i)_{i \in I}$ are \emph{mutually independent relative to $\Sigma$} if for any disjoint subsets $J_1, J_2 \subseteq I$ and sets $A_i \in \bigvee_{k \in J_i} \cF_k$ we have $\mu_y(A_1 \cap A_2) = \mu_y(A_1) \mu_y(A_2)$ for $\nu$-almost-every $y \in Y$. In the case $|I| = 2$, we simply say that the two $\sigma$-algebras are independent (or independent relative to $\Sigma$). By associating a (possibly uncountable) partition $\alpha$ with the $\sigma$-algebra of Borel sets that can be expressed as unions of sets from $\alpha$, and by associating a function with the smallest $\sigma$-algebra making it measurable, we can also speak of independence of partitions and of functions. We record a basic lemma.

\begin{lem} \label{lem:ind}
Let $(X, \mu)$ be a standard probability space, let $\alpha$ be a partition, and let $\Sigma$ and $\cF$ be sub-$\sigma$-algebras. If $\alpha$ is independent with $\Sigma$ relative to $\cF$ and $\alpha$ is independent with $\cF$, then $\alpha$ is independent with $\Sigma$.
\end{lem}

\begin{proof}
Let $\phi : (X, \mu) \rightarrow (Y, \nu)$ be the factor associated to $\cF$, and let $\mu = \int_Y \mu_y \ d \nu(y)$ be the corresponding disintegration of $\mu$. Fix $A \in \Borel(X)$ that is a union of sets from $\alpha$ and fix $C \in \Sigma$. Since $\alpha$ is independent with $\Sigma$ relative to $\cF$, we have that $\mu_y(A \cap C) = \mu_y(A) \mu_y(C)$ for $\nu$-almost-every $y \in Y$. Since $\alpha$ is also independent with $\cF$, we have that $\mu_y(A) = \mu(A)$ for $\nu$-almost-every $y \in Y$. Thus $\mu_y(A \cap C) = \mu(A) \mu_y(C)$ for $\nu$-almost-every $y \in Y$. Therefore
$$\mu(A \cap C) = \int_Y \mu_y(A \cap C) \ d \nu(y) = \int_Y \mu(A) \mu_y(C) \ d \nu(y) = \mu(A) \mu(C).$$
We conclude that $\alpha$ and $\Sigma$ are independent.
\end{proof}

A \emph{probability vector} $\pv$ is a finite or countable tuple $(p_i)_{0 \leq i < |\pv|}$ of non-negative real numbers that sum to $1$. We write $|\pv|$ for the length of $\pv$. If $\qv$ is another probability vector then their $\ell^1$-distance is
$$|\pv - \qv| = \sum_i |p_i - q_i|,$$
with the convention that $p_i = 0$ when $i \geq |\pv|$ and $q_j = 0$ when $j \geq |\qv|$.

In this paper we will implicitly assume that all countable partitions $\alpha$ of $(X, \mu)$ are both Borel and ordered: $\alpha = \{A_i : 0 \leq i < |\alpha|\}$. We write $\dist(\alpha)$, or $\dist_\mu(\alpha)$, for the \emph{distribution of $\alpha$}, which is the probability vector having $i^{\text{th}}$ term $\mu(A_i)$. If $\beta = \{B_i : 0 \leq i < |\beta|\}$ is another ordered partition, then their distance in the measure algebra is
$$\dB_\mu(\alpha, \beta) = \sum_i \mu(A_i \symd B_i),$$
with the convention $A_i = \varnothing$ when $i \geq |\alpha|$ and $B_j = \varnothing$ when $j \geq |\beta|$. Notice that
$$|\dist(\alpha) - \dist(\beta)| \leq \dB_\mu(\alpha, \beta).$$
We say that $\beta$ refines $\alpha$, written $\beta \geq \alpha$, if each $A \in \alpha$ is a union of members of $\beta$. Similarly, we say $\qv$ is a refinement of $\pv$ if there is a function $f : \{0, \ldots, |\qv|\} \rightarrow \{0, \ldots, |\pv|\}$ satisfying
$$\forall 0 \leq i < |\pv| \qquad \sum_{j \in f^{-1}(i)} q_j = p_i.$$

For a countable group $G$ and a Borel action $G \acts X$ we let $\M_G(X)$ and $\E_G(X)$ denote the sets of invariant Borel probability measures and invariant ergodic Borel probability measures, respectively. We also write $\sinv = \sinv_G$ for the $\sigma$-algebra of $G$-invariant Borel subsets of $X$. The action $G \acts X$ is \emph{aperiodic} if every orbit is infinite. We similarly call a {\pmp} action $G \acts (X, \mu)$ \emph{aperiodic} if $\mu$-almost-every orbit is infinite. We write $\Stab(x) = \{g \in G : g \cdot x = x\}$ for the stabilizer of $x$. For $\mathcal{C} \subseteq \Borel(X)$ we write $\salg(\mathcal{C})$ for the $\sigma$-algebra on $X$ generated by $\mathcal{C}$ and $\salg_G(\mathcal{C})$ for the $G$-invariant $\sigma$-algebra generated by $\mathcal{C}$. A sub-$\sigma$-algebra $\Sigma$ is \emph{countably generated} if there is a countable collection $\mathcal{C}$ so that $\salg(\mathcal{C}) = \Sigma$ (equality in the purely Borel sense). Of course, for any fixed probability measure $\mu$ on $X$, every sub-$\sigma$-algebra $\Sigma$ coincides, modulo $\mu$-null sets, with a countably generated $\sigma$-algebra. If $\Sigma$ is a $G$-invariant sub-$\sigma$-algebra and $(Y, \nu)$ is the corresponding factor via $\phi : X \rightarrow Y$, then there is a $\nu$-almost-everywhere unique {\pmp} action of $G$ on $(Y, \nu)$ so that $\phi$ is $G$-equivariant. We say that $\Sigma$ (or $\phi$) is \emph{class-bijective} if $\Stab(\phi(x)) = \Stab(x)$ for $\mu$-almost-every $x \in X$, and we say that $\mu \res \Sigma$ is ergodic if $\phi_*(\mu)$ is ergodic (or equivalently, if every $G$-invariant set in $\Sigma$ is either $\mu$-null or $\mu$-conull). In the purely Borel context, we call $\Sigma$ class-bijective if it is countably generated and for all $x \neq y \in X$ lying in a common $G$-orbit there is $A \in \Sigma$ with $x \in A$ and $y \not\in A$. For a partition $\alpha$ of $X$ and $g \in G$ we set $g \cdot \alpha = \{g \cdot A : A \in \alpha\}$, and for a finite set $W \subseteq G$ we set $\alpha^W = \bigvee_{w \in W} w^{-1} \cdot \alpha$.

If $G$ is a countable group, $G \acts (X, \mu)$ is a {\pmp} action, $\cF$ is a $G$-invariant sub-$\sigma$-algebra, and $\xi \subseteq \Borel(X)$ is a collection of sets, then we define the \emph{(outer) Rokhlin entropy of $\xi$ relative to $\cF$} to be
$$\rh_G(\xi \given \cF) = \inf \Big\{ \sH(\alpha \given \cF \vee \sinv) : \alpha \text{ a countable partition with } \xi \subseteq \salg_G(\alpha) \vee \cF \vee \sinv \Big\}.$$
When we wish to emphasize the measure $\mu$ we write $\rh_{G,\mu}(\xi \given \cF)$. Notice that for ergodic actions $\sinv = \{\varnothing, X\}$ modulo null sets. When $\xi = \Borel(X)$ is the Borel $\sigma$-algebra of $X$, we simply write $\rh_G(X, \mu \given \cF)$ for the \emph{Rokhlin entropy of $(X, \mu)$ relative to $\cF$}.

One of the key features of Rokhlin entropy that we will use is that it is countably sub-additive \cite[Cor. 4.6]{AS}. Specifically, fix a {\pmp} action $G \acts (X, \mu)$, let $\xi \subseteq \Borel(X)$, let $\cF$ be a $G$-invariant $\sigma$-algebra, and let $(\Sigma_n)_{n \geq 1}$ be an increasing sequence of $G$-invariant $\sigma$-algebras with $\xi \subseteq \cF \vee \bigvee_n \Sigma_n$. Then
$$\rh_G(\xi \given \cF) \leq \rh_G(\Sigma_1 \given \cF) + \sum_{n = 2}^\infty \rh_G(\Sigma_n \given \Sigma_{n-1} \vee \cF).$$

We will also need the following fact.

\begin{lem} \label{lem:enlarge}
Let $X$ be a standard Borel space and let $G \acts X$ be an aperiodic Borel action. For every $\epsilon > 0$ there exists a countably generated class-bijective $G$-invariant sub-$\sigma$-algebra $\Omega$ with $\rh_{G,\mu}(\Omega) < \epsilon$ for all $\mu \in \M_G(X)$. In particular, if $\Sigma$ is a $G$-invariant sub-$\sigma$-algebra, $\xi \subseteq \Borel(X)$, and $\mu \in \M_G(X)$, then setting $\cF = \Sigma \vee \Omega$ we have that $\cF$ is class-bijective and
$$\rh_{G,\mu}(\xi \given \Sigma) - \epsilon \leq \rh_{G,\mu}(\xi \given \cF) \leq \rh_{G,\mu}(\xi \given \Sigma).$$
\end{lem}

\begin{proof}
By work of the author and Tucker-Drob (see \cite{ST14} for free actions; the case of non-free actions is in preparation), there is a $G$-equivariant class-bijective Borel map $\phi : X \rightarrow \{0, 1\}^G$ satisfying $\rh_G(\{0,1\}^G, \phi_*(\mu)) < \epsilon$ for all $\mu \in \M_G(X)$. Set $\Omega = \phi^{-1}(\Borel(\{0,1\}^G))$. Clearly $\rh_{G,\mu}(\Omega) < \epsilon$ for all $\mu \in \M_G(X)$. Now fix $\Sigma$, $\xi$, and $\mu$ and set $\cF = \Sigma \vee \Omega$. Clearly $\rh_{G,\mu}(\xi \given \cF) \leq \rh_{G,\mu}(\xi \given \Sigma)$ since $\cF \supseteq \Sigma$, and the other inequality follows from sub-additivity:
\begin{equation*}
\rh_{G,\mu}(\xi \given \Sigma) \leq \rh_{G,\mu}(\cF \given \Sigma) + \rh_{G,\mu}(\xi \given \cF) \leq \epsilon + \rh_{G,\mu}(\xi \given \cF)\qedhere
\end{equation*}
\end{proof}

\section{Non-free Bernoulli shifts and Bernoulli partitions} \label{sec:nonfree}

For a countable group $G$ we let $\Sub(G)$ denote the space of all subgroups of $G$. A base for the topology on $\Sub(G)$ is given by the basic open sets $\{H \in \Sub(G) : H \cap T = F\}$ as $F \subseteq T$ range over the finite subsets of $G$. An \emph{invariant random subgroup}, or \emph{IRS}, of $G$ is a Borel probability measure $\theta$ on $\Sub(G)$ which is invariant under the conjugation action of $G$. This concept was first introduced in \cite{AGV}. Every {\pmp} action $G \acts (X, \mu)$ produces an IRS $\Stab_*(\mu)$ via the push-forward of $\mu$ under the stabilizer map $\Stab : X \rightarrow \Sub(G)$. We call $\Stab_*(\mu)$ the \emph{stabilizer type} of $G \acts (X, \mu)$.

\begin{lem} \label{lem:stabequal}
Let $G \acts (X, \mu)$ and $G \acts (Y, \nu)$ be {\pmp} actions and let $\phi : X \rightarrow Y$ be a factor map. The following are equivalent:
\begin{enumerate}
\item [\rm (i)] $\phi$ is class-bijective;
\item [\rm (ii)] $\Stab(\phi(x)) = \Stab(x)$ for $\mu$-almost-every $x \in X$;
\item [\rm (iii)] $\Stab_*(\mu) = \Stab_*(\nu)$.
\end{enumerate}
\end{lem}

\begin{proof}
(i) and (ii) are equivalent by definition, and (ii) clearly implies (iii). We will argue (iii) implies (ii). So assume $\Stab_*(\mu) = \Stab_*(\nu)$. Then for every $g \in G$ we have
\begin{align*}
0 & = \Stab_*(\nu)(\{\Gamma \leq G : g \in \Gamma\}) - \Stab_*(\mu)(\{\Gamma \leq G : g \in \Gamma\})\\
 & = \mu(\{x \in X : g \in \Stab(\phi(x)) \setminus \Stab(x)\}).
\end{align*}
As $G$ is countable it follows that $\mu(\{x \in X : \Stab(\phi(x)) \neq \Stab(x)\}) = 0$.
\end{proof}

Every IRS of $G$ is the stabilizer type of some {\pmp} action of $G$ \cite{AGV}. This fact follows from the construction of non-free Bernoulli shifts which we now discuss. Let $(L, \lambda)$ be a standard probability space. We let $G$ act on $L^G$ by the standard left-shift action: $(g \cdot x)(t) = x(g^{-1} t)$ for $g, t \in G$ and $x \in L^G$. For $H \in \Sub(G)$, we identify $L^{H \backslash G}$ with the set of points $x \in L^G$ with $H \subseteq \Stab(x)$, and we consider the corresponding Borel probability measure $\lambda^{H \backslash G}$ on $L^G$ which is supported on $L^{H \backslash G}$. If $\theta$ is an IRS of $G$ which is supported on the infinite-index subgroups of $G$, then we define the \emph{type-$\theta$ Bernoulli shift with base space $(L, \lambda)$} to be the standard shift-action of $G$ on $L^G$ equipped with the $G$-invariant probability measure
$$\lambda^{\theta \backslash G} : = \int_{H \in \Sub(G)} \lambda^{H \backslash G} \ d \theta(H).$$
We say simply `Bernoulli shift' when we are referring to the type-$\delta_{\{1_G\}}$ Bernoulli shift, where $\delta_{\{1_G\}}$ is the single-point mass. If $H \in \Sub(G)$ has infinite index in $G$ and $\lambda$ is non-trivial, then $\Stab(x) = H$ for $\lambda^{H \backslash G}$-almost-every $x \in L^G$. Thus $\theta$ is indeed the stabilizer type of $G \acts (L^G, \lambda^{\theta \backslash G})$. Note that if $\theta = \Stab_*(\mu)$ for a {\pmp} action $G \acts (X, \mu)$, then $\theta$ is supported on the infinite-index subgroups of $G$ if and only if the action $G \acts (X, \mu)$ is aperiodic.

It is well known that Bernoulli factors are in one-to-one correspondence with partitions possessing certain independence properties, commonly referred to as Bernoulli partitions. As we work in the less familiar setting of non-free Bernoulli shifts, we review this correspondence in detail.

\begin{defn} \label{defn:bover}
Let $G \acts (X, \mu)$ be an aperiodic {\pmp} action. Set $\theta = \Stab_*(\mu)$ and let $\mu = \int \mu_\Gamma \ d \theta(\Gamma)$ be the disintegration of $\mu$ over $\theta$. We say that a (possibly uncountable) partition $\alpha$ is \emph{$G$-Bernoulli} if: (i) $\alpha$ is independent with the stabilizer map; and (ii) for $\theta$-almost-every $\Gamma \leq G$ and every finite set $W \subseteq G$ which maps injectively into $G / \Gamma$, the partitions $w^{-1} \cdot \alpha$, $w \in W$, are mutually $\mu_\Gamma$-independent. Furthermore, if $\cF$ is a $G$-invariant sub-$\sigma$-algebra then we say $\alpha$ is \emph{$G$-Bernoulli over $\cF$} if in addition we have: (iii) $\salg_G(\alpha)$ is independent with $\cF$ relative to the stabilizer map.
\end{defn}

We remark that it is standard in the literature to say that a partition $\alpha$ as in the above definition is \emph{$G$-Bernoulli relative to $\cF$}. We have no intention of overturning this convention. However, as we frequently use the term 'relatively independent' we feel that in the present paper its best to say `$G$-Bernoulli over' in order to avoid any potential misunderstanding.

We verify the correspondence between Bernoulli factors and Bernoulli partitions.

\begin{lem} \label{lem:oneone}
Let $G \acts (X, \mu)$ be an aperiodic {\pmp} action with stabilizer type $\theta$, and let $\mu = \int_{\Gamma \leq G} \mu_\Gamma \ d \theta(\Gamma)$ be the disintegration of $\mu$ with respect to $\Stab$. Also let $G \acts (Y, \nu)$ be a factor via the map $f : X \rightarrow Y$, set $\nu_\Gamma = f_*(\mu_\Gamma)$, and set $\cF = f^{-1}(\Borel(Y))$. Fix a non-trivial (possibly uncountable) probability space $(L, \lambda)$ and let $\beta$ be the partition of $L$ into points. Then the following two classes of objects are naturally in one-to-one correspondence:
\begin{enumerate}
\item[\rm (1)] $G$-equivariant maps $q : X \rightarrow L^G$ such that $q \times f$ is a factor map from $G \acts (X, \mu)$ onto $G \acts (L^G \times Y, \int \lambda^{\Gamma \backslash G} \times \nu_\Gamma \ d \theta(\Gamma))$;
\item[\rm (2)] Borel measure-preserving maps $Q : (X, \mu) \rightarrow (L, \lambda)$ such that the partition $\alpha = Q^{-1}(\beta)$ is $G$-Bernoulli over $\cF$.
\end{enumerate}
Specifically, the correspondence is given by $q(x)(g) = Q(g^{-1} \cdot x)$ for $g \in G$, $x \in X$.
\end{lem}

\begin{proof}
(1) $\rightarrow$ (2). Fix $q : X \rightarrow L^G$ with the stated property. Setting $Q(x) = q(x)(1_G)$ for all $x \in X$, its immediate that $Q_*(\mu) = \lambda$. Let us abuse notation slightly by letting $\beta = \{B_\ell : \ell \in L\}$ also denote the partition of $L^G \times Y$ defined by $B_\ell = \{(z, y) : z(1_G) = \ell\}$, and by letting $\Borel(Y)$ also denote the $\sigma$-algebra of Borel subsets of $L^G \times Y$ measurable with respect to the second factor. Set $\kappa = \int \lambda^{\Gamma \backslash G} \times \nu_\Gamma \ d \theta(\Gamma)$, set $\phi = \Stab_*(\kappa)$, and let $\kappa = \int \kappa_\Gamma \ d \phi(\Gamma)$ be the disintegration of $\kappa$ relative to $\Stab$. Notice that $\lambda^{\Gamma \backslash G} \times \nu_\Gamma$-almost-every point in $L^G \times Y$ has stabilizer $\Gamma$, so in fact $\kappa_\Gamma = \lambda^{\Gamma \backslash G} \times \nu_\Gamma$. From this observation it now quickly follows that $\beta$ is $G$-Bernoulli over $\Borel(Y)$. Since $(q \times f)$ preserves stabilizers, meaning $\Stab(x) = \Stab((q \times f)(x))$ for $\mu$-almost-every $x \in X$, we conclude that $\alpha = (q \times f)^{-1}(\beta)$ is $G$-Bernoulli over $\cF = (q \times f)^{-1}(\Borel(Y))$.

(2) $\rightarrow$ (1). Suppose that $Q : (X, \mu) \rightarrow (L, \lambda)$ and $\alpha = Q^{-1}(\beta)$ have the stated properties, and define $q(x)(g) = Q(g^{-1} \cdot x)$ for $g \in G$, $x \in X$. Let's say $\alpha = \{A_\ell : \ell \in L\}$ where $A_\ell = Q^{-1}(\ell)$. It suffices to show that $(q \times f)_*(\mu_\Gamma) = \lambda^{\Gamma \backslash G} \times \nu_\Gamma$ for $\theta$-almost-every $\Gamma$. Property (iii) of Definition \ref{defn:bover} says that, for $\theta$-almost-every $\Gamma$, $\salg_G(\alpha)$ and $\cF$ are $\mu_\Gamma$-independent and thus $(q \times f)_*(\mu_\Gamma) = q_*(\mu_\Gamma) \times \nu_\Gamma$. Finally, properties (i) and (ii) immediately imply that $q_*(\mu_\Gamma) = \lambda^{\Gamma \backslash G}$ for $\theta$-almost-every $\Gamma$.
\end{proof}

We will also work with the following closely related notion.

\begin{defn} \label{defn:wbover}
Let $G \acts (X, \mu)$ be an aperiodic {\pmp} action, and let $\cF$ be a $G$-invariant sub-$\sigma$-algebra. Set $\theta = \Stab_*(\mu)$ and let $\mu = \int \mu_\Gamma \ d \theta(\Gamma)$ be the disintegration of $\mu$ over $\theta$. We say that a partition $\alpha$ is \emph{weakly $G$-Bernoulli over $\cF$} if (a) $\alpha$ is independent with the stabilizer map relative to $\cF$, and (b) for $\theta$-almost-every $\Gamma \leq G$ and every finite set $W \subseteq G$ which maps injectively into $G / \Gamma$, the partitions $w^{-1} \cdot \alpha$, $w \in W$, are mutually $\mu_\Gamma$-independent relative to $\cF$.
\end{defn}

We would like to explicitly observe how these two definitions simplify in the case of free actions. We leave the proof of the following lemma as an easy exercise.

\begin{lem} \label{lem:simpbern}
Let $G \acts (X, \mu)$ be a free {\pmp} action (such as an aperiodic action with $G = \Z$), let $\cF$ be a $G$-invariant sub-$\sigma$-algebra, and let $\alpha$ be a Borel partition. Then:
\begin{enumerate}
\item[\rm (1)] $\alpha$ is $G$-Bernoulli over $\cF$ if and only if the $G$-translates of $\alpha$ are mutually independent and $\salg_G(\alpha)$ is independent with $\cF$;
\item[\rm (2)] $\alpha$ is weakly $G$-Bernoulli over $\cF$ if and only if the $G$-translates of $\alpha$ are mutually independent relative to $\cF$.
\end{enumerate}
\end{lem}

The notion of weak Bernoullicity will be particularly convenient when discussing free non-ergodic actions, as the following lemma illustrates.

\begin{lem} \label{lem:nonergb}
Let $G \acts (X, \mu)$ be a free {\pmp} action, let $\alpha$ be a partition, and let $\cF$ be a $G$-invariant sub-$\sigma$-algebra. The following are equivalent:
\begin{enumerate}
\item[\rm (1)] $\alpha$ is weakly $(G, \mu)$-Bernoulli over $\sinv_G$ and $\salg_G(\alpha)$ is independent with $\cF$ relative to $\sinv_G$;
\item[\rm (2)] $\alpha$ is $(G, \nu)$-Bernoulli over $\cF$ for almost-every ergodic component $\nu$ of $\mu$.
\end{enumerate}
\end{lem}

\begin{proof}
The equivalence is based off of the simple fact that, for a $\sigma$-algebra $\Sigma$, the property of independence relative to $\Sigma$ is equivalent to the property of $\nu$-independence for almost-every fiber measure $\nu$ in the disintegration of $\mu$ relative to $\Sigma$. One needs to only recall Definitions \ref{defn:bover} and \ref{defn:wbover} and observe that the disintegration of $\mu$ relative to $\sinv_G$ coincides with the ergodic decomposition of $\mu$.
\end{proof}

Finally, we clarify the difference between Bernoullicity and weak Bernoullicity in a special case.

\begin{lem} \label{lem:wsbern}
Let $G \acts (X, \mu)$ be an aperiodic {\pmp} action, let $\cF$ be a $G$-invariant sub-$\sigma$-algebra, and let $\alpha$ be a Borel partition. If $\Stab$ is $\cF$-measurable then the following are equivalent:
\begin{enumerate}
\item[\rm (1)] $\alpha$ is independent with $\cF$ and weakly $G$-Bernoulli over $\cF$;
\item[\rm (2)] $\alpha$ is $G$-Bernoulli over $\cF$.
\end{enumerate}
\end{lem}

\begin{proof}
We will refer to properties (i), (ii), and (iii) of Definition \ref{defn:bover} and properties (a) and (b) of Definition \ref{defn:wbover}.

(1) $\Rightarrow$ (2). Since $\Stab$ is $\cF$-measurable and $\alpha$ is independent with $\cF$, we get that $\alpha$ is independent with $\Stab$, establishing (i). An equivalent formulation of property (b) is that for $\theta$-almost-every $\Gamma \leq G$, every finite $W \subseteq G$, and every $g \in G$ with $g \Gamma \not\subseteq W \Gamma$, we have that $g^{-1} \cdot \alpha$ and $\alpha^W$ are $\mu_\Gamma$-independent relative to $\cF$. By $G$-invariance, $g^{-1} \cdot \alpha$ is independent with $\cF$. Furthermore, since $\Stab$ is $\cF$-measurable, it follows that $g^{-1} \cdot \alpha$ is $\mu_\Gamma$-independent with $\cF$ for $\theta$-almost-every $\Gamma \leq G$. So it follows from Lemma \ref{lem:ind} that $g^{-1} \cdot \alpha$ and $\alpha^W$ are $\mu_\Gamma$-independent. This establishes property (ii). So for $\theta$-almost-every $\Gamma \leq G$ and every finite $W \subseteq G$ mapping injectively into $G / \Gamma$, the partitions $w^{-1} \cdot \alpha$, $w \in W$, are $\mu_\Gamma$-independent with $\cF$ and mutually $\mu_\Gamma$-independent relative to $\cF$. Hence $\salg_G(\alpha)$ is $\mu_\Gamma$-independent with $\cF$ for $\theta$-almost-every $\Gamma \leq G$. Equivalently, $\salg_G(\alpha)$ and $\cF$ are independent relative to $\Stab$, proving (iii).

(2) $\Rightarrow$ (1). Property (iii) implies $\alpha$ is independent with $\cF$ relative to $\Stab$. Combined with property (i) and Lemma \ref{lem:ind}, it follows that $\alpha$ is independent with $\cF$. Since $\Stab$ is $\cF$-measurable, property (a) is trivially true. Properties (ii) and (iii) imply that for $\theta$-almost-every $\Gamma \leq G$ and every finite $W \subseteq G$ mapping injectively into $G / \Gamma$, the partitions $w^{-1} \cdot \alpha$, $w \in W$, are mutually $\mu_\Gamma$-independent and $\alpha^W$ is $\mu_\Gamma$-independent with $\cF$. It follows immediately that the partitions $w^{-1} \cdot \alpha$, $w \in W$, are mutually $\mu_\Gamma$-independent relative to $\cF$, establishing (b).
\end{proof}

\section{Transformations in the full-group}

For a {\pmp} action $G \acts (X, \mu)$, we let $E_G^X$ denote the \emph{induced orbit equivalence relation}:
$$E_G^X = \{(x, y) : \exists g \in G \ g \cdot x = y\}.$$
The \emph{full-group} of $E_G^X$, denoted $[E_G^X]$, is the set of all Borel bijections $T : X \rightarrow X$ satisfying $T(x) \ E_G^X \ x$ for all $x \in X$. In order to prove our main theorem we will apply classical results, such as {\sinai}'s original factor theorem, to the $\Z$-actions induced by aperiodic elements $T \in [E_G^X]$. Its a simple exercise to check that every $T \in [E_G^X]$ preserves $\mu$.

The group $[E_G^X]$ is quite large, but the following definition introduces an important constraint that is useful in entropy theory.

\begin{defn} \label{defn:express}
Let $G \acts (X, \mu)$ be a {\pmp} action, let $T \in [E_G^X]$, and let $\cF$ be a $G$-invariant sub-$\sigma$-algebra. We say that $T$ is \emph{$\cF$-expressible} if there is a $\cF$-measurable partition $\{Q_g \: g \in G\}$ of $X$ such that $T(x) = g \cdot x$ for every $x \in Q_g$ and all $g \in G$.
\end{defn}

We recall two elementary lemmas from \cite{S2}.

\begin{lem} \label{lem:expmove}
Let $G \acts (X, \mu)$ be a {\pmp} action, let $\cF$ be a $G$-invariant sub-$\sigma$-algebra, and let $T \in [E_G^X]$ be $\cF$-expressible. Then $T(A) \in \salg_G(A) \vee \cF$ for every set $A \in \Borel(X)$. In particular, every $G$-invariant $\sigma$-algebra containing $\cF$ must be $T$-invariant.
\end{lem}

\begin{lem} \label{lem:expgroup}
Let $G \acts (X, \mu)$ be a {\pmp} action and let $\cF$ be a $G$-invariant sub-$\sigma$-algebra. Then the set of $\cF$-expressible elements of $[E_G^X]$ is a group under the operation of composition.
\end{lem}

In the more familiar language of cocycles, the condition that $T$ be $\cF$-expressible is equivalent to the existence of a cocycle $c : \Z \times X \rightarrow G$ that connects the action of $T$ to the action of $G$ and is $\cF$-measurable in the second coordinate. Since we work with non-free actions, the partition described in Definition \ref{defn:express} (and the cocycle) might not be unique. However the following fact will be useful to us.

\begin{lem} \label{lem:eqtest}
Let $G \acts (X, \mu)$ be a {\pmp} action, let $\cF$ be a $G$-invariant class-bijective sub-$\sigma$-algebra, and let $T \in [E_G^X]$ be $\cF$-expressible. Then for each $h \in G$ the set $\{x \in X : T(x) = h \cdot x\}$ is $\cF$-measurable.
\end{lem}

\begin{proof}
Fix a $\cF$-measurable partition $\{Q_g : g \in G\}$ as in Definition \ref{defn:express}. Our assumption that $\cF$ is class-bijective implies that the map $\Stab$ is $\cF$-measurable. Thus
\begin{equation*}
\{x \in X : T(x) = h \cdot x\} = \bigcup_{g \in G} \Big( Q_g \cap \{y \in X : g^{-1} h \in \Stab(y)\} \Big) \in \cF.\qedhere
\end{equation*}
\end{proof}

We recall the following simple lemma that highlights the connection between $\cF$-expressibility and entropy.

\begin{lem} \cite[Lem. 8.5]{AS} \label{lem:expent}
Let $G \acts (X, \mu)$ be a {\pmp} ergodic action, let $\cF$ be a $G$-invariant sub-$\sigma$-algebra, and let $T \in [E_G^X]$ be aperiodic and $\cF$-expressible. Then $\rh_G(\alpha \given \cF) \leq \ksh_T(\alpha \given \cF)$ for every partition $\alpha$.
\end{lem}

The next lemma illustrates how Bernoullicity can be transferred between the action of $G$ and the action of a $\cF$-expressible transformation. Ultimately, when we prove our main theorem we will need to use two elements $S, T \in [E_G^X]$. For ease of later reference, we will denote the transformation by $S$ in the following lemma.

\begin{lem} \label{lem:useS}
Let $G \acts (X, \mu)$ be an aperiodic {\pmp} action, let $\cF$ be a $G$-invariant class-bijective sub-$\sigma$-algebra, and let $S \in [E_G^X]$ be aperiodic and $\cF$-expressible. Assume $\beta$ is a partition which is $G$-weakly Bernoulli over $\cF$. Then
\begin{enumerate}
\item [\rm (i)] $\beta$ is $S$-weakly Bernoulli over $\cF$;
\item [\rm (ii)] if $\mu \res \cF$ is $S$-ergodic then $\mu \res \salg_S(\beta) \vee \cF$ is $S$-ergodic as well;
\item [\rm (iii)] if $\beta$ is furthermore $G$-Bernoulli over $\cF$ then it is $S$-Bernoulli over $\cF$;
\item [\rm (iv)] if $\alpha \subseteq \salg_S(\beta) \vee \cF$ is any partition which is $S$-Bernoulli over $\cF$, then $\alpha$ is $G$-Bernoulli over $\cF$.
\end{enumerate}
\end{lem}

\begin{proof}
(i). Since $S$ acts freely on $X$, its enough to show that for every finite $K \subseteq \Z$ the $S^k$-translates of $\beta$, $k \in K$, are mutually independent relative to $\cF$. Let $G \acts (Y, \nu)$ be the factor of $(X, \mu)$ associated to $\cF$, and let $\mu = \int \mu_y \ d \nu(y)$ be the disintegration of $\mu$ over $\nu$. Since $\cF$ is class-bijective, the stabilizer map $\Stab : X \rightarrow \Sub(G)$ is constant $\mu_y$-almost-everywhere for $\nu$-almost-every $y$. Also observe that $S$ descends to a transformation of $Y$ since $S$ is $\cF$-expressible. Let $Y_0$ be the set of $y \in Y$ which have an infinite $S$-orbit and which have the property that for every finite $W \subseteq G$ that maps injectively to $W \cdot y$, the $W^{-1}$-translates of $\beta$ are mutually $\mu_y$-independent. Our assumptions imply that $\nu(Y_0) = 1$. Now fix $y \in Y_0$ and a finite set $K \subseteq \Z$. By Lemma \ref{lem:expgroup}, each $S^k$ is $\cF$-expressible. So there are $w(k) \in G$, $k \in K$, with $w(k) \cdot x = S^k(x)$ for $\mu_y$-almost-every $x \in X$. This implies that $w(k)^{-1} \cdot \beta = S^{-k}(\beta)$ $\mu_y$-almost-everywhere. Additionally, since the $S$-orbit of $y$ is infinite, the map $\{w(k) : k \in K\} \rightarrow \{w(k) \cdot y : k \in K\}$ is injective. So the partitions $w(k)^{-1} \cdot \beta = S^{-k}(\beta)$, $k \in K$, are mutually $\mu_y$-independent. We conclude $\beta$ is $S$-weakly Bernoulli over $\cF$.

(ii). Fix an $S$-invariant set $A \subseteq \salg_S(\beta) \vee \cF$. Since $\mu$ is $S$-invariant we have that $S_*(\mu_y) = \mu_{S(y)}$ for $\nu$-almost-every $y \in Y$. Consequently, $\mu_y(A) = \mu_{S(y)}(S(A)) = \mu_{S(y)}(A)$ so the map $y \mapsto \mu_y(A)$ is $S$-invariant. Our assumption that $\mu \res \cF$ is $S$-ergodic implies that $\nu$ is $S$-ergodic and therefore there is a constant $a$ with $\mu_y(A) = a$ for $\nu$-almost-every $y$. Then $\mu(A) = a$ so it suffices to show $a \in \{0, 1\}$. So fix $\epsilon > 0$, pick $m \in \N$, and let $Y'$ be the set of $y \in Y$ such that $\mu_y(A) = a$ and such that there exists a set $D$ that is a union of sets from $\bigvee_{i=-m}^m S^i(\beta)$ satisfying $\mu_y(D \symd A) < \epsilon$. By picking $m$ large enough we can assume that $\nu(Y') > 0$. In particular, $Y'$ must have an infinite intersection with some $S$-orbit, so we can find $y, y' \in Y'$ and $k > 2m$ with $S^k(y) = y'$. Let $D, D'$ be unions of sets from $\bigvee_{i=-m}^m S^i(\beta)$ with $\mu_y(D \symd A), \mu_{y'}(D' \symd A) < \epsilon$. Notice that $\mu_y(S^{-k}(D') \symd S^{-k}(A)) = \mu_{S^k(y)}(D' \symd A) = \mu_{y'}(D' \symd A) < \epsilon$. By (i), $\beta$ is $S$-weakly Bernoulli over $\cF$ and therefore $\mu_y(D \cap S^{-k}(D')) = \mu_y(D) \mu_y(S^{-k}(D'))$. So
\begin{align*}
a = \mu_y(A) = \mu_y(A \cap S^{-k}(A)) & < \mu_y(D \cap S^{-k}(D')) + 2\epsilon\\
 & = \mu_y(D) \mu_y(S^{-k}(D')) + 2 \epsilon < a^2 + 2 \epsilon(a + 1) + 4 \epsilon^2.
\end{align*}
Letting $\epsilon$ tend to $0$, we obtain $a \leq a^2$ and hence $a \in \{0, 1\}$.

(iii). Note that the $G$-stabilizer map is $\cF$-measurable and that the $S$-stabilizer map is essentially trivial. Thus the equivalence in Lemma \ref{lem:wsbern} holds for both the actions of $G$ and $S$. In particular, our assumption implies $\beta$ is independent with $\cF$, and (i) tells us that $\beta$ is $S$-weakly Bernoulli over $\cF$. Therefore $\beta$ is $S$-Bernoulli over $\cF$.

(iv). Now suppose that $\alpha \subseteq \salg_S(\beta) \vee \cF$ is $S$-Bernoulli over $\cF$. Since $S$ acts freely, Lemma \ref{lem:wsbern} implies that $\alpha$ is independent with $\cF$. Additionally, since $\cF$ and $\mu$ are $G$-invariant, we have that for fixed $g \in G$ the partitions $g \cdot S^k(\alpha)$, $k \in \Z$, are mutually independent and $g \cdot \salg_S(\alpha)$ is independent with $\cF$. Let $Y_0$ be the set of $y \in Y$ such that: (a) for every $g \in G$ the partitions $g \cdot S^k(\alpha)$, $k \in \Z$, are mutually $\mu_y$-independent, (b) the partitions $w^{-1} \cdot \beta$, $w \in W$, are mutually $\mu_y$-independent whenever $W$ is a finite subset of $G$ that maps injectively to $W \cdot y$, and (c) such that $g \cdot \alpha \subseteq g \cdot \salg_S(\beta)$ modulo $\mu_y$-null sets for every $g \in G$. Note $\nu(Y_0) = 1$. Fix $y \in Y_0$ and fix a finite $W \subseteq G$ that maps injectively to $W \cdot y$. Let $\{H_i : i \in \N\}$ be the partition of $G$ where $g, g' \in G$ lie in the same piece of this partition if and only if there is $k \in \Z$ with $S^k(g \cdot y) = g' \cdot y$. Set $V_i = H_i \cap W$, fix $v_i \in V_i$ for each $i$, and set $K_i = \{k \in \Z : S^k(v_i \cdot y) \in V_i \cdot y\}$. Note that $|K_i| = |V_i|$. For each $i$, modulo $\mu_y$-null sets we have
$$\{w^{-1} \cdot \alpha : w \in V_i\} = \{v_i^{-1} \cdot S^{-k}(\alpha) : k \in K_i\}.$$
Since $y \in Y_0$, for each fixed $i$ the partitions $w^{-1} \cdot \alpha$, $w \in V_i$, are mutually $\mu_y$-independent. Furthermore, we have (modulo $\mu_y$-null sets)
$$\bigvee_{w \in V_i} w^{-1} \cdot \alpha \subseteq \bigvee_{w \in V_i} w^{-1} \cdot \salg_S(\beta) = \bigvee_{g \in H_i} g^{-1} \cdot \beta.$$
As $y \in Y_0$ it follows that the $\sigma$-algebras $\bigvee_{g \in H_i} g^{-1} \cdot \beta$, $i \in \N$, are mutually $\mu_y$-independent. Therefore the $\sigma$-algebras $\bigvee_{w \in V_i} w^{-1} \cdot \alpha$, $i \in \N$, are mutually $\mu_y$-independent. We conclude that the $W^{-1}$-translates of $\alpha$ are mutually $\mu_y$-independent. Since the $G$-stabilizer map is $\cF$-measurable, its also trivially true that $\alpha$ is independent with this map relative to $\cF$. So $\alpha$ is $G$-weakly Bernoulli over $\cF$. We previously noted that $\alpha$ is independent with $\cF$, so Lemma \ref{lem:wsbern} implies that $\alpha$ is $G$-Bernoulli over $\cF$.
\end{proof}

Finally, we end this section with a technical lemma about the existence of automorphisms $S \in [E_G^X]$ having nice ergodicity properties.

\begin{lem} \label{lem:grabs}
Let $X$ be a standard Borel space, let $G \acts X$ be an aperiodic Borel action, and let $\cF$ be a countably generated class-bijective $G$-invariant sub-$\sigma$-algebra. Then there exists a $\cF$-expressible $S \in [E_G^X]$ so that $\nu \res \cF$ is $S$-ergodic for every $\nu \in \E_G(X)$.
\end{lem}

Before proving this we need a few lemmas.

\begin{lem} \label{lem:unisize}
Let $X$ be a standard Borel space, let $G \acts X$ be an aperiodic Borel action, and let $A \subseteq X$ be Borel. If $\nu(A) \geq r$ for every $\nu \in \E_G(X)$ then there is a Borel set $B \subseteq A$ with $\nu(B) = r$ for every $\nu \in \E_G(X)$.
\end{lem}

\begin{proof}
Without loss of generality we may assume $\E_G(X) \neq \varnothing$ and $r > 0$ as otherwise there is nothing to show. Since $G$ acts aperiodically, each $\nu \in \E_G(X)$ must be non-atomic. In particular, $X$ is uncountable. Since all uncountable standard Borel spaces are Borel isomorphic \cite[Thm. 15.6]{K95}, we can assume $X = [0, 1]$. By the ergodic decomposition theorem \cite{Far62,Var63}, there is a $G$-invariant Borel map $e : X \rightarrow \E_G(X)$ such that $\nu(e^{-1}(\nu)) = 1$ for every $\nu \in \E_G(X)$. Now define
$$B = \{x \in A : e(x)([0,x] \cap A) \leq r\}$$
Equivalently, $x \in A$ belongs to $B$ if and only if for every rational $q < x$ we have $e(x)([0,q] \cap A) \leq r$. Thus $B$ is a Borel set and clearly for every $\nu \in \E_G(X)$ we have $\nu(B) = r$ since $\nu$ is non-atomic.
\end{proof}

\begin{lem} \label{lem:nepush}
Let $X$ be a standard Borel space, let $G \acts X$ be a Borel action, and let $A, B \subseteq X$ with $\nu(A) = \nu(B)$ for all $\nu \in \E_G(X)$. Then there exists Borel sets $A' \subseteq A$, $B' \subseteq B$, and a Borel bijection $\phi : A' \rightarrow B'$ such that $\phi(x) \in G \cdot x$ for all $x \in A'$ and $\nu(A \setminus A') = \nu(B \setminus B') = 0$ for all $\nu \in \E_G(X)$.
\end{lem}

\begin{proof}
Fix an enumeration $g_0, g_1, \ldots$ of $G$. Set $A_0 = A \cap g_0^{-1} \cdot B$, $B_0 = g_0 \cdot A_0$, and inductively define
$$A_n = \left( A \setminus \bigcup_{k<n} A_k \right) \cap g_n^{-1} \cdot \left( B \setminus \bigcup_{k<n} B_k \right), \qquad B_n = g_n \cdot A_n.$$
Set $A' = \cup_n A_n$, $B' = \cup_n B_n$, and define $\phi : A' \rightarrow B'$ by setting $\phi(a) = g_n \cdot a$ for $a \in A_n$. Now fix $\nu \in \E_G(X)$. Since $\nu$ is $G$-invariant we have $\nu(A') = B'$ so $\nu(A \setminus A') = \nu(B \setminus B')$. It is immediate from the definitions that $G \cdot (A \setminus A')$ is disjoint from $G \cdot (B \setminus B')$. If $\nu(A \setminus A') = \nu(B \setminus B') > 0$ then ergodicity of $\nu$ would imply that these disjoint sets would each have full measure, which is impossible.
\end{proof}

The following lemma is a non-ergodic analogue of the well-known fact that every ergodic {\pmp} countable Borel equivalence relation contains an ergodic hyperfinite sub-equivalence relation \cite[Thm. 3.5]{K10}. Recall that a Borel equivalence relation $R$ on $X$ is \emph{hyperfinite} if there is an increasing sequence of Borel equivalence relations $(R_n)_{n \in \N}$ such that $R = \bigcup_n R_n$ and each $R_n$ is finite (i.e. each class of $R_n$ is finite). Equivalently, $R$ is hyperfinite if and only if there is a Borel action $\Z \acts X$ with $R = E_\Z^X$ \cite[Prop. 1.2]{JKL}.

\begin{lem} \label{lem:pregrabs}
Let $X$ be a standard Borel space and let $G \acts X$ be an aperiodic Borel action. Then there exists $S \in [E_G^X]$ such that $\E_G(X) \subseteq \E_S(X)$.
\end{lem}

\begin{proof}
We provide a non-ergodic version of the argument in \cite[Prop. 9.3.2]{Zi}. Since $X$ is standard Borel, there is a countable algebra $\mathcal{C} \subseteq \Borel(X)$ with $\salg(\mathcal{C}) = \Borel(X)$. Fix a sequence $(C_n)_{n \in \N}$ of elements of $\mathcal{C}$ with $C_0 = \varnothing$ and with the property that every $C \in \mathcal{C}$ appears in the sequence infinitely many times. We will inductively build a sequence of finite Borel partitions $\alpha_n$ and a sequence of finite groups $\Gamma_n \leq [E_G^X]$ satisfying the following conditions for every $n \in \N$:
\begin{enumerate}
\item[\rm (i)] $\alpha_{n+1}$ is finer than $\alpha_n$ and $\Gamma_n \subseteq \Gamma_{n+1}$;
\item[\rm (ii)] for every $A \in \alpha_n$ and $\nu \in \E_G(X)$ the map $\gamma \in \Gamma_n \mapsto \gamma(A)$ is a bijection from $\Gamma_n$ to $\alpha_n$ up to $\nu$-null sets;
\item[\rm (iii)] for every $\nu \in \E_G(X)$ there is a set $D$ that is a union of sets from $\alpha_n$ with $\nu(C_n \symd D) < 1/n$.
\end{enumerate}
Note that (ii) implies that $\nu(A) = |\alpha_n|^{-1}$ for every $\nu \in \E_G(X)$ and $A \in \alpha_n$. To begin the construction, let $\alpha_0 = \{X\}$ be the trivial partition and let $\Gamma_0 = \{\id_X\}$ consist of the identity map on $X$.

Inductively assume that $\alpha_{n-1}$ and $\Gamma_{n-1}$ have been constructed satisfying condition (ii). Fix any set $A \in \alpha_{n-1}$. Let $\beta$ be the partition of $A$ induced by the sets $\gamma^{-1}(C_n) \cap A$, $\gamma \in \Gamma_{n-1}$. Set $m = n |\beta| + 1$. The map $\nu \in \E_G(X) \mapsto (\lfloor m |\alpha_{n-1}| \nu(B) \rfloor)_{B \in \beta}$ is Borel and takes finitely many values (here $\lfloor r \rfloor$ denotes the greatest integer less than or equal to $r \in \R$). By the ergodic decomposition theorem \cite{Far62,Var63}, there is a $G$-invariant Borel map $e : X \rightarrow \E_G(X)$ such that $\nu(e^{-1}(\nu)) = 1$ for every $\nu \in \E_G(X)$. Consequently, by breaking $X$ into a finite number of $G$-invariant Borel sets and constructing a $m |\alpha_{n-1}|$-piece partition $\alpha_n$ on each piece separately, we may assume that there are fixed integers $(k_B)_{B \in \beta}$ with $\lfloor m |\alpha_{n-1}| \nu(B) \rfloor = k_B$ for every $\nu \in \E_G(X)$ and $B \in \beta$. By repeatedly applying Lemma \ref{lem:unisize}, we can obtain, for each $B \in \beta$, disjoint Borel sets $B_1, B_2, \ldots, B_{k_B}$ contained in $B$ with $\nu(B_i) = 1/(m |\alpha_{n-1}|)$ for every $\nu \in \E_G(X)$.

Using Lemma \ref{lem:unisize}, extend the family $\{B_i : B \in \beta, \ 1 \leq i \leq k_B\}$ to a partition $\xi$ of $A$ with $\nu(D) = 1/(m |\alpha_{n-1}|)$ for all $D \in \xi$, $\nu \in \E_G(X)$ and with $|\xi| = m$. By Lemma \ref{lem:nepush}, there is a Borel bijection $\theta : A \rightarrow A$ so that, for every $\nu \in \E_G(X)$, $\theta$ cyclically permutes the members of $\xi$ up to $\nu$-null sets and satisfies $\theta^{|\xi|} = \id_A$ everywhere. Let $A'$ be the set of $x \in A$ for which the set $\{\gamma(x) : \gamma \in \Gamma_{n-1}\}$ meets every class of $\alpha_{n-1}$, and set $A'' = \bigcap_{i=0}^{|\xi|-1} \theta^i(A')$. Note that $\nu(A \setminus A'') = 0$ for all $\nu \in \E_G(X)$. Now define $\phi \in [E_G^X]$ by setting $\phi(x) = \gamma \circ \theta \circ \gamma^{-1}(x)$ if $\gamma \in \Gamma_{n-1}$ and $x \in \gamma(A'')$ and set $\phi(x) = x$ otherwise. Notice that $\phi$ commutes with every $\gamma \in \Gamma_{n-1}$ and $\phi^{|\xi|} = \id_X$ everywhere.

Now let $\alpha_n$ be any Borel partition of cardinality $|\xi| \cdot |\alpha_{n-1}|$ that is finer than $\alpha_{n-1}$ and satisfies
$$\alpha_n \res \bigcup_{\gamma \in \Gamma_{n-1}} \gamma(A'') = \{\gamma(D \cap A'') : D \in \xi, \ \gamma \in \Gamma_{n-1}\}.$$
Also let $\Gamma_n$ be the group generated by $\Gamma_{n-1}$ and $\phi$. It is immediate that clauses (i) and (ii) are satisfied. Now consider (iii) and fix $\nu \in \E_G(X)$. By construction, for any $\gamma \in \Gamma_{n-1}$ the set $\gamma^{-1}(C_n) \cap A$ is a union of sets from $\beta$. Also by construction, for each $B \in \beta$ there is a set $D$ that is a union of sets from $\xi$ with $D \subseteq B$ and $\nu(B \setminus D) < 1/(m |\alpha_{n-1}|)$. Consequently, there is a set $D_\gamma$ that is a union of sets from $\xi$ with $D_\gamma \subseteq \gamma^{-1}(C_n) \cap A$ and $\nu((\gamma^{-1}(C_n) \cap A) \setminus D_\gamma) < |\beta|/(m |\alpha_{n-1}|)$. It follows that $D = \bigcup_{\gamma \in \Gamma_{n-1}} \gamma(D_\gamma)$ is a union of sets from $\alpha_n$ and $\nu(C_n \symd D) < |\beta| / m < 1/n$.

Set $\Gamma = \bigcup_n \Gamma_n$. Then $E_\Gamma^X = \bigcup_n E_{\Gamma_n}^X$ so $E_\Gamma^X$ is hyperfinite and there is a Borel automorphism $S : X \rightarrow X$ with $E_S^X = E_\Gamma^X$ \cite[Prop. 1.2]{JKL}. Since $E_\Gamma^X \subseteq E_G^X$ we necessarily have $S \in [E_G^X]$. Now fix $\nu \in \E_G(X)$ and consider a $S$-invariant Borel set $Y \subseteq X$. Then $Y$ is $\Gamma$-invariant. So for any $n$ and $A \in \alpha_n$ we have $\nu(Y \cap A) = \nu(\gamma(Y \cap A)) = \nu(Y \cap \gamma(A))$ for each $\gamma \in \Gamma_n$, hence clause (ii) implies that $\nu(Y \cap A) = \nu(Y) / |\alpha_n| = \nu(Y) \nu(A)$. As each $C \in \mathcal{C}$ equals $C_n$ for infinitely many $n$, it follows from (iii) that $\nu(Y \cap C) = \nu(Y) \nu(C)$ for all $C \in \mathcal{C}$. In fact, this holds for all $C \in \Borel(X) = \salg(\mathcal{C})$. Plugging in $C = Y$ gives $\nu(Y) \in \{0, 1\}$. Thus $\nu \in \E_S(X)$.
\end{proof}

\begin{lem} \label{lem:borelfactor}
Let $X$ and $Z$ be standard Borel spaces, $G \acts X$ an aperiodic Borel action, $\cF$ a countably generated class-bijective $G$-invariant sub-$\sigma$-algebra, and $f : X \rightarrow Z$ a $\cF$-measurable function. Then there is a standard Borel space $Y$, an aperiodic Borel action $G \acts Y$, a Borel function $\hat{f} : Y \rightarrow Z$, and a class-bijective $G$-equivariant Borel map $\phi : X \rightarrow Y$ with $\cF = \phi^{-1}(\Borel(Y))$ and $f = \hat{f} \circ \phi$.
\end{lem}

\begin{proof}
Fix a countable algebra $\cA = \{A_n : n \in \N\}$ with $\salg_G(\cA) = \cF$. Write $2 = \{0, 1\}$ and define $\phi : X \rightarrow (2^\N)^G$ by $\phi(x)(g)(n) = \chi_{A_n}(g^{-1} \cdot x)$, where $\chi_{A_n}$ is the characteristic function of $A_n$. This map is $G$-equivariant since $\phi(t \cdot x)(g)(n) = \chi_{A_n}(g^{-1} t \cdot x) = \phi(x)(t^{-1} g)(n) = (t \cdot \phi(x))(g)(n)$. Let $Y$ be the set of points in $(2^\N)^G$ having an infinite $G$-orbit. Then $Y$ is Borel so $(Y, \Borel(Y))$ is a standard Borel space where $\Borel(Y) = \Borel((2^\N)^G) \res Y$ \cite[Cor. 13.4]{K95}. Also, $\phi$ is class-bijective since $\cF$ is and hence $\phi(X) \subseteq Y$. It is clear that $\phi^{-1}(C_n) = A_n$ where $C_n = \{y \in Y : y(1_G)(n) = 1\}$. Since $\salg_G(\{C_n : n \in \N\}) = \Borel(Y)$ and $\salg_G(\cA) = \cF$ it follows that $\phi^{-1}(\Borel(Y)) = \cF$. Finally, since every standard Borel space is Borel isomorphic to a Borel subset of $2^\N$ and since each $z \in 2^\N$ is described by its coordinates $z(n) \in \{0, 1\}$, it suffices to consider the case where $f : X \rightarrow \{0, 1\}$. As $f^{-1}(1) \in \cF = \phi^{-1}(\Borel(Y))$, there is $B \in \Borel(Y)$ with $f^{-1}(1) = \phi^{-1}(B)$. Now set $\hat{f} = \chi_B : Y \rightarrow \{0, 1\}$.
\end{proof}

\begin{proof}[Proof of Lemma \ref{lem:grabs}]
By Lemma \ref{lem:borelfactor} there is an aperiodic Borel action $G \acts Y$ on a standard Borel space $Y$ and a class-bijective $G$-equivariant Borel map $\phi : X \rightarrow Y$ with $\cF = \phi^{-1}(\Borel(Y))$. Apply Lemma \ref{lem:pregrabs} to get $S' \in [E_G^Y]$ with $\E_G(Y) \subseteq \E_{S'}(Y)$. Fix an enumeration $g_0, g_1, \ldots$ of $G$ and for each $y \in Y$ let $c(y) = g_i$ where $i$ is least with $g_i \cdot y = S'(y)$. Now lift $S'$ to a $\cF$-expressible $S \in [E_G^X]$ by setting $S(x) = c(\phi(x)) \cdot x$. Then for any $\nu \in \E_G(X)$ we have $\phi_*(\nu) \in \E_G(Y) \subseteq \E_{S'}(Y)$. So $\phi_*(\nu)$ is $S'$-ergodic and hence $\nu \res \cF$ is $S$-ergodic.
\end{proof}

\section{Non-ergodic actions of $\Z$} \label{sec:nonerg}

In this section we retrace and elaborate upon the work of Kieffer and Rahe which, in the case of $\Z$-actions, revealed the non-ergodic factor theorem to be a direct consequence of the perturbative factor theorem \cite{KiRa}. The purpose for including this section is that we will later require a factor theorem for $\Z$-actions which is simultaneously relative, non-ergodic, and perturbative, and no such result is formally stated in the literature (though it is known to experts). Let us first recall the relative perturbative factor theorem for $\Z$.

\begin{lem}[essentially Thouvenot \cite{Th75} (see also \cite{Or70b})] \label{lem:zfolk}
Let $\Z \acts (X, \mu)$ be an aperiodic {\pmp} ergodic action, let $\cF$ be a $\Z$-invariant sub-$\sigma$-algebra, and let $\pv$ be a probability vector. If $\sH(\pv) \leq \ksh_\Z(X, \mu \given \cF)$ then there is a partition $\alpha$ with $\dist(\alpha) = \pv$ which is $\Z$-Bernoulli over $\cF$.

Moreover, there is a function $\psd_\Z: \R_+^2 \rightarrow \R_+$ such that, for any $\Z \acts (X, \mu)$, $\cF$, and $\pv$ satisfying the assumptions above, if $\sH(\pv) \leq M$ and $\xi$ is a partition satisfying
$$|\dist(\xi) - \pv| + |\sH(\xi) - \sH(\pv)| + |\ksh_\Z(\xi \given \cF) - \sH(\pv)| < \psd_\Z(M, \epsilon)$$
then the partition $\alpha$ above may be chosen so that $\dB_\mu(\alpha, \xi) < \epsilon$.
\end{lem}

In order to explain the proof of the above lemma, we will need the following fact.

\begin{lem} \cite[Fact 1.1.11]{Do11} \label{lem:drop}
For every $M, \epsilon > 0$ there exists $\eta = \eta(M, \epsilon) > 0$ with the following property. Let $(X, \mu)$ be a standard probability space, let $\cF$ be a sub-$\sigma$-algebra, let $(Y, \nu)$ be the factor of $(X, \mu)$ associated to $\cF$, and let $\mu = \int \mu_y \ d \nu(y)$ be the disintegration of $\mu$ over $\nu$. If $\gamma$ is a partition of $X$ with $\sH(\gamma) \leq M$ and $\sH(\gamma \given \cF) > \sH(\gamma) - \eta(M, \epsilon)$, then
$$\nu(\{y \in Y : |\dist_{\mu_y}(\gamma) - \dist_\mu(\gamma)| < \epsilon\}) > 1 - \epsilon.$$
\end{lem}

\begin{proof}[Proof sketch of Lemma \ref{lem:zfolk}]
The first statement is well-known and is implied by the work of Thouvenot \cite{Th75}. The second statement is stated a tad differently than it has previously appeared in the literature but is certainly known to experts. We briefly describe how to obtain the precise statement above, but we leave the details to the reader.

Working with $k$-piece partitions and length-$k$ probability vectors, Thouvenot proved the above statement with two modifications: (i) in place of $\psd_\Z$ he used a quantity $\psd_0(k, \epsilon)$ depending only on $k$ and $\epsilon$; and (ii) he omitted the $|\sH(\xi) - \sH(\pv)|$ term from the inequality. Inspection of his proof reveals that, in the special case where $\ksh_\Z(\xi \given \cF) < \sH(\pv)$ \cite[Prop. 2]{Th75}, the quantity $\psd_0(k, \epsilon)$ ultimately depends upon the uniform continuity of Shannon entropy on the space of length-$k$ probability vectors, and on a weakened version of Lemma \ref{lem:drop}. By insisting that $|\sH(\xi) - \sH(\pv)|$ be small, and by using the stronger Lemma \ref{lem:drop}, we can instead choose a $\psd_1(M, \epsilon)$ depending only upon $M$ and $\epsilon$. Next, in the case $\ksh_\Z(\xi \given \cF) \geq \sH(\pv)$, one takes a continuous path of partitions $(\xi_t)_{t \in [0, 1]}$ with $\xi_0 = \xi$ and $\sH(\xi_1) < \sH(\pv)$. Then there will be $s \in [0, 1]$ with $\ksh_\Z(\xi_s \given \cF) < \sH(\pv)$, allowing the prior argument to be applied. Along this path, the Shannon entropy does not need to drop by much, so Lemma \ref{lem:path} from the next section implies that there is a path where the distribution of $\xi_s$ does not change too much either, in fact the change in distribution is bounded by a function of the change in Shannon entropy. So our statement above holds for $k$-piece partitions and length-$k$ probability vectors simultaneously for every $k$. Finally, given a countable partition $\xi$ we can find finite partitions $\xi'$ coarser than $\xi$ which approximate $\dist(\xi)$, $\sH(\xi)$, and $\ksh_\Z(\xi \given \cF)$ arbitrarily well. Similarly, if $|\pv| = \infty$ we can choose a sequence of finite probability vectors $\pv^n$ satisfying $|\pv^n - \pv| + 2 |\sH(\pv^n) - \sH(\pv)| < (1/2) \psd_1(M, 2^{-n-1} \epsilon )$. One can then perform a standard construction of a Cauchy sequence of partitions $\alpha^n$ which are $\Z$-Bernoulli over $\cF$ and have distribution $\pv^n$, and the limiting partition $\alpha$ will have the desired properties. It suffices to set $\psd_\Z(M, \epsilon) = \psd_1(M, \epsilon / 2)$ for this construction.
\end{proof}

For a standard Borel space $X$ we write $\Prt$ for the set of countable ordered Borel partitions $\xi = \{C_i : 0 \leq i < |\xi|\}$ of $X$. If $\mu$ is a Borel probability measure on $X$, then we write $\Prt_\sH(\mu)$ for the set of $\xi \in \Prt$ with $\sH_\mu(\xi) < \infty$. For $\cQ \subseteq \Prt$ and $\gamma \in \Prt$ we set
$$\dB_\mu(\gamma, \cQ) = \inf_{\xi \in \cQ} \dB_\mu(\gamma, \xi).$$
Recall that for a Borel probability measure $\mu$, the Rokhlin distance between $\xi, \gamma \in \Prt$ is
$$\dR_\mu(\xi, \gamma) = \sH_\mu(\xi \given \gamma) + \sH_\mu(\gamma \given \xi)$$
(this may be infinite). The Rokhlin distance doesn't take the order of the partitions into account, so we adjust it by setting
$$\dC_\mu(\xi, \gamma) = \dR_\mu(\xi, \gamma) + \dB_\mu(\xi, \gamma).$$
Observe that the map $\xi \in \Prt \mapsto \sH_\mu(\xi)$ is $\dC_\mu$-continuous.

For an ergodic aperiodic {\pmp} action $G \acts (X, \mu)$, a $G$-invariant sub-$\sigma$-algebra $\cF$,  and a probability vector $\pv$, set
$$\Bp{\mu}{\pv}{\cF} = \Big\{ \alpha \in \Prt : \alpha \text{ is } (G, \mu)\text{-Bernoulli over } \cF \text{ and } \dist_\mu(\alpha) = \pv \Big\}.$$
Additionally, if $\xi \in \Prt$ and $\sH(\pv) < \infty$ then we define the deficiency
$$\Def{\mu}{\pv}{\xi}{\cF} = |\dist_\mu(\xi) - \pv| + |\sH_\mu(\xi) - \sH(\pv)| + |\rh_{G,\mu}(\xi \given \cF) - \sH(\pv)|.$$
By \cite[Lem. 5.2.(iii)]{AS}, the map $\xi \in \Prt \mapsto \rh_{G,\mu}(\xi \given \cF)$ is $\dC_\mu$-continuous. Hence $\xi \in \Prt \mapsto \Def{\mu}{\pv}{\xi}{\cF}$ is $\dC_\mu$-continuous as well. Similarly, the map $\mu \mapsto \sH_\mu(\xi)$ is Borel and by \cite[Cor. 6.5.(ii)]{AS} the map $\mu \in \{\nu \in \M_G(X) : \sH_\nu(\xi) < \infty\} \mapsto \rh_{G,\mu}(\xi \given \cF)$ is Borel provided $\cF$ is countably generated. So the map $\mu \in \M_G(X) \mapsto \Def{\mu}{\pv}{\xi}{\cF}$ is Borel whenever $\cF$ is countably generated. In this section and the next we will focus on aperiodic $\Z$-actions, in which case $\rh_{\Z,\mu}(\xi \given \cF)$ can be replaced with $\ksh_{\Z,\mu}(\xi \given \cF)$ in the definition of $\Def{\mu}{\pv}{\xi}{\cF}$ as the two quantities are equal \cite[Cor. 7.2]{AS}.

\begin{lem}\cite[Lem. 3]{KiRa} \label{lem:openb}
Let $X$ be a standard Borel space and let $G \acts X$ be a Borel action. Pick a subset $\cQ_\nu \subseteq \Prt_\sH(\nu)$ for each $\nu \in \E_G(X)$. Assume that $\{\nu \in \E_G(X) : \gamma \in \cQ_\nu\}$ is Borel for every $\gamma \in \Prt$ and that $\cQ_\nu$ is $\dC_\nu$-open for every $\nu \in \E_G(X)$. Then for fixed $\gamma \in \Prt$ the map $\nu \in \E_G(X) \mapsto \dB_\nu(\gamma, \cQ_\nu)$ is Borel.
\end{lem}

\begin{proof}
For fixed $k \in \N$, this is proven in \cite{KiRa} in the setting of $k$-piece partitions with the assumption that each $\cQ_\nu$ is $\dB_\nu$-open. Their proof works in our setting with obvious minor modifications.
\end{proof}

\begin{thm}\cite[Theorem 3]{KiRa} \label{thm:universe}
Let $X$ be a standard Borel space and let $G \acts X$ be a Borel action. Pick a non-empty subset $\cQ_\nu \subseteq \Prt_\sH(\nu)$ for each $\nu \in \E_G(X)$. Also pick functions $\psi_n : \E_G(X) \times \Prt \rightarrow [0, \infty)$, $n \in \N$, with the property that for all $\nu \in \E_G(X)$ the map $\gamma \mapsto \psi_n(\nu, \gamma)$ is $\dB_\nu$-continuous and for all $\gamma \in \Prt$ the map $\nu \mapsto \psi_n(\nu, \gamma)$ is Borel. Assume that for every $\nu \in \E_G(X)$ and $\gamma \in \Prt$, $\inf_n \psi_n(\nu, \gamma) = 0$ if and only if $\gamma \in \cQ_\nu$, and assume that for all $\gamma \in \Prt$ the map $\nu \in \E_G(X) \mapsto \dB_\nu(\gamma, \cQ_\nu)$ is Borel. Then there is a Borel partition $\alpha$ with $\alpha \in \bigcap_{\nu \in \E_G(X)} \cQ_\nu$.
\end{thm}

\begin{proof}
This is proved in \cite{KiRa} for $k$-piece partitions, but their proof works in our setting without modification.
\end{proof}

We now present the non-ergodic relative perturbative factor theorem for $\Z$-actions.

\begin{thm} \label{thm:kira}
Let $X$ be a standard Borel space, let $\Z \acts X$ be an aperiodic Borel action, and let $\cF$ be a countably generated $\Z$-invariant sub-$\sigma$-algebra. Fix $M, \epsilon > 0$ and a countable partition $\xi$ of $X$, and let $\psd : \R_+^2 \rightarrow \R_+$ be as in Lemma \ref{lem:zfolk}. If $\nu \mapsto \pv^\nu$ is a Borel map assigning to each $\nu \in \E_\Z(X)$ a probability vector $\pv^\nu$ satisfying $\sH(\pv^\nu) < \infty$ and $\sH(\pv^\nu) \leq \ksh_\Z(X, \nu \given \cF)$, then there is a Borel partition $\alpha$ with $\alpha \in \Bp{\nu}{\pv^\nu}{\cF}$ for every $\nu \in \E_\Z(X)$ and satisfying
$$\forall \nu \in \E_\Z(X) \quad \Big( \sH(\pv^\nu) \leq M \text{ and } \Def{\nu}{\pv^\nu}{\xi}{\cF} < \psd(M, \epsilon) \Longrightarrow \dB_\nu(\alpha, \xi) < \epsilon \Big).$$
\end{thm}

\begin{proof}
Set $N = \{ \nu : \sH(\pv^\nu) \leq M \text{ and } \Def{\nu}{\pv^\nu}{\xi}{\cF} < \psd(M, \epsilon) \}$. For $\nu \in N$ set $\cQ_\nu = \{\beta \in \Bp{\nu}{\pv^\nu}{\cF} : \dB_\nu(\xi, \beta) < \epsilon\}$ and for $\nu \not\in N$ set $\cQ_\nu = \Bp{\nu}{\pv^\nu}{\cF}$. By Lemma \ref{lem:zfolk} each $\cQ_\nu \neq \varnothing$. It suffices to build an ordered partition $\alpha$ with $\alpha \in \cQ_\nu$ for every $\nu \in \E_\Z(X)$.

We claim that for every $\gamma \in \Prt$ the map $\nu \mapsto \dB_\nu(\gamma, \cQ_\nu)$ is Borel. For $n \in \N$ set
$$\cQ_\nu^n = \begin{cases}
\{\gamma : \Def{\nu}{\pv^\nu}{\gamma}{\cF} < \delta(M, 1 / n) \text{ and } \dB_\nu(\xi, \gamma) < \epsilon - 1 / n\} & \text{if } \nu \in N\\
\{\gamma : \Def{\nu}{\pv^\nu}{\gamma}{\cF} < \delta(\sH(\pv^\nu), 1 / n)\} & \text{if } \nu \not\in N.
\end{cases}$$
For each $n$ the family $(\cQ_\nu^n)_{\nu \in \E_\Z(X)}$ satisfies the assumptions of Lemma \ref{lem:openb}, so the map $\nu \mapsto \dB_\nu(\gamma, \cQ_\nu^n)$ is Borel for every $\gamma \in \Prt$. Now fix $\gamma \in \Prt$. Each $\beta \in \cQ_\nu$ lies in $\cQ_\nu^n$ for all but finitely many $n$, so $\dB_\nu(\gamma, \cQ_\nu) \geq \limsup_{n \rightarrow \infty} \dB_\nu(\gamma, \cQ_\nu^n)$. On the other hand, by definition of $\psd$ we have $\dB_\nu(\gamma, \cQ_\nu) \leq \dB_\nu(\gamma, \cQ_\nu^n) + 1/n$. Therefore $\dB_\nu(\gamma, \cQ_\nu) = \lim_{n \rightarrow \infty} \dB_\nu(\gamma, \cQ_\nu^n)$. We conclude that the map $\nu \mapsto \dB_\nu(\gamma, \cQ_\nu)$ is Borel for every $\gamma \in \Prt$.

For $\gamma = \{C_i : i \in \N\} \in \Prt$ and $n \in \N$, let $\gamma_{(n)}$ be the partition of $X$ into the sets $C_i$, $0 \leq i < n$, and $\bigcup_{i = n}^\infty C_i$. For $n \in \N$ define $\psi_n(\nu, \gamma) = \Def{\nu}{\pv^\nu}{\gamma_{(n)}}{\cF} + 1/n$ when $\nu \not\in N$, and define
$$\psi_n(\nu, \gamma) = \Def{\nu}{\pv^\nu}{\gamma_{(n)}}{\cF} + 1/n + \max(0, \ n \cdot \dB_\nu(\xi, \gamma) - n \cdot \epsilon + 1)$$
when $\nu \in N$. Then $\psi_n(\cdot, \gamma)$ is Borel and $\psi_n(\nu, \cdot)$ is $\dB_\nu$-continuous. Also, $\beta \in \cQ_\nu$ if and only if $\inf_n \psi_n(\nu, \beta) = 0$. Therefore by Theorem \ref{thm:universe} there is a partition $\alpha \in \bigcap_{\nu \in \E_G(X)} \cQ_\nu$.
\end{proof}

We remark that we do not know how to apply the above argument to the case of a non-amenable group $G$. The issue is that for an action $G \acts X$, $\nu \in \E_G(X)$, and a Bernoulli partition $\beta \in \Bp{\nu}{\pv^\nu}{\cF}$, we may have $\rh_{G,\nu}(\beta \given \cF) < \sH_\nu(\beta)$ and hence $\Def{\nu}{\pv^\nu}{\beta}{\cF} > 0$. So it is no longer true that $\inf_n \psi_n(\nu, \beta) = 0$ if and only if $\beta \in \cQ_\nu$. One could at first hope to fix this by using the relative Rokhlin entropy of the factor associated with $\salg_G(\beta) \vee \cF$ relative to the image of $\cF$ in place of $\rh_{G,\nu}(\beta \given \cF)$. Indeed one would expect (and in many cases can verify) that this entropy would equal $\sH_\nu(\beta)$. However, with this change $\Def{\nu}{\pv^\nu}{\gamma}{\cF}$ would no longer be a continuous function of $\gamma$. We do not know how to overcome this difficulty, and we therefore will directly prove non-ergodic versions of the factor theorem for non-amenable groups.

\section{A special perturbation}

In this section we develop a specialized version of Theorem \ref{thm:kira} that will be important for our main theorem. Unlike Theorem \ref{thm:kira}, the main result in this section will be stated entirely in terms of a collection of non-ergodic measures $\Omega$, rather than in terms of their ergodic components. A key feature will be the perturbation of a partition $\xi$ to a Bernoulli partition $\alpha$ in a manner which only slightly changes its Shannon entropy.

We start with a technical lemma. Below we write $\pone$ for the trivial probability vector which has $0^{\text{th}}$-coordinate $1$ and all other coordinates $0$.

\begin{lem} \label{lem:path}
There is a monotone increasing continuous function $\theta : [0, \infty) \rightarrow [0, 2)$ with $\theta(0) = 0$ satisfying the following property. If $\pv = (p_i)$ is a probability vector with $p_0 = \max(\pv)$ and $\sH(\pv) < \infty$ then, setting $\pv^t = t \pone + (1 - t)\pv$, we have $\sH(\pv^t) \leq \sH(\pv)$ and $|\pv - \pv^t| \leq \theta(\sH(\pv) - \sH(\pv^t))$ for all $t \in [0, 1]$.
\end{lem}

\begin{proof}
For $x \in [0, 1]$ set $f(x) = \sH(x, 1 - x)$. For $r \in [0, 1)$ define
$$\psi(r) = \inf \{f(s) + r \cdot f'(s) - f(s + r) : s \in (0, 1 - r]\}.$$
It may be helpful to note that $f(s) + r \cdot f'(s)$ is a tangent-line approximation to $f(s + r)$. Notice that when $r$ is $0$ every value in the above set is $0$ and hence $\psi(0) = 0$. Furthermore, for fixed $s$ the function $r \mapsto f(s) + r \cdot f'(s) - f(s + r)$ is continuous, and it is strictly increasing since $f$ is concave down. Therefore $\psi$ is continuous and strictly increasing. Notice that $\rng(\psi) = [0, \infty)$ since $f'(s) \rightarrow \infty$ as $s \rightarrow 0$. Set $\theta = 2 \cdot \psi^{-1}$.

Fix $\pv$ with $p_0 = \max(\pv)$. Note that for any $i \neq 0$ we have
$$\sH(\pv) = f(p_0 + p_i) + (p_0 + p_i) \cdot f \left( \frac{p_0}{p_0 + p_i} \right) + (1 - p_0 - p_i) \cdot \sH \left( \frac{1}{1 - p_0 - p_i} \cdot (p_u)_{u \neq 0, i} \right).$$
Thus if we increase the value of $p_0$, decrease the value of $p_i$ by the same amount, and leave all other coordinates the same, then $\sH(\pv)$ will decrease. By repeating this with $i$ varying, we find that if we increase $p_0$ and decrease all other coordinates then $\sH(\pv)$ will decrease.

Set $\qv = \frac{1}{1 - p_0} \cdot (p_i)_{i > 0}$ and $\pv^t = t \pone + (1 - t) \pv$. Note that $|\pv - \pv^t| = 2 t (1 - p_0)$ and
\begin{equation} \label{eqn:path}
\sH(\pv^t) = f(t + (1 - t)p_0) + (1 - t - (1 - t)p_0) \cdot \sH(\qv).
\end{equation}
From the previous paragraph we know that $\sH(\pv^t)$ is decreasing. In particular, $f'(p_0) \leq \sH(\qv)$ since the derivative of $\sH(\pv^t)$ at $t = 0$ is
$$(1 - p_0) f'(p_0) - (1 - p_0) \cdot \sH(\qv) \leq 0.$$
Now fix $t \in [0, 1]$, set $r = \frac{1}{2} |\pv - \pv^t| = t (1 - p_0)$ and set $s = p_0 \leq 1 - r$. Using $f'(p_0) \leq \sH(\qv)$ and (\ref{eqn:path}), we obtain
\begin{align*}
\sH(\pv) - \sH(\pv^t) & = f(s) + (1 - s) \cdot \sH(\qv) - f(s + r) - (1 - s - r) \cdot \sH(\qv)\\
 & \geq f(s) - f(s + r) + r \cdot f'(s) \geq \psi( \textstyle{\frac{1}{2}} |\pv - \pv^t|).
\end{align*}
Therefore $|\pv - \pv^t| \leq \theta(\sH(\pv) - \sH(\pv^t))$ as claimed.
\end{proof}

The previous lemma allows us to choose probability vectors in an intelligent way.

\begin{cor} \label{cor:avgpath}
For every $\epsilon > 0$ there is $\kappa = \kappa(\epsilon) > 0$ having the following property. Let $X$ be a standard Borel space and let $x \mapsto \qv^x$ and $x \mapsto h_x \leq \sH(\qv^x) < \infty$ be Borel maps. Also let $\Omega \subseteq \Prob(X)$ be Borel and let $m : X \rightarrow \Omega$ be a surjective Borel function with $\omega(m^{-1}(\omega)) = 1$ for all $\omega \in \Omega$. Then there is a Borel map $x \mapsto \pv^x$ such that $\sH(\pv^x) = h_x$, $\forall \omega \in \Omega \ \sH(\int_X \pv^x \ d \omega(x)) \leq \sH(\int_X \qv^x \ d \omega(x))$, and
$$\forall \omega \in \Omega \ \forall x \in m^{-1}(\omega) \quad \Big( |\textstyle{\int_X} \qv^z \ d \omega(z) - \qv^x| < \kappa \text{ and } |\sH(\qv^x) - h_x| < \kappa \Big) \Longrightarrow |\pv^x - \qv^x| < \epsilon.$$
\end{cor}

\begin{proof}
Set $f(t) = H(t, 1 - t)$ for $t \in [0, 1]$ and let $\theta$ be as in Lemma \ref{lem:path}. Choose $0 < \kappa < \epsilon / 2$ so that $2 \kappa + \theta(f(\kappa) + \kappa) < \epsilon$. For $\omega \in \Omega$ set $\qv^\omega = \int_X \qv^x \ d \omega(x)$ and let $i(\omega) \in \N$ be least with $q_{i(\omega)}^\omega = \max(\qv^\omega)$. Note that $\qv^\omega$ and $i(\omega)$ are Borel functions of $\omega$. For $i \in \N$ let $\pone_i$ be the probability vector with $i^{\text{th}}$-coordinate $1$ and all other coordinates $0$. For $x \in X$ and $t \in [0, 1]$ set $\qv^{x,t} = t \pone_{i(m(x))} + (1 - t) \qv^x$. The value $\sH(\qv^{x,t})$ moves continuously from $\sH(\qv^x)$ to $0$, so we may define $\pv^x = \qv^{x,v}$ where $v$ is least satisfying $\sH(\qv^{x,v}) = h_x$. Then $x \mapsto \pv^x$ is Borel and $\sH(\pv^x) = h_x$. Set $\pv^\omega = \int_X \pv^x \ d \omega(x)$. Then $p_{i(\omega)}^\omega \geq q_{i(\omega)}^\omega$ while $p_j^\omega \leq q_j^\omega$ for all $j \neq i(\omega)$. As argued in the proof of Lemma \ref{lem:path}, this implies that $\sH(\pv^\omega) \leq \sH(\qv^\omega)$.

Finally, fix $\omega$ and $x \in m^{-1}(\omega)$ with $|\qv^\omega - \qv^x| < \kappa$ and $\sH(\qv^x) - \kappa < h_x \leq \sH(\qv^x)$. Define $v$ as above. For any $j \in \N$ we have
$$q_j^x \leq q_j^\omega + |q_j^\omega - q_j^x| \leq q_{i(\omega)}^\omega + |q_j^\omega - q_j^x| \leq q_{i(\omega)}^x + |q_{i(\omega)}^\omega - q_{i(\omega)}^x| + |q_j^\omega - q_j^x| < q_{i(\omega)}^x + \kappa.$$
Thus $\max(\qv^x) - q_{i(\omega)}^x < \kappa$. So there is a $u$ with $u (1 - q_{i(\omega)}^x) < \kappa$ and $q^{x,u}_{i(\omega)} = \max(\qv^{x,u})$. If $u \geq v$ then $|\pv^x - \qv^x| \leq |\qv^{x,u} - \qv^x| = 2 u (1 - q_{i(\omega)}^x) < 2 \kappa < \epsilon$ and we are done. So assume $u < v$ and set $\rv^x = \frac{1}{1 - q_{i(\omega)}^x} \cdot (q_j^x)_{j \neq i(\omega)}$. Notice
$$\sH(\qv^{x,t}) = f(t + (1 - t)q_{i(\omega)}^x) + (1 - t - (1 - t)q_{i(\omega)}^x) \cdot \sH(\rv^x).$$
Using the concavity of $f$ and the inequality $u + (1 - u)q_{i(\omega)}^x = q_{i(\omega)}^x + u (1 - q_{i(\omega)}^x) < q_{i(\omega)}^x + \kappa$, we observe that
\begin{align*}
\sH(\qv^{x,u}) - h_x & = \sH(\qv^{x,u}) - \sH(\qv^x) + \sH(\qv^x) - h_x\\
 & \leq f(u + (1 - u)q_{i(\omega)}^x) - f(q_{i(\omega)}^x) + \kappa\\
 & \leq f(\kappa) - f(0) + \kappa.
\end{align*}
Noting that $f(0) = 0$ and applying Lemma \ref{lem:path} to $\qv^{x,u}$ we conclude
\begin{align*}
|\qv^x - \pv^x| \leq |\qv^x - \qv^{x,u}| + |\qv^{x,u} - \qv^{x,v}| & \leq 2 u (1 - q_{i(\omega)}^x) + \theta(\sH(\qv^{x,u}) - h_x)\\
 & < 2 \kappa + \theta(f(\kappa) + \kappa) < \epsilon.\qedhere
\end{align*}
\end{proof}

Now we present the main result of this section, a specialized version of Theorem \ref{thm:kira}. It may be helpful to recall Lemma \ref{lem:nonergb}.

\begin{cor} \label{cor:sfolk}
Let $X$ be a standard Borel space, $\Z \acts X$ an aperiodic Borel action, and $\cF$ a countably-generated $\Z$-invariant sub-$\sigma$-algebra. Let $\Omega \subseteq \M_\Z(X)$ be a Borel set and let $m : X \rightarrow \Omega$ be a $\Z$-invariant Borel surjection with $\omega(m^{-1}(\omega)) = 1$ for all $\omega \in \Omega$. If $\xi$ is a countable Borel partition of $X$ with $\ksh_{\Z,\omega}(\xi \given \cF) < \infty$ for all $\omega \in \Omega$ then there is a Borel partition $\alpha$ such that, for every $\omega \in \Omega$, $\alpha$ is weakly $(\Z, \omega)$-Bernoulli over $\sinv_\Z$, $\salg_\Z(\alpha)$ is $\omega$-independent with $\cF$ relative to $\sinv_\Z$, and $\sH_\omega(\alpha \given \sinv_\Z) = \ksh_{\Z,\omega}(\xi \given \cF)$.

In fact, for $M, \epsilon > 0$,  there exists $\upsilon = \upsilon(M, \epsilon) > 0$ with the following property. If $\sH_\omega(\xi) \leq M$ and $\sH_\omega(\xi) - \ksh_{\Z,\omega}(\xi \given \cF) < \upsilon$ for all $\omega \in \Omega$, then the partition $\alpha$ above may be chosen with the additional property that $\dB_\omega(\alpha, \xi) < \epsilon$ and $|\sH_\omega(\alpha) - \sH_\omega(\xi)| < \epsilon$ for every $\omega \in \Omega$.
\end{cor}

\begin{proof}
The first statement is nearly an immediate consequence of Theorem \ref{thm:kira}, as our argument will reveal below. We focus on proving the second claim. Let $\psd : \R_+^2 \rightarrow \R_+$ be as in Lemma \ref{lem:zfolk}. Choose $M' > 12 M / \epsilon$ and set $\epsilon' = (1/2) \psd(M', \epsilon / 2)$. Next, let $\kappa = \kappa(\epsilon')$ be as in Corollary \ref{cor:avgpath}. If necessary, we may shrink $\kappa$ so that $\kappa < \epsilon / 12$ and $\kappa < (1/2) \psd(M', \epsilon / 2)$. Finally, let $\upsilon = \eta(M, \kappa)$ be as in Lemma \ref{lem:drop}. If necessary, shrink $\upsilon$ so that $\upsilon < (1 / 12) \kappa \epsilon$.

Let $\Z \acts X$, $\cF$, $m : X \rightarrow \Omega$, and $\xi$ be as in the statement of the corollary. We seek to apply Corollary \ref{cor:avgpath} to the space $\E_\Z(X)$ where $\qv^\nu = \dist_\nu(\xi)$ and $h^\nu = \ksh_{\Z,\nu}(\xi \given \cF)$ for $\nu \in \E_\Z(X)$. To do this we must define a Borel set $L \subseteq \Prob(\E_\Z(X))$ and a surjective Borel map $\tau : \E_\Z(X) \rightarrow L$ with $\lambda(\tau^{-1}(\lambda)) = 1$ for all $\lambda \in L$. In brief, each $\nu \in \E_G(X)$ gets mapped to a unique $\omega \in \Omega$, the $\nu$-almost-everywhere constant value of $m$, and then $\omega$ gets mapped to the measure $\lambda \in \Prob(\E_\Z(X))$ describing its ergodic decomposition.

We now define $L$ and $\tau$ precisely. By the ergodic decomposition theorem \cite{Far62,Var63}, there is a $\Z$-invariant Borel map $e : X \rightarrow \E_\Z(X)$ such that $\nu(e^{-1}(\nu)) = 1$ for every $\nu \in \E_\Z(X)$. For $\nu \in \E_\Z(X)$ we define $\tau(\nu) \in \Prob(\E_\Z(X))$ to be $e_*(\omega)$, where $\omega = m(x)$ for $\nu$-almost-every $x \in X$. In other words, for Borel $N \subseteq \E_\Z(X)$
$$\tau(\nu)(N) = \int_X e_*(m(x))(N) \ d \nu(x).$$
This last formula shows that $\tau$ is Borel. Set $L = \tau(\E_\Z(X)) = \{e_*(\omega) : \omega \in \Omega\}$. Since $\omega = \int_X e(x) \ d \omega(x) = \int_{\E_\Z(X)} \nu \ d e_*(\omega)(\nu)$, we see that the map $\omega \mapsto e_*(\omega)$ from $\Omega$ to $L$ is injective, and thus $L$ is Borel \cite[Cor. 15.2]{K95}. Finally, fix $\omega \in \Omega$ and set $\lambda = e_*(\omega)$. Since $\omega = \int \nu \ d e_*(\omega)(\nu)$ and $m(x) = \omega$ for $\omega$-almost-every $x \in X$, it follows that for $\lambda$-almost-every $\nu \in \E_\Z(X)$ we have $m(x) = \omega$ for $\nu$-almost-every $x \in X$ and hence $\tau(\nu) = \lambda$. So $\lambda(\tau^{-1}(\lambda)) = 1$.

Apply Corollary \ref{cor:avgpath} to obtain a Borel map $\nu \mapsto \pv^\nu$ with $\sH(\pv^\nu) = \ksh_{\Z,\nu}(\xi \given \cF)$ such that, for all $\omega$, $\sH(\int_{\E_\Z(X)} \pv^\nu \ d e_*(\omega)(\nu)) \leq \sH(\int_{\E_\Z(X)} \dist_\nu(\xi) \ d e_*(\omega)(\nu)) = \sH_\omega(\xi)$, and such that for all $\omega$ and every $\nu \in \tau^{-1}(e_*(\omega))$ we have
$$\Big( |\dist_\omega(\xi) - \dist_\nu(\xi)| < \kappa \text{ and } |\sH_\nu(\xi) - \ksh_{\Z,\nu}(\xi \given \cF)| < \kappa \Big) \Longrightarrow |\pv^\nu - \dist_\nu(\xi)| < \epsilon'.$$
Now apply Theorem \ref{thm:kira} to obtain a Borel partition $\alpha \in \bigcap_\nu \Bp{\nu}{\pv^\nu}{\cF}$ satisfying
$$\forall \nu \in \E_\Z(X) \quad \Big( \sH(\pv^\nu) \leq M' \text{ and } \Def{\nu}{\pv^\nu}{\xi}{\cF} < \psd(M', \epsilon / 2) \Longrightarrow \dB_\nu(\alpha, \xi) < \epsilon / 2 \Big).$$
In particular, $\alpha$ is weakly $(\Z,\omega)$-Bernoulli over $\sinv_\Z$ and $\salg_\Z(\alpha)$ is independent with $\cF$ relative to $\sinv_\Z$ by Lemma \ref{lem:nonergb}.

Now fix $\omega \in \Omega$ and set $\lambda = e_*(\omega)$. We will check that $\dB_\omega(\alpha, \xi), |\sH_\omega(\alpha) - \sH_\omega(\xi)| < \epsilon$. Kolmogorov--Sinai entropy is an affine function on the space of invariant measures, so
$$\int_{\E_\Z(X)} \ksh_{\Z,\nu}(\xi \given \cF) \ d \lambda(\nu) = \ksh_{\Z,\omega}(\xi \given \cF).$$
Combined with the inequalities $\ksh_{\Z,\omega}(\xi \given \cF) \leq \sH_\omega(\xi) \leq M$, this implies
\begin{equation} \label{eqn:sfolk1}
\lambda(\{\nu \in \E_\Z(X) : \ksh_{\Z,\nu}(\xi \given \cF) \leq M'\}) \geq 1 - \frac{M}{M'} > 1 - \frac{\epsilon}{12}.
\end{equation}
Next, by applying Lemma \ref{lem:drop} to the inequality
$$\sH_\omega(\xi \given \sinv_\Z) \geq \ksh_{\Z,\omega}(\xi \given \cF \vee \sinv_\Z) = \ksh_{\omega,\Z}(\xi \given \cF) > \sH_\omega(\xi) - \upsilon$$
we find that
\begin{equation} \label{eqn:sfolk2}
\lambda(\{\nu \in \E_\Z(X) : |\dist_\nu(\xi) - \dist_\omega(\xi)| < \kappa\}) > 1 - \kappa > 1 - \epsilon / 12.
\end{equation}
Finally, since $\sH_\nu(\xi) - \ksh_{\Z,\nu}(\xi \given \cF) \geq 0$ for all $\nu$ and
\begin{align*}
\int_{\E_\Z(X)} \Big(\sH_\nu(\xi) - \ksh_{\Z,\nu}(\xi \given \cF) \Big) \ d \lambda(\nu) & = \sH_\omega(\xi \given \sinv_\Z) - \ksh_{\Z,\omega}(\xi \given \cF)\\
 & \leq \sH_\omega(\xi) - \ksh_{\Z,\omega}(\xi \given \cF) < \upsilon,
\end{align*}
we must have
\begin{equation} \label{eqn:sfolk3}
\lambda(\{\nu \in \E_\Z(X) : |\sH_\nu(\xi) - \ksh_{\Z,\nu}(\xi \given \cF)| < \kappa\}) > 1 - \frac{\upsilon}{\kappa} > 1 - \frac{\epsilon}{12}.
\end{equation}
Let $N$ be the set of $\nu$ with $\ksh_{\Z,\nu}(\xi \given \cF) \leq M'$, $|\dist_\nu(\xi) - \dist_\omega(\xi)| < \kappa$, and $|\sH_\nu(\xi) - \ksh_{\Z,\nu}(\xi \given \cF)| < \kappa$. Inequalities (\ref{eqn:sfolk1}) (\ref{eqn:sfolk2}) (\ref{eqn:sfolk3}) show that $\lambda(N) > 1 - \epsilon / 4$.

Now notice that for $\nu \in N$ the construction of $\pv^\nu$ gives $|\pv^\nu - \dist_\nu(\xi)| < \epsilon'$. Additionally, we have $\sH(\pv^\nu) = \ksh_{\Z,\nu}(\xi \given \cF) \leq M'$ and
\begin{align*}
\Def{\nu}{\pv^\nu}{\xi}{\cF} & = |\dist_\nu(\xi) - \pv^\nu| + |\sH_\nu(\xi) - \sH(\pv^\nu)| + |\ksh_{\Z,\nu}(\xi \given \cF) - \sH(\pv^\nu)|\\
 & < \epsilon' + |\sH_\nu(\xi) - \ksh_{\Z,\nu}(\xi \given \cF)|  + 0\\
 & < \epsilon' + \kappa < \psd(M', \epsilon / 2).
\end{align*}
So $\dB_\nu(\alpha, \xi) < \epsilon / 2$ for $\nu \in N$. Since $\lambda(N) > 1 - \epsilon / 4$ we obtain
$$\dB_\omega(\alpha, \xi) = \int_{\E_\Z(X)} \dB_\nu(\alpha, \xi) \ d \lambda(\nu) < \frac{\epsilon}{2} + \frac{\epsilon}{2} = \epsilon.$$
Lastly, the ergodic decomposition of $\omega$ coincides with the disintegration of $\omega$ relative to $\sinv_\Z$, so
$$\sH_\omega(\alpha \given \sinv_\Z) = \int_{\E_\Z(X)} \sH_\nu(\alpha) \ d \lambda(\nu) = \int_{\E_\Z(X)} \ksh_{\Z,\nu}(\xi \given \cF) = \ksh_{\Z,\omega}(\xi \given \cF).$$
Therefore, 
$$\sH_\omega(\xi) - \epsilon < \ksh_{\Z,\omega}(\xi \given \cF) = \sH_\omega(\alpha \given \sinv_\Z) \leq \sH_\omega(\alpha) = \sH \left( \int_{\E_\Z(X)} \pv^\nu \ d \lambda(\nu) \right) \leq \sH_\omega(\xi),$$
and thus $|\sH_\omega(\alpha) - \sH_\omega(\xi)| < \epsilon$.
\end{proof}

\section{A smooth division into hyperfinite relations}

One of the key ingredients in the proof of the main theorem is a new construction related to countable Borel equivalence relations. Specifically, for an aperiodic countable Borel equivalence relation $E$ on a standard Borel space $X$, we show that there is a smooth division $X = \sqcup_{0 \leq r \leq 1} X_r$ of $X$ into Borel sets $X_r$ with the property that the restriction of $E$ to each $X_r$ is aperiodic and hyperfinite. This section is devoted to proving this fact.

Recall that for a standard Borel space $X$, a \emph{Borel graph} on $X$ is a Borel, symmetric, anti-reflexive subset $\Gamma \subseteq X \times X$. For $A \subseteq X$ we write $\Nbr_\Gamma(A)$ for the \emph{$\Gamma$-neighborhood of $A$}, i.e. the set of $x \in X$ such that there is $a \in A$ with $(x, a) \in \Gamma$. The \emph{degree of $x \in X$} is $\deg_\Gamma(x) = |\Nbr_\Gamma(x)|$, and $\Gamma$ is \emph{locally finite} if $\deg_\Gamma(x) < \infty$ for all $x \in X$. A set $A \subseteq X$ is \emph{$\Gamma$-independent} if for all $a, b \in A$ we have $(a, b) \not\in \Gamma$, and a function $f : X \rightarrow \N$ is a \emph{$\Gamma$-coloring} if $f^{-1}(i)$ is $\Gamma$-independent for every $i \in \N$. It was proven by Kechris--Solecki--Todorcevic in \cite{KST99} that every locally finite Borel graph $\Gamma$ admits a Borel $\Gamma$-coloring $f : X \rightarrow \N$ and also admits a maximal $\Gamma$-independent set that is Borel. We write $E_\Gamma^X$ for the equivalence relation on $X$ given by the connected components of $\Gamma$, and we say $\Gamma$ is \emph{aperiodic} if every class of $E_\Gamma^X$ is infinite.

\begin{lem} \label{lem:einf}
Let $E$ be an aperiodic countable Borel equivalence relation on a standard Borel space $X$. Then there is a Borel function $f : X \rightarrow \N$ such that $|[x]_E \cap f^{-1}(i)| = \infty$ for all $x \in X$ and all $i \in \N$.
\end{lem}

\begin{proof}
We claim that under these assumptions there is a Borel set $A \subseteq X$ with $|[x]_E \cap A| = |[x]_E \cap (X \setminus A)| = \infty$ for all $x \in X$. By \cite[Lem. 3.25]{JKL} there exists an aperiodic $S \in [E]$. Let $\Gamma$ be the graph $\{(x, S(x)) : x \in X\} \cup \{(S(x), x) : x \in X\}$. By \cite[Prop. 4.2 and Prop. 4.6]{KST99}, there is a maximal $\Gamma$-independent Borel set $A \subseteq X$. Independence gives $A \cap S(A) = \varnothing$ and maximality implies $X = S^{-1}(A) \cup A \cup S(A)$. Clearly $A$ has the desired properties, establishing our preliminary claim.

Set $X_0 = X$. Inductively on $n$, given $X_n \subseteq X$ having infinite intersection with every $E$-class, apply the previous claim to the relation $E_n = E \res (X_n \times X_n)$ to get a Borel set $A_n \subseteq X_n$ with $|[x]_{E_n} \cap A_n| = |[x]_{E_n} \cap (X_n \setminus A_n)| = \infty$ for all $x \in X_n$. Then $A_n$ and $X_n \setminus A_n$ have infinite intersection with every $E$-class. Set $X_{n+1} = X_n \setminus A_n$ and continue the induction. This creates a sequence $(A_n)_{n \in \N}$ of pairwise-disjoint Borel sets having infinite intersection with every $E$-class. To complete the proof, define $f(x) = n$ if $x \in A_n$ and $f(x) = 0$ if $x \not\in \bigcup_n A_n$.
\end{proof}

\begin{lem} \label{lem:ginf}
Let $X$ be a standard Borel space, let $E$ be an aperiodic countable Borel equivalence relation on $X$, and let $\Gamma \subseteq E$ be a locally-finite Borel graph on $X$. Then there is a Borel $\Gamma$-coloring $f : X \rightarrow \N$ such that $|[x]_E \cap f^{-1}(i)| = \infty$ for all $x \in X$ and all $i \in \N$.
\end{lem}

\begin{proof}
By \cite{KST99} there is a $\Gamma$-coloring $h : X \rightarrow \{n \in \N : n \geq 1\}$ and a maximal $\Gamma$-independent set $P_0$ that is Borel. For $i \geq 1$ set $P_i = h^{-1}(i) \setminus P_0$, so that $\{P_n : n \in \N\}$ is a partition of $X$ into $\Gamma$-independent sets. Since $E$ is aperiodic, $\Gamma$ is locally finite, and $P_0$ is a maximal $\Gamma$-independent set, we must have that $[x]_E \cap P_0$ is infinite for every $x \in X$. Consider the restriction of $E$ to $P_0 \times P_0$ and apply the previous lemma to obtain a Borel function $f : P_0 \rightarrow \N$ with $|[x]_E \cap P_0 \cap f^{-1}(n)| = \infty$ for every $x \in X$ and $n \in \N$. Now suppose $f$ is defined on $\bigcup_{0 \leq j < i} P_j$ and gives distinct values to $\Gamma$-adjacent points. We will extend $f$ to $\bigcup_{0 \leq j \leq i} P_j$. For each $y \in P_i$ define $f(y)$ to be the least element of $\N \setminus f(\Nbr_\Gamma(y) \cap \bigcup_{0 \leq j < i} P_j)$. Then $f$ continues to be a Borel function, and since $(P_i \times P_i) \cap \Gamma = \varnothing$, $f$ continues to assign distinct values to $\Gamma$-adjacent points in $\bigcup_{0 \leq j \leq i} P_j$. Continuing by induction, we obtain a Borel $\Gamma$-coloring $f : X \rightarrow \N$. For each $x \in X$ and $n \in \N$ we have $|[x]_E \cap f^{-1}(n)| \geq |[x]_E \cap P_0 \cap f^{-1}(n)| = \infty$.
\end{proof}

Recall that a Borel equivalence relation $F$ on $X$ is \emph{smooth} if there is a standard Borel space $Y$ and a Borel function $f : X \rightarrow Y$ such that $x \ F \ x' \Leftrightarrow f(x) = f(x')$. Also recall that a Borel equivalence relation $R$ on $X$ is \emph{hyperfinite} if there is an increasing sequence of Borel equivalence relations $(R_n)_{n \in \N}$ such that $R = \bigcup_n R_n$ and each $R_n$ is finite (i.e. each class of $R_n$ is finite).

\begin{thm} \label{thm:slice}
Let $E$ be an aperiodic countable Borel equivalence relation on a standard Borel space $X$. Then there is a smooth Borel equivalence relation $F$ such that $E \cap F$ is aperiodic and hyperfinite.
\end{thm}

\begin{proof}
By a result of Feldman and Moore \cite{FM}, there is a countable group $G$ and a Borel action $G \acts X$ such that $E = \{(x, y) : \exists g \in G \ g \cdot x = y\}$. Fix an increasing sequence of finite symmetric sets $K_n \subseteq G$ with $\bigcup_n K_n = G$. We will build a sequence of Borel functions $f_n : X \rightarrow \N$ and an increasing sequence of finite Borel equivalence relations $R_n \subseteq E$ satisfying:
\begin{enumerate}
\item [\rm (i)] for all $n \in \N$, $k \leq n$, and $x, y \in X$, if $x \ R_n \ y$ then $f_k(x) = f_k(y)$;
\item [\rm (ii)] for all $x \in X$, $n \in \N$, and $v \in \N^n$, $|[x]_E \cap (f_0 \times \cdots \times f_{n-1})^{-1}(v)| = \infty$;
\item [\rm (iii)] for all $n \geq 1$ and $x, y \in X$, if $y \in K_n \cdot x$ then either $x \ R_{n-1} \ y$ or else $f_k(x) \neq f_k(y)$ for some $k \leq n$;
\item [\rm (iv)] for all $n \geq 1$, every $R_n$-class contains at least two $R_{n-1}$-classes.
\end{enumerate}
To begin let $f_0 : X \rightarrow \N$ be given by Lemma \ref{lem:einf} and let $R_0 = \{(x, x) : x \in X\}$ be the equivalence relation of equality.

Now assume that $f_k$ and $R_k$ have been constructed for all $k < n$ and satisfy (i) through (iv). Since $R_{n-1}$ is finite, there is a Borel set $Y \subseteq X$ that meets every $R_{n-1}$-class in precisely one point and there is a Borel function $t : X \rightarrow Y$ satisfying $t(x) \ R_{n-1} \ x$ for every $x$. Define an equivalence relation $H$ on $Y$ by the rule $y \ H \ y' \Leftrightarrow (y \ E \ y') \wedge (\forall k < n \ f_k(y) = f_k(y'))$. Then $H$ is Borel, $H \subseteq E \cap (Y \times Y)$, and clauses (i) and (ii) imply that $H$ is aperiodic. Define a Borel graph $\Gamma$ on $Y$ by the rule
$$(y, y') \in \Gamma \Longleftrightarrow (y \neq y') \wedge (y \ H \ y') \wedge ([y]_{R_{n-1}} \cap K_n \cdot [y']_{R_{n-1}} \neq \varnothing).$$
Note that $\Gamma$ is locally finite since $R_{n-1}$ and $K_n$ are finite. By Lemma \ref{lem:ginf} there is a Borel $\Gamma$-coloring $h : Y \rightarrow \N$ such that $|[y]_H \cap h^{-1}(i)| = \infty$ for all $y \in Y$ and $i \in \N$. Then the equivalence relation $H'$ defined by $y \ H' \ y' \Leftrightarrow (y \ H \ y') \wedge (h(y) = h(y'))$ is an aperiodic countable Borel equivalence relation. Pick any finite Borel equivalence relation $S \subseteq H'$ with the property that every $S$-class contains at least two points. Now define $f_n(x) = h(t(x))$, and set $x \ R_n \ x'$ if and only if $(t(x), t(x')) \in S$. Then $f_n$ and $R_n$ are Borel, and its immediately seen that (i), (ii), and (iv) continue to hold. Also (iii) holds due to how the graph $\Gamma$ was defined and the fact that the function $h$ is a $\Gamma$-coloring.

Now define the smooth Borel equivalence relation $F$ by the rule
$$x \ F \ y \Longleftrightarrow f_\infty(x) = f_\infty(y),$$
where $f_\infty : X \rightarrow \N^\N$ is defined by $f_\infty(x)(n) = f_n(x)$. Also set $R_\infty = \bigcup_n R_n$. Clearly $R_\infty$ is hyperfinite, and since every $R_{n+1}$-class contains at least two $R_n$-classes, $R_\infty$ is aperiodic. To finish the proof it suffices to show that $R_\infty = E \cap F$. Clause (i) immediately gives the containment $R_\infty \subseteq E \cap F$. For the reverse direction, suppose $(x, y) \in E \cap F$. Let $n$ be sufficiently large that $y \in K_n \cdot x$. Then clause (iii) implies that either $(x, y) \in R_{n-1} \subseteq R_\infty$ or else $(x, y) \not\in F$. The latter contradicts our assumption, so the former must hold and $E \cap F \subseteq R_\infty$.
\end{proof}

We will rely on the following corollary to Theorem \ref{thm:slice}.

\begin{cor} \label{cor:slice}
Let $X$ be a standard Borel space, $G \acts X$ an aperiodic Borel action, and let $\cF$ be a countably generated class-bijective $G$-invariant sub-$\sigma$-algebra. Then there is an aperiodic $\cF$-expressible $T \in [E_G^X]$ and a $T$-invariant $\cF$-measurable function $f : X \rightarrow [0, 1]$ with the property that $x, y \in X$ lie in the same $G$-orbit and have equal $f$ values if and only if they lie in the same $T$-orbit.
\end{cor}

\begin{proof}
By Lemma \ref{lem:borelfactor} there is an aperiodic Borel action $G \acts Y$ on a standard Borel space $Y$ and a class-bijective $G$-equivariant Borel map $\phi : X \rightarrow Y$ with $\cF = \phi^{-1}(\Borel(Y))$. Let $E_G^Y$ be the orbit equivalence relation on $Y$, and let $F'$ be the smooth Borel equivalence relation given by Theorem \ref{thm:slice}. Since $F'$ is smooth and every standard Borel space embeds into $[0, 1]$, there is a Borel function $h : Y \rightarrow [0, 1]$ such that $y \ F' \ y' \Leftrightarrow h(y) = h(y')$. Since $F' \cap E_G^Y$ is hyperfinite and aperiodic, there is an aperiodic automorphism $S$ of $Y$ such that its induced orbit equivalence relation, $E^Y_S$, coincides with $F' \cap E_G^Y$ \cite[Prop. 1.2]{JKL}. In particular, $S \in [E_G^Y]$ so there is a Borel partition $\{B^S_g : g \in G\}$ of $Y$ with $S(y) = g \cdot y$ for all $g \in G$ and $y \in B^S_g$. Define a $\cF$-measurable function $f : X \rightarrow [0, 1]$ by $f(x) = h \circ \phi(x)$, and define an aperiodic $\cF$-expressible $T \in [E_G^X]$ by setting $T(x) = g \cdot x$ if $\phi(x) \in B^S_g$. Notice that $\phi \circ T = S \circ \phi$ and hence $f$ is $T$-invariant since $h$ is $S$-invariant. Clearly any two points lying in the same $T$-orbit will also lie in the same $G$-orbit and have equal $f$ values. So consider the reverse scenario where $x \in X$, $x' \in G \cdot x$, and $f(x) = f(x')$. Setting $y = \phi(x)$ and $y' = \phi(x')$, we have $y' \in G \cdot y$ and $h(y) = h(y')$. It follows there is $n \in \Z$ with $S^n(y) = y'$. Then $\phi(T^n(x)) = y' = \phi(x')$. Since $\phi$ is class-bijective this implies $T^n(x) = x'$.
\end{proof}

The transformation $T$ in the previous corollary cannot be ergodic. In fact, from the perspective of the $G$-action $T$ must be highly non-ergodic, as the $T$-invariant function $f$ must separate all $T$-orbits within any common $G$-orbit.

\section{The external past} \label{sec:extp}

The notion of `past' plays a significant role in the classical entropy theory of actions of $\Z$. An important consequence of Corollary \ref{cor:slice} is that it allows the creation of a useful notion of past. Specifically, the action of $T$ imbues a natural order to every $T$-orbit, and the function $f$ creates a linear order on the collection of $T$-orbits contained in any $G$-orbit. Combined, these create a total ordering of every $G$-orbit and hence a notion of `past' for each point. Given $x$, we refer to the points $g \cdot x$ with $f(g \cdot x) < f(x)$ as the `external past' of $x$ (as this portion of $x$'s past is external to the $T$-orbit of $x$).

For a partition $\alpha$ of $X$ and $Y \subseteq X$ we define another partition of $X$ by
$$\alpha \res Y = \{X \setminus Y\} \cup \{A \cap Y : A \in \alpha\}.$$
Similarly, for a $\sigma$-algebra $\Sigma$ we write $\Sigma \res Y$ for the $\sigma$-algebra on $X$ generated by the sets $\{X\} \cup \{B \cap Y : B \in \Sigma\}$.

\begin{defn}
Let $G \acts X$ be an aperiodic Borel action, let $\cF$ be a countably generated class-bijective $G$-invariant sub-$\sigma$-algebra, and let $T \in [E_G^X]$ and $f : X \rightarrow [0, 1]$ be as in Corollary \ref{cor:slice}. For a countable partition $\xi$ of $X$ and $S \subseteq [0, 1]$ we will write $\xi_S$ for the partition $\xi_S = \xi \res f^{-1}(S)$. We define the \emph{external past} of $\xi$ as
$$\sP_\xi = \cF \vee \bigvee_{t \in [0, 1]} \Big( \salg_G(\xi_{[0,t)}) \res f^{-1}([t,1]) \Big).$$
\end{defn}

In other words, the function $f$ gives a quasi-ordering to the $T$-orbits (which, on every $G$-orbit restricts to a total ordering), and the external past of $\xi$ consists of the sets which you can measure by using $G$ to travel to strictly ``smaller'' $T$-orbits and looking at the partition $\xi$ (we also include $\cF$ in this $\sigma$-algebra for technical reasons).

\begin{lem} \label{lem:allpast}
Let $G \acts (X, \mu)$ be an aperiodic {\pmp} action, let $\cF$ be a $G$-invariant class-bijective sub-$\sigma$-algebra, and let $T \in [E_G^X]$ and $f : X \rightarrow [0, 1]$ be as in Corollary \ref{cor:slice}. Fix a countable partition $\xi$.
\begin{enumerate}
\item[\rm (i)] If $\cF$ is countably generated then so is $\sP_\xi$.
\item[\rm (ii)] $\sP_\xi$ is $T$-invariant.
\item[\rm (iii)] $\salg_T(\xi) \vee \sP_\xi \subseteq \cF \vee \salg_G(\xi)$.
\item[\rm (iv)] $(\salg_T(\xi) \vee \sP_\xi) \res f^{-1}([0, t)) \subseteq \cF \vee \salg_G(\xi_{[0, t)})$ for every $t \in [0, 1]$.
\item[\rm (v)] If $\beta \subseteq \salg_T(\xi) \vee \sP_\xi$, then $\sP_\beta \subseteq \sP_\xi$.
\end{enumerate}
\end{lem}

\begin{proof}
(i). For any set $C$ we have $C \cap f^{-1}([0, t)) = \bigcup_{t > q \in \Q} C \cap f^{-1}([0, q))$. So
\begin{align*}
\salg_G(\xi_{[0,t)}) \res f^{-1}([t, 1]) & \subseteq \bigvee_{t > q \in \Q} \salg_G(\xi_{[0,q)}) \res f^{-1}([t, 1])\\
 & \subseteq \cF \vee \bigvee_{t > q \in \Q} \salg_G(\xi_{[0,q)}) \res f^{-1}([q, 1]).
\end{align*}
Thus in the definition of $\sP_\xi$ one may take the join over $[0,1] \cap \Q$ rather than $[0, 1]$.

(ii). Since $f$ is $\cF$-measurable, we can write $\sP_\xi = \bigvee_{t \in [0, 1]} (\cF \vee \salg_G(\xi_{[0,t)})) \res f^{-1}([t,1])$. By Lemma \ref{lem:expmove} $\cF \vee \salg_G(\xi_{[0,t)})$ is $T$-invariant. Also $f$ is $T$-invariant, so the claim follows.

(iii). The $\sigma$-algebra $\cF \vee \salg_G(\xi)$ is $G$ and $T$ invariant by Lemma \ref{lem:expmove}, and it contains the sets $f^{-1}([t,1])$ and the partitions $\xi_{[0, t)}$. So the claim is immediate.

(iv). Fix $t$. Using Lemma \ref{lem:expmove} and the $T$-invariance of $f$ we get $\salg_T(\xi) \res f^{-1}([0, t)) = \salg_T(\xi_{[0,t)}) \subseteq \cF \vee \salg_G(\xi_{[0,t)})$. Next, consider $\sP_\xi$. If $s \geq t$ then $(\salg_G(\xi_{[0,s)}) \res f^{-1}([s, 1])) \res f^{-1}([0, t)) \subseteq \cF$. On the other hand, when $s < t$ we have $\salg_G(\xi_{[0,s)}) \subseteq \cF \vee \salg_G(\xi_{[0,t)})$. This establishes the claim.

(v). $\beta_{[0, t)} = \beta \res f^{-1}([0, t))$ is contained in $\cF \vee \salg_G(\xi_{[0,t)})$ by (iv). So $\salg_G(\beta_{[0,t)}) \subseteq \cF \vee \salg_G(\xi_{[0, t)})$. Now take restrictions to $f^{-1}([t, 1])$ and join over $t \in [0, 1]$.
\end{proof}

For $x \in X$ we define a quasi-order on $G$ by setting $u \leq_x v$ if either $f(u^{-1} \cdot x) < f(v^{-1} \cdot x)$ or $f(u^{-1} \cdot x) = f(v^{-1} \cdot x)$ and there is $m \geq 0$ with $T^{-m}(u^{-1} \cdot x) = v^{-1} \cdot x$. When $x$ has trivial stabilizer $\leq_x$ is a total order on $G$. We write $u <_x v$ when $u \leq_x v$ but $\neg (v \leq_x u)$. Its a simple consequence of Lemmas \ref{lem:expgroup} and \ref{lem:eqtest} that for every $u, v \in G$ the set $\{x \in X : u <_x v\}$ is $\cF$-measurable.

\begin{lem} \label{lem:order}
For all $u \neq v \in G$ and every $D \in \cF$ with $D \subseteq \{x \in X : u <_x v\}$ we have
$$\Big[ u \cdot \Big( \textstyle{\bigvee_{k \leq 0} T^k(\xi)} \vee \sP_\xi \Big) \Big] \res D \subseteq v \cdot \Big( \textstyle{\bigvee_{k < 0} T^k(\xi)} \vee \sP_\xi \Big).$$
\end{lem}

\begin{proof}
$D$ is the union of the sets
\begin{equation*}
\begin{array}{cclc}
D_m & = & \{x \in D : u^{-1} \cdot x = T^m(v^{-1} \cdot x)\} \qquad & m \geq 1\\
D_q & = & \{x \in D : f(u^{-1} \cdot x) < q \leq f(v^{-1} \cdot x)\} \qquad & q \in \Q \cap [0, 1].
\end{array}
\end{equation*}
So it suffices to prove the claim for each $D_m$ and $D_q$.

First consider $D_m$. Note $D_m \in \cF$ by Lemma \ref{lem:eqtest}. For any set $A$ we have
\begin{align*}
(u \cdot A) \cap D_m & = \{x \in D_m : u^{-1} \cdot x \in A\}\\
 & = \{x \in D_m : T^m(v^{-1} \cdot x) \in A\} = (v \cdot T^{-m}(A)) \cap D_m.
\end{align*}
Since $\sP_\xi$ is $T$-invariant, it immediately follows that
\begin{align*}
\Big[ u \cdot \Big( \textstyle{\bigvee_{k \leq 0} T^k(\xi)} \vee \sP_\xi \Big) \Big] \res D_m & = \Big[ v \cdot \Big( \textstyle{\bigvee_{k \leq 0} T^{k-m}(\xi)} \vee \sP_\xi \Big) \Big] \res D_m\\
 & \subseteq v \cdot \Big( \textstyle{\bigvee_{k < 0} T^k(\xi)} \vee \sP_\xi \Big).
\end{align*}

Now consider $D_q$. Again note $D_q \in \cF$. Since $u^{-1} \cdot D_q \subseteq f^{-1}([0, q))$, from Lemma \ref{lem:allpast}.(iv) we get
$$\Big( \textstyle{\bigvee_{k \leq 0} T^k(\xi)} \vee \sP_\xi \Big) \res u^{-1} \cdot D_q \subseteq \Big( \cF \vee \salg_G(\xi_{[0,q)}) \Big) \res u^{-1} \cdot D_q.$$
By multiplying throughout by $v^{-1} u$ and using the fact that $v^{-1} \cdot D_q \subseteq f^{-1}([q, 1])$, we obtain
\begin{align*}
\Big[ v^{-1} u \cdot \Big( \textstyle{\bigvee_{k \leq 0} T^k(\xi)} \vee \sP_\xi \Big) \Big] \res v^{-1} \cdot D_q & \subseteq \Big[ v^{-1} u \cdot \Big( \cF \vee \salg_G(\xi_{[0, q)}) \Big) \Big] \res v^{-1} \cdot D_q\\
 & \subseteq \Big( \cF \vee \salg_G(\xi_{[0,q)}) \Big) \res f^{-1}([q, 1])\\
 & \subseteq \sP_\xi.\qedhere
\end{align*}
\end{proof}

In analogy with the earlier Lemma \ref{lem:useS}, we relate $T$-Bernoullicity over $\sP_\xi$ with $G$-Bernoullicity over $\cF$. Below we write $\sinv_T$ for the $\sigma$-algebra of $T$-invariant Borel sets.

\begin{lem} \label{lem:useT}
Let $G \acts (X, \mu)$ be an aperiodic {\pmp} action, and let $\cF$ be a $G$-invariant class-bijective sub-$\sigma$-algebra. Let $T \in [E_G^X]$ and $f : X \rightarrow [0, 1]$ be as given by Corollary \ref{cor:slice}.
\begin{enumerate}
\item[\rm (1)] Let $\alpha$ be a partition of $X$. If $\alpha$ is $G$-Bernoulli over $\cF$ then $\alpha$ is $T$-Bernoulli over $\sP_\alpha$.
\item[\rm (2)] Assume that $\sinv_T \subseteq \cF$. Let $\xi$ be a partition of $X$ and let $\beta \subseteq \salg_T(\xi) \vee \sP_\xi$ be a partition. If $\beta$ is $T$-weakly Bernoulli over $\sinv_T$ and $\salg_T(\beta)$ is independent with $\sP_\xi$ relative to $\sinv_T$, then $\beta$ is $G$-weakly Bernoulli over $\cF$.
\end{enumerate}
\end{lem}

\begin{proof}
Let $G \acts (Y, \nu)$ be the factor associated with $\cF$, and let $\mu = \int \mu_y \ d \nu(y)$ be the disintegration of $\mu$ over $\nu$. Since $\cF$ is class-bijective we have that $\Stab : X \rightarrow \Sub(G)$ is $\cF$-measurable and that $\Stab(y) = \Stab(x)$ for $\mu_y$-almost-every $x \in X$. Also note that $T$ and $f$ both descend to $Y$ since they are $\cF$-expressible and $\cF$-measurable, respectively.

(1). Lemma \ref{lem:useS}.(iii) tells us that $\alpha$ is $T$-Bernoulli over $\cF$. By Lemma \ref{lem:simpbern}, this means the $T$-translates of $\alpha$ are mutually independent and $\salg_T(\alpha)$ is independent with $\cF$. Lemmas \ref{lem:ind} and \ref{lem:simpbern} imply that it will be sufficient to check that $\salg_T(\alpha)$ is independent with $\sP_\alpha$ relative to $\cF$.

Let us express $\sP_\alpha$ as
$$\sP_\alpha = \cF \vee \bigvee_{t \in [0, 1]} \textstyle  \left( \bigvee_{g \in G} g \cdot \alpha_{[0, t)} \right) \res f^{-1}([t, 1]).$$
Notice that $f^{-1}([t, 1])$ is $\mu_y$-conull when $t \leq f(y)$ and otherwise is $\mu_y$-null. Similarly $g \cdot \alpha_{[0, t)}$ equals $g \cdot \alpha$ mod $\mu_y$-null sets when $f(g^{-1} \cdot y) < t$ and is $\mu_y$-trivial otherwise. Thus $\sP_\alpha = \alpha^{\{g \in G : f(g \cdot y) < f(y)\}}$ mod $\mu_y$-null sets.

Let $Y_0$ be the set of $y \in Y$ that have the property that for every finite $W \subseteq G$ that maps injectively to $W \cdot y$, the $W^{-1}$-translates of $\alpha$ are mutually $\mu_y$-independent. Note $\nu(Y_0) = 1$. Fix $y \in Y_0$. Let $P \subseteq \{g \in G : f(g \cdot y) < f(y)\}$ and $K \subseteq \Z$ be finite sets. By Lemma \ref{lem:expgroup} each $T^k$ is $\cF$-expressible, hence there is a $W \subseteq G$ with $W \cdot y = \{T^k(y) : k \in K\}$. In particular $\alpha^W = \bigvee_{k \in K} T^{-k}(\alpha)$ mod $\mu_y$-null sets. Since for every $w \in W$ there is a $k \in K$ with $f(w \cdot y) = f(T^k(y)) = f(y)$, we see that $W \cdot y \cap P \cdot y = \varnothing$. Therefore the partitions $\bigvee_{k \in K} T^{-k}(\alpha) = \alpha^W$ and $\alpha^P$ are $\mu_y$-independent. We conclude that $\salg_T(\alpha)$ is independent with $\sP_\alpha$ relative to $\cF$, completing the proof of (1).

(2). Since the stabilizer map is $\cF$-measurable, it is trivially true that $\beta$ is independent with $\Stab^{-1}(\Borel(\Sub(G)))$ relative to $\cF$. For the same reason, the second condition in Definition \ref{defn:wbover} becomes equivalent to the statement: for $\nu$-almost-every $y \in Y$ and every finite $W \subseteq G$ that maps injectively to $W^{-1} \cdot y$, the $W$-translates of $\beta$ are mutually $\mu_y$-independent.

From Lemma \ref{lem:allpast}.(v) we see that $\salg_T(\beta)$ is independent with $\sP_\beta$ relative to $\sinv_T$. Combined with our other assumption on $\beta$ and Lemma \ref{lem:nonergb}, we get that for almost-every $T$-ergodic component $\nu$ of $\mu$, $\beta$ is $(T, \nu)$-Bernoulli over $\sP_\beta$, meaning that the $T$-translates of $\beta$ are mutually $\nu$-independent and $\salg_T(\beta)$ is $\nu$-independent with $\sP_\beta$. In particular, $\beta$ is $\nu$-independent with $\bigvee_{k < 0} T^k(\beta) \vee \sP_\beta$ for almost-every $T$-ergodic component $\nu$ of $\mu$. Equivalently, $\beta$ is $\mu$-independent with $\bigvee_{k < 0} T^k(\beta) \vee \sP_\beta$ relative to $\sinv_T$. Since $\sinv_T \subseteq \cF$, $\beta$ is in fact $\mu$-independent with $\bigvee_{k < 0} T^k(\beta) \vee \sP_\beta$ relative to $\cF$. Noting the $G$-invariance of $\mu$ and $\cF$, we conclude that for every $g \in G$, $g \cdot \beta$ is $\mu$-independent with $g \cdot (\bigvee_{k < 0} T^k(\beta) \vee \sP_\beta)$ relative to $\cF$.

Let $Y_0$ be the set of $y \in Y$ such that for all $g \in G$ the partition $g \cdot \beta$ is $\mu_y$-independent with $g \cdot (\bigvee_{k < 0} T^k(\beta) \vee \sP_\beta)$. By the previous paragraph $\nu(Y_0) = 1$. Now fix $y \in Y_0$ and fix a finite set $W \subseteq G$ that maps injectively to $W^{-1} \cdot y$. Recall the quasi-ordering map $x \mapsto \leq_x$ described above. This map is $\cF$-measurable and hence constant $\mu_y$-almost-everywhere. Let's denote this constant by $\leq_y$. Since $|W^{-1} \cdot y| = |W| < \infty$, $\leq_y$ restricts to a total order on $W$, and we may enumerate the elements of $W$ as $w_1, \ldots, w_n$ where $w_j <_y w_i$ for all $j < i$. Set $P_i = \{w_j^{-1} : j < i\}$. By Lemma \ref{lem:order} we have, modulo $\mu_y$-null sets,
$$\beta^{P_i} \subseteq w_i \cdot \Big( \textstyle{\bigvee_{k < 0}} T^k(\beta) \vee \sP_\beta \Big).$$
The $\sigma$-algebra on the right is $\mu_y$-independent with $w_i \cdot \beta$. So $w_i \cdot \beta$ is $\mu_y$-independent with $\beta^{P_i}$. This holds for every $i$, so the $W$-translates of $\beta$ are mutually $\mu_y$-independent.
\end{proof}

Finally, we present a crucial lemma relating the entropies of the $G$-action and the $T$-action. In the author's opinion, this lemma represents the primary way that Rokhlin entropy is invoked in the proof of the main theorem.

\begin{lem} \label{lem:sent}
If $\xi$ is a countable partition with $\sH(\xi \given \cF) < \infty$, then
$$\rh_G(\xi \given \cF) \leq \ksh_T(\xi \given \sP_\xi).$$
\end{lem}

\begin{proof}
Pick a finite ordered partition $\cQ = \{Q_i : 1 \leq i \leq n\}$ of $[0, 1]$ into sub-intervals with rational endpoints and with $q < q'$ whenever $j < i$, $q \in Q_j$, $q' \in Q_i$. Set $Y_i = \bigcup_{j < i} Q_j$. Define
$$\sP_\xi^\cQ = \cF \vee \bigvee_{1 \leq i \leq n} \salg_G(\xi_{Y_i}) \res f^{-1}(Q_i).$$
Notice that the sets $f^{-1}(Q_i)$ are $T$-invariant. Let $\mu_i$ denote the normalized restriction of $\mu$ to $f^{-1}(Q_i)$. Since Kolmogorov--{\sinai} entropy is an affine function on the space of $T$-invariant probability measures, we have
\begin{align*}
\sum_{i = 1}^n \ksh_{T,\mu}(\xi_{Q_i} \given \cF \vee \salg_G(\xi_{Y_i})) & = \sum_{i = 1}^n \sum_{j = 1}^n \mu(f^{-1}(Q_j)) \cdot \ksh_{T,\mu_j}(\xi_{Q_i} \given \cF \vee \salg_G(\xi_{Y_i}))\\
 & = \sum_{i = 1}^n \mu(f^{-1}(Q_i)) \cdot \ksh_{T,\mu_i}(\xi_{Q_i} \given \cF \vee \salg_G(\xi_{Y_i}))\\
 & = \sum_{i = 1}^n \mu(f^{-1}(Q_i)) \cdot \ksh_{T,\mu_i}(\xi \given \sP_\xi^\cQ)\\
 & = \ksh_{T,\mu}(\xi \given \sP_\xi^\cQ).
\end{align*}
The $\sigma$-algebras $\cF \vee \salg_G(\xi_{Y_i})$ are $G$-invariant. So by Lemma \ref{lem:expent} and sub-additivity of Rokhlin entropy, we have
\begin{align*}
\rh_{G,\mu}(\xi \given \cF) & \leq \sum_{i = 1}^n \rh_{G,\mu}(\xi_{Q_i} \given \cF \vee \salg_G(\xi_{Y_i}))\\
 & \leq \sum_{i = 1}^n \ksh_{T,\mu}(\xi_{Q_i} \given \cF \vee \salg_G(\xi_{Y_i})) = \ksh_{T,\mu}(\xi \given \sP_\xi^\cQ).
\end{align*}
Since $\sP_\xi$ is the join of the $\sigma$-algebras $\sP_\xi^\cQ$ as $\cQ$ varies and since $\sH(\xi \given \sP_\xi) \leq \sH(\xi \given \cF) < \infty$, we have $\inf_{\cQ} \ksh_T(\xi \given \sP_\xi^\cQ) = \ksh_T(\xi \given \sP_\xi)$. Therefore $\rh_G(\xi \given \cF) \leq \inf_{\cQ} \ksh_T(\xi \given \sP_\xi^\cQ) = \ksh_T(\xi \given \sP_\xi)$.
\end{proof}

\section{The factor theorem} \label{sec:sinai}

Now we present the proof of the main theorem and its variations. We remind the reader that we work with actions that are not necessarily free and build factor maps to Bernoulli shifts that are not necessarily free (see Section \ref{sec:nonfree}).

\begin{thm} \label{thm:main}
For every countably infinite group $G$, every aperiodic ergodic {\pmp} action $G \acts (X, \mu)$, every $G$-invariant sub-$\sigma$-algebra $\Sigma$, and every probability vector $\pv$, if $\sH(\pv) \leq \rh_G(X, \mu \given \Sigma)$ then there is a partition $\alpha$ with $\dist(\alpha) = \pv$ which is $G$-Bernoulli over $\Sigma$.

In fact, every countably infinite group satisfies the relative perturbative factor theorem. Specifically, there is a function $\psd : \R_+^2 \rightarrow \R_+$ such that, for all $G \acts (X, \mu)$, $\Sigma$, and $\pv$ satisfying the assumptions above, if $\sH(\pv) \leq M$ and $\xi$ is a partition satisfying
$$|\dist(\xi) - \pv| + |\sH(\xi) - \sH(\pv)| + |\rh_G(\xi \given \Sigma) - \sH(\pv)| < \psd(M, \epsilon)$$
then the partition $\alpha$ above may be chosen so that $\dB_\mu(\alpha, \xi) < \epsilon$.
\end{thm}

\begin{proof}
The proof of the full theorem has many parts which we break into individual claims. We assume throughout that $G \acts (X, \mu)$, $\Sigma$, and $\pv$ have all of the properties assumed in the statement of the theorem.

\begin{claim} \label{claim1}
Let $\cF$ be a $G$-invariant class-bijective sub-$\sigma$-algebra, let $T \in [E_G^X]$ and $f : X \rightarrow [0, 1]$ be $\cF$-expressible and $\cF$-measurable, respectively, as given by Corollary \ref{cor:slice}, and assume that $\sinv_T \subseteq \cF$. If $\xi$ is a partition satisfying $\sH(\xi \given \cF) < \infty$ and $\sH(\pv) \leq \rh_G(\xi \given \cF)$, then there is a partition $\alpha \subseteq \salg_G(\xi) \vee \cF$ with $\dist(\alpha) = \pv$ such that $\alpha$ is $G$-Bernoulli over $\cF$.
\end{claim}

\begin{proof}[Proof of Claim]
Notice that $\sinv_T \subseteq \cF \subseteq \salg_T(\xi) \vee \sP_\xi$. By working with the $T$-action on the factor associated to $\salg_T(\xi) \vee \sP_\xi$ and by applying the non-ergodic, relative version of {\sinai}'s factor theorem (specifically Corollary \ref{cor:sfolk} with $\Omega = \{\mu \res \salg_T(\xi) \vee \sP_\xi\}$), we obtain a partition $\beta \subseteq \salg_T(\xi) \vee \sP_\xi$ such that $\sH(\beta \given \sinv_T) = \ksh_T(\xi \given \sP_\xi)$, $\beta$ is $T$-weakly Bernoulli over $\sinv_T$, and $\salg_T(\beta)$ is independent with $\sP_\xi$ relative to $\sinv_T$. Lemma \ref{lem:useT}.(2) tells us that $\beta$ is $G$-weakly Bernoulli over $\cF$.

Notice that $\sH(\beta \given \cF) = \sH(\beta \given \sinv_T)$ since $\beta$ is independent with $\sP_\xi$ relative to $\sinv_T$ and $\sinv_T \subseteq \cF \subseteq \sP_\xi$. Also notice that Lemma \ref{lem:allpast}.(iii) gives $\beta \subseteq \salg_T(\xi) \vee \sP_\xi \subseteq \salg_G(\xi) \vee \cF$.

Apply Lemma \ref{lem:grabs} to get an aperiodic $\cF$-expressible $S \in [E_G^X]$ so that $\mu \res \cF$ is $S$-ergodic. By Lemma \ref{lem:useS}, $\beta$ is $S$-weakly Bernoulli over $\cF$ and hence $\ksh_S(\beta \given \cF) = \sH(\beta \given \cF)$. In particular, applying Lemma \ref{lem:sent} we obtain
$$\ksh_S(\beta \given \cF) = \sH(\beta \given \cF) = \sH(\beta \given \sinv_T) = \ksh_T(\xi \given \sP_\xi) \geq \rh_G(\xi \given \cF) \geq \sH(\pv).$$
Lemma \ref{lem:useS} tells us that $\mu \res \salg_S(\beta) \vee \cF$ is $S$-ergodic. Now, working with the ergodic $S$-action on the factor associated to $\salg_S(\beta) \vee \cF$ and applying the relative {\sinai} factor theorem (Lemma \ref{lem:zfolk}), we obtain a partition $\alpha \subseteq \salg_S(\beta) \vee \cF$ such that $\dist(\alpha) = \pv$ and such that $\alpha$ is $S$-Bernoulli over $\cF$. By Lemma \ref{lem:useS} $\alpha$ is $G$-Bernoulli over $\cF$. Furthermore, Lemma \ref{lem:expmove} gives $\salg_S(\beta) \vee \cF \subseteq \salg_G(\beta) \vee \cF$ and hence
\begin{equation*}
\alpha \subseteq \salg_S(\beta) \vee \cF \subseteq \salg_G(\beta) \vee \cF \subseteq \salg_G(\xi) \vee \cF.\qedhere
\end{equation*}
\end{proof}

\begin{claim} \label{claim2}
If $\sH(\pv) < \rh_G(X, \mu \given \Sigma) < \infty$, then there is a partition $\alpha$ with $\dist(\alpha) = \pv$ which is $G$-Bernoulli over $\Sigma$.
\end{claim}

\begin{proof}
By definition of Rokhlin entropy, we can find a partition $\xi$ satisfying $\sH(\xi \given \Sigma) < \infty$ and $\salg_G(\xi) \vee \Sigma = \Borel(X)$. It follows that $\sH(\pv) < \rh_G(X, \mu \given \Sigma) = \rh_G(\xi \given \Sigma)$. Next, apply Lemma \ref{lem:enlarge} to obtain a class-bijective $G$-invariant sub-$\sigma$-algebra $\cF' \supseteq \Sigma$ such that $\sH(\pv) < \rh_G(\xi \given \cF')$. Apply Corollary \ref{cor:slice} to get a $\cF'$-expressible $T \in [E_G^X]$ and a $\cF'$-measurable $f : X \rightarrow [0, 1]$. Set $\cF = \cF' \vee \salg_G(\sinv_T)$. For every partition $\gamma \subseteq \sinv_T$, Lemma \ref{lem:expent} gives $\rh_G(\gamma \given \cF') \leq \ksh_T(\gamma \given \cF') = 0$. From sub-additivity of entropy it follows $\rh_G(\cF \given \cF') = 0$ and thus, by sub-additivity again, $\rh_G(\xi \given \cF) = \rh_G(\xi \given \cF') > \sH(\pv)$. As $T$ is $\cF$-expressible and $f$ is $\cF$-measurable, we can apply Claim \ref{claim1} to obtain a partition $\alpha$ with $\dist(\alpha) = \pv$ such that $\alpha$ is $G$-Bernoulli over $\cF$. In particular, $\alpha$ is $G$-Bernoulli over $\Sigma$.
\end{proof}

\begin{claim} \label{claim3}
If $\rh_G(X, \mu \given \Sigma) = \infty$, then there is a partition $\alpha$ with $\dist(\alpha) = \pv$ which is $G$-Bernoulli over $\Sigma$.
\end{claim}

\begin{proof}[Proof of Claim]
Apply Lemma \ref{lem:enlarge} to obtain a $G$-invariant class-bijective sub-$\sigma$-algebra $\cF' \supseteq \Sigma$ with $\rh_G(X, \mu \given \cF') = \infty$. Apply Corollary \ref{cor:slice} to get a $\cF'$-expressible $T \in [E_G^X]$ and a $\cF'$-measurable $f : X \rightarrow [0, 1]$. Set $\cF = \cF' \vee \salg_G(\sinv_T)$. Just as in the proof of Claim \ref{claim2} we have $\rh_G(\cF \given \cF') = 0$ and thus $\rh_G(X, \mu \given \cF) = \infty$. A difficulty in the infinite entropy case stems from a potential, but unconfirmed, possible defect of Rokhlin entropy: it may be that $\sup \{\rh_G(\xi \given \cF) : \sH(\xi \given \cF) < \infty\}$ is finite even though $\rh_G(X, \mu \given \cF) = \infty$. However, the inverse-limit formula for Rokhlin entropy obtained in \cite[Thm. 6.3]{AS} shows that there is $c > 0$ and an increasing sequence $(\xi_n)_{n \in \N}$ of finite partitions with $\xi_0 = \{X\}$ and with $\rh_G(\xi_n \given \salg_G(\xi_{n-1}) \vee \cF) > c$ for all $n \geq 1$.

Pick a probability vector $\qv$ with $0 < \sH(\qv) < c$. For each $n \in \N$ set $\cF_n = \salg_G(\xi_n) \vee \cF$. For every $n \geq 1$, apply Claim \ref{claim1} to $\qv$, $\xi_n$, $\cF_{n-1}$, $T$, and $f$ to obtain a partition $\alpha_n \subseteq \salg_G(\xi_n) \vee \cF_{n-1} = \cF_n$ such that $\dist(\alpha_n) = \qv$ and $\alpha_n$ is $G$-Bernoulli over $\cF_{n-1}$. Since $\Sigma \subseteq \cF_n$ for all $n \in \N$, it follows that that $\bigvee_{n \geq 1} \alpha_n$ is $G$-Bernoulli over $\Sigma$. Since the $\alpha_n$'s are mutually independent and have the same distribution, every class in $\bigvee_{n \geq 1} \alpha_n$ has measure $0$. Therefore there is a coarser partition $\alpha \leq \bigvee_{n \geq 1} \alpha_n$ with $\dist(\alpha) = \pv$. Of course, $\alpha$ is $G$-Bernoulli over $\Sigma$ as well.
\end{proof}

Let $\psd_\Z : \R^2_+ \rightarrow \R_+$ be as described in Lemma \ref{lem:zfolk}. For $M, \epsilon > 0$ set $\epsilon' =  (1/2) \min((1/5) \psd_\Z(M, \epsilon / 4), \ \epsilon / 4)$ and let $\upsilon = \upsilon(M + 1, \epsilon')$ be as described in Corollary \ref{cor:sfolk}. Finally, define $\psd(M, \epsilon) = (1/2) \min(\upsilon / 6, \ \epsilon', \ 1/3)$.

\begin{claim} \label{claim4}
Suppose that $\cF$ is a $G$-invariant class-bijective sub-$\sigma$-algebra, that $T \in [E_G^X]$ and $f : X \rightarrow [0, 1]$ are $\cF$-expressible and $\cF$-measurable, respectively, as given by Corollary \ref{cor:slice}. Also assume that $\sinv_T \subseteq \cF$. Let $M, \epsilon > 0$ and let $\xi$ be a countable partition of $X$. If $\sH(\pv) \leq M$, $\sH(\pv) < \ksh_T(\xi \given \sP_\xi) + \rh_G(X, \mu \given \salg_G(\xi) \vee \cF)$, and
$$|\dist(\xi) - \pv| + |\sH(\xi) - \sH(\pv)| + |\ksh_T(\xi \given \sP_\xi) - \sH(\pv)| < 3 \psd(M, \epsilon),$$
then there is a partition $\alpha$ with $\dB_\mu(\alpha, \xi) < \epsilon / 2$ and $\dist(\alpha) = \pv$ such that $\alpha$ is $G$-Bernoulli over $\cF$. 
\end{claim}

\begin{proof}[Proof of Claim]
We first prepare a partition $\zeta$ which will be needed a bit later in the argument. If $\sH(\pv) \leq \ksh_T(\xi \given \sP_\xi)$ then we set $\zeta = \{X\}$, and otherwise we apply Claim \ref{claim2} or Claim \ref{claim3} to obtain a partition $\zeta$ which is $G$-Bernoulli over $\salg_G(\xi) \vee \cF$ and satisfies
$$\sH(\zeta) = \sH(\pv) - \ksh_T(\xi \given \sP_\xi) < \rh_G(X, \mu \given \salg_G(\xi) \vee \cF).$$
In either case, we have that $\zeta$ is $G$-Bernoulli over $\salg_G(\xi) \vee \cF$ and $\sH(\pv) \leq \ksh_T(\xi \given \sP_\xi) + \sH(\zeta)$.

Observe that $\sH(\xi) \leq \sH(\pv) + 3 \psd(M, \epsilon) \leq M + 1$ and
$$\ksh_T(\xi \given \sP_\xi) > \sH(\pv) - 3 \psd(M, \epsilon) > \sH(\xi) - 6 \psd(M, \epsilon) > \sH(\xi) - \upsilon.$$
By working with the $T$-action on the factor associated with $\salg_T(\xi) \vee \sP_\xi$, and recalling that $\sinv_T \subseteq \cF \subseteq \salg_T(\xi) \vee \sP_\xi$, we may apply the strongest statement in Corollary \ref{cor:sfolk} (with $\Omega = \{\mu \res \salg_T(\xi) \vee \sP_\xi\}$) to obtain a partition $\beta \subseteq \salg_T(\xi) \vee \sP_\xi$ such that $\beta$ is $T$-weakly Bernoulli over $\sinv_T$, $\salg_T(\beta)$ is independent with $\sP_\xi$ relative to $\sinv_T$, $\sH(\beta \given \sinv_T) = \ksh_T(\xi \given \sP_\xi)$, $\dB_\mu(\beta, \xi) < \epsilon'$, and $|\sH(\beta) - \sH(\xi)| < \epsilon'$. In particular,
\begin{equation} \label{eqn:claim4}
|\dist(\beta) - \pv| + |\sH(\beta) - \sH(\pv)| < 2 \epsilon' + |\dist(\xi) - \pv| + |\sH(\xi) - \sH(\pv)|.
\end{equation}
Lemma \ref{lem:useT}.(2) tells us that $\beta$ is $G$-weakly Bernoulli over $\cF$. Also, $\beta \subseteq \salg_G(\xi) \vee \cF$ by Lemma \ref{lem:allpast}.(iii), and therefore $\zeta$ is $G$-Bernoulli over $\salg_G(\beta) \vee \cF$.

Apply Lemma \ref{lem:grabs} to get an aperiodic $\cF$-expressible $S \in [E_G^X]$ so that $\mu \res \cF$ is $S$-ergodic. By Lemma \ref{lem:useS}, $\beta$ is $S$-weakly Bernoulli over $\cF$. As in Claim \ref{claim1} we have
$$\ksh_S(\beta \given \cF) = \sH(\beta \given \cF) = \sH(\beta \given \sinv_T) = \ksh_T(\xi \given \sP_\xi).$$
Combining the above equation, (\ref{eqn:claim4}), and our assumption we obtain
$$|\dist(\beta) - \pv| + |\sH(\beta) - \sH(\pv)| + |\ksh_S(\beta \given \cF) - \sH(\pv)| < 3 \psd(M, \epsilon) + 2 \epsilon' < \psd_\Z(M, \epsilon / 4).$$
Since $\zeta$ is $G$-Bernoulli over $\salg_G(\beta) \vee \cF$, Lemma \ref{lem:useS} implies that $\zeta$ is $S$-Bernoulli over $\salg_G(\beta) \vee \cF$. So we have
$$\ksh_S(\beta \vee \zeta \given \cF) = \ksh_S(\beta \given \cF) + \sH(\zeta) = \ksh_T(\xi \given \sP_\xi) + \sH(\zeta) \geq \sH(\pv).$$
Lemma \ref{lem:useS} also tells us that $\mu \res \salg_S(\beta \vee \zeta) \vee \cF$ is $S$-ergodic. Working within the ergodic $S$-action on the factor associated to $\salg_S(\beta \vee \zeta) \vee \cF$, the two previous inequalities allow us to apply Lemma \ref{lem:zfolk} to the partition $\beta$. This results in a partition $\alpha \subseteq \salg_S(\beta \vee \zeta) \vee \cF$ which is $S$-Bernoulli over $\cF$ and satisfies $\dist(\alpha) = \pv$ and $\dB_\mu(\alpha, \beta) < \epsilon / 4$. Lemma \ref{lem:useS} implies that $\alpha$ is $G$-Bernoulli over $\cF$. Also, $\dB_\mu(\alpha, \xi) \leq \dB_\mu(\alpha, \beta) + \dB_\mu(\beta, \xi) < \epsilon / 2$.
\end{proof}

\begin{claim} \label{claim5}
If $\sH(\pv) \leq M$ and $\xi$ is a partition of $X$ with
$$|\dist(\xi) - \pv| + |\sH(\xi) - \sH(\pv)| + |\rh_G(\xi \given \Sigma) - \sH(\pv)| < \psd(M, \epsilon),$$
then there is a partition $\alpha$ with $\dB_\mu(\alpha, \xi) < \epsilon$ and $\dist(\alpha) = \pv$ such that $\alpha$ is $G$-Bernoulli over $\Sigma$.
\end{claim}

\begin{proof}[Proof of Claim]
Fix a sequence of probability vectors $\pv^n$ satisfying $\sH(\pv^n) < \sH(\pv)$ and
$$|\pv^n - \pv| + 2 |\sH(\pv^n) - \sH(\pv)| < \min\limits_{m \leq n} (1/2) \psd(M, 2^{-m}).$$
Then $\pv^n \rightarrow \pv$ as $n \rightarrow \infty$ and we have
\begin{equation} \label{eqn:claim5a}
|\pv^n - \pv^{n+1}| + 2 |\sH(\pv^n) - \sH(\pv^{n+1})| < \psd(M, 2^{-n}).
\end{equation}

Apply Lemma \ref{lem:enlarge} to obtain a class-bijective $G$-invariant sub-$\sigma$-algebra $\Sigma_1 \supseteq \Sigma$ with $\rh_G(\Sigma_1 \given \Sigma) < \sH(\pv) - \sH(\pv^1)$. Next, by working within the factor associated to $\Sigma_1$, we can apply Lemma \ref{lem:enlarge} to obtain a class-bijective $G$-invariant sub-$\sigma$-algebra $\Sigma_2$ with $\Sigma_1 \supseteq \Sigma_2 \supseteq \Sigma$ and $\rh_G(\Sigma_2 \given \Sigma) < \sH(\pv) - \sH(\pv^2)$. By repeating this process, we obtain a decreasing sequence of class-bijective $G$-invariant sub-$\sigma$-algebras $\Sigma_n$ such that each one contains $\Sigma$ and $\rh_G(\Sigma_n \given \Sigma) < \sH(\pv) - \sH(\pv^n)$. For each $n$, let $T_n \in [E_G^X]$ and $f_n : X \rightarrow [0, 1]$ be $\Sigma_n$-expressible and $\Sigma_n$-measurable, respectively, as given be Corollary \ref{cor:slice}. Set $\cF_n = \Sigma_n \vee \bigvee_{k \geq n} \salg_G(\sinv_{T_k})$ and note that the $\cF_n$'s are decreasing. Additionally, as explained in the proof of Claim \ref{claim2}, $\rh_G(\cF_n \given \Sigma_n) = 0$ and thus sub-additivity of entropy gives $\rh_G(\cF_n \given \Sigma) = \rh_G(\Sigma_n \given \Sigma) < \sH(\pv) - \sH(\pv^n)$. We will write $\sP_\xi^n$ for the external past of $\xi$ obtained from $T_n$, $f_n$, $\cF_n$.

The sub-additivity of entropy implies that for any $\mathcal{C} \subseteq \Borel(X)$
$$0 \leq \rh_G(\mathcal{C} \given \Sigma) - \rh_G(\mathcal{C} \given \cF_n) \leq \rh_G(\cF_n \given \Sigma) < \sH(\pv) - \sH(\pv^n).$$
In particular, $\rh_G(\xi \given \cF_n)$ converges to $\rh_G(\xi \given \Sigma)$, and using $\mathcal{C} = \Borel(X)$ above we get
\begin{equation} \label{eqn:claim5b}
\sH(\pv^n) < \sH(\pv) - \rh_G(X, \mu \given \Sigma) + \rh_G(X, \mu \given \cF_n) \leq \rh_G(X, \mu \given \cF_n).
\end{equation}

Pick $n$ sufficiently large with $2^{-n} < \epsilon / 2$ and
$$|\dist(\xi) - \pv^n| + |\sH(\xi) - \sH(\pv^n)| + |\rh_G(\xi \given \cF_n) - \sH(\pv^n)| < \psd(M, \epsilon).$$
By our assumptions $\sH(\xi \given \cF_n) \leq \sH(\xi) < \infty$, so from Lemma \ref{lem:sent} we obtain
$$\rh_G(\xi \given \cF_n) \leq \ksh_{T_n}(\xi \given \sP_\xi^n) \leq \sH(\xi).$$
As $|\rh_G(\xi \given \cF_n) - \sH(\xi)| < 2 \psd(M, \epsilon)$, we deduce that
$$|\dist(\xi) - \pv^n| + |\sH(\xi) - \sH(\pv^n)| + |\ksh_{T_n}(\xi \given \sP_\xi^n) - \sH(\pv^n)| < 3 \psd(M, \epsilon).$$
Also, from (\ref{eqn:claim5b}) we obtain
\begin{align*}
\sH(\pv^n) < \rh_G(X, \mu \given \cF_n) & \leq \rh_G(\xi \given \cF_n) + \rh_G(X, \mu \given \salg_G(\xi) \vee \cF_n)\\
 & \leq \ksh_{T_n}(\xi \given \sP_\xi^n) + \rh_G(X, \mu \given \salg_G(\xi) \vee \cF_n).
\end{align*}
By Claim \ref{claim4}, there is a partition $\alpha_n$ with $\dist(\alpha_n) = \pv^n$ and $\dB_\mu(\alpha_n, \xi) < \epsilon / 2$ such that $\alpha_n$ is $G$-Bernoulli over $\cF_n$.

Next we construct a sequence of partitions $\alpha_k$, $k \geq n$, such that $\dB_\mu(\alpha_{k+1}, \alpha_k) < 2^{-k-1}$, $\dist(\alpha_k) = \pv^k$, and such that $\alpha_k$ is $G$-Bernoulli over $\cF_k$. Let $k \geq n$ and inductively assume that $\alpha_k$ has been constructed. As $\cF_{k+1} \subseteq \cF_k$, we see that $\alpha_k$ is $G$-Bernoulli over $\cF_{k+1}$ and hence $T_{k+1}$-Bernoulli over $\sP_{\alpha_k}^{k+1}$ by Lemma \ref{lem:useT}.(1). Consequently, $\ksh_{T_{k+1}}(\alpha_k \given \sP_{\alpha_k}^{k+1}) = \sH(\alpha_k)$ and therefore (\ref{eqn:claim5a}) gives
$$|\dist(\alpha_k) - \pv^{k+1}| + |\sH(\alpha_k) - \sH(\pv^{k+1})| + |\ksh_{T_{k+1}}(\alpha_k \given \sP_{\alpha_k}^{k+1}) - \sH(\pv^{k+1})| < \psd(M, 2^{-k}).$$
Additionally, from (\ref{eqn:claim5b}), sub-additivity, and Lemma \ref{lem:sent} we see that
\begin{align*}
\sH(\pv^{k+1}) < \rh_G(X, \mu \given \cF_{k+1}) & \leq \rh_G(\alpha_k \given \cF_{k+1}) + \rh_G(X, \mu \given \salg_G(\alpha_k) \vee \cF_{k+1})\\
 & \leq \ksh_{T_{k+1}}(\alpha_k \given \sP_{\alpha_k}^{k+1}) + \rh_G(X, \mu \given \salg_G(\alpha_k) \vee \cF_{k+1}).
\end{align*}
Therefore the assumptions of Claim \ref{claim4} are met. So there is a partition $\alpha_{k+1}$ with $\dB_\mu(\alpha_{k+1}, \alpha_k) < 2^{-k-1}$ and $\dist(\alpha_{k+1}) = \pv^{k+1}$ and such that $\alpha_{k+1}$ is $G$-Bernoulli over $\cF_{k+1}$. This completes the construction of the $\alpha_k$'s.

Let $\alpha$ be the limit of the $\alpha_k$'s. Then $\dist(\alpha) = \pv$ and
$$\dB_\mu(\alpha, \xi) \leq \sum_{k \geq n} \dB_\mu(\alpha_{k+1}, \alpha_k) + \dB_\mu(\alpha_n, \xi) < 2^{-n} + \epsilon / 2 < \epsilon.$$
Finally, since $\Sigma$ is contained in every $\cF_k$, $\alpha$ is $G$-Bernoulli over $\Sigma$.
\end{proof}

\begin{claim} \label{claim6}
If $\sH(\pv) \leq \rh_G(X, \mu \given \Sigma)$ then there is a partition $\alpha$ which is $G$-Bernoulli over $\Sigma$ and satisfies $\dist(\alpha) = \pv$.
\end{claim}

\begin{proof}[Proof of Claim]
If $\sH(\pv) < \rh_G(X, \mu \given \Sigma)$ or if $\rh_G(X, \mu \given \Sigma) = \infty$, then we are done by Claims \ref{claim2} and \ref{claim3}. So assume $\sH(\pv) = \rh_G(X, \mu \given \Sigma) < \infty$. In this case we follow the proof of Claim \ref{claim5}, changing only the fourth paragraph (in that paragraph $n \in \N$ was chosen and we found a partition $\alpha_n$ with $\dist(\alpha_n) = \pv^n$ that was $G$-Bernoulli over $\cF_n$ and was close to $\xi$). In place of that paragraph, (\ref{eqn:claim5b}) implies that we can apply Claim \ref{claim2} to obtain a partition $\alpha_1$ with $\dist(\alpha_1) = \pv^1$ that is $G$-Bernoulli over $\cF_1$. Continuing the proof of Claim \ref{claim5}, the following paragraph proceeds to build a Cauchy sequence of partitions $\alpha_n$. Clearly the limit of these partitions, $\alpha$, will be $G$-Bernoulli over $\Sigma$ and satisfy $\dist(\alpha) = \pv$.
\end{proof}

This now completes the proof, as the theorem follows from the final two claims.
\end{proof}

\begin{thm}[The non-ergodic, relative, perturbative factor theorem] \label{thm:nepfact}
Let $X$ be a standard Borel space, let $G \acts X$ be an aperiodic Borel action, and let $\Sigma$ be a countably generated $G$-invariant sub-$\sigma$-algebra. Fix $M, \epsilon > 0$ and a countable partition $\xi$ of $X$. If $\nu \mapsto \pv^\nu$ is a Borel map assigning to each $\nu \in \E_G(X)$ a probability vector $\pv^\nu$ satisfying $\sH(\pv^\nu) \leq \rh_G(X, \nu \given \Sigma)$, then there is a Borel partition $\alpha$ so that for every $\nu \in \E_G(X)$: $\dist_\nu(\alpha) = \pv^\nu$, $\alpha$ is $(G, \nu)$-Bernoulli over $\Sigma$, and if
$$\sH(\pv^\nu) \leq M \quad \text{and} \quad |\dist_\nu(\xi) - \pv^\nu| + |\sH_\nu(\xi) - \sH(\pv^\nu)| + |\rh_{G,\nu}(\xi \given \Sigma) - \sH(\pv^\nu)| < \delta(M, \epsilon)$$
then $\dB_\nu(\alpha, \xi) < \epsilon$.
\end{thm}

\begin{proof}
We claim that with minor (and often obvious) modifications, the proof of Theorem \ref{thm:main} can be applied for all $\nu \in \E_G(X)$ simultaneously. We will explain why this is so and will describe the less obvious modifications.

By the ergodic decomposition theorem \cite{Far62,Var63}, there is a $G$-invariant Borel map $e : X \rightarrow \E_G(X)$ such that $\nu(e^{-1}(\nu)) = 1$ for every $\nu \in \E_G(X)$. Notice that at any time, we can partition $\E_G(X)$ into a countable number of Borel pieces and perform a Borel construction on the $e$-preimage of each of these pieces one at a time. In particular, up to such countable divisions, whenever considering one of the Claims \ref{claim1} through \ref{claim6}, we can assume that the assumptions hold for all $\nu \in \E_G(X)$ and that in Claims \ref{claim4} and \ref{claim5} $M$ and $\epsilon$ do not depend on $\nu \in \E_G(X)$. Similarly, whenever applying Lemma \ref{lem:enlarge} or Corollary \ref{cor:sfolk} we can assume that $M$ and $\epsilon$ are independent of $\nu \in \E_G(X)$.

Next, note that Claims \ref{claim1} through \ref{claim6} should be modified to assume that any sub-$\sigma$-algebras mentioned are countably generated, and similarly only countably generated $\sigma$-algebras should be created in their proofs. In most cases this is automatic, as Lemma \ref{lem:allpast}.(i) implies that external pasts are countably generated and Lemma \ref{lem:enlarge} produces countably generated $\sigma$-algebras. The only exception to this are $\sigma$-algebras of the form $\sinv_T$, for $T \in [E_G^X]$, which are often not countably generated. We replace these $\sigma$-algebras as follows. Let $e_T : X \rightarrow \E_T(X)$ be the ergodic decomposition map for $T$. If $A \subseteq X$ is $T$-invariant then $A' = e_T^{-1}(\{\nu \in \E_T(X) : \nu(A) = 1\})$ satisfies $\nu(A \symd A') = 0$ for all $\nu \in \E_T(X)$ and therefore
$$\forall \mu \in \M_T(X) \quad \mu(A \symd A') = \int_X e_T(x)(A \symd A') \ d \mu(x) = 0.$$
So defining $\sinv_T' = e_T^{-1}(\Borel(\E_T(X)))$, we have that $\sinv_T'$ is countably generated and $\sinv_T' = \sinv_T$ mod $\mu$-null sets for all $\mu \in \M_T(X) \supseteq \M_G(X)$. Thus we can replace $\sinv_T$ with $\sinv_T'$ in the proof, and the difference will be negligible for any measure of interest to us. Additionally, we may enlarge $\sinv_T'$ so that it contains $\sinv_G'$.

It suffices to check that the constructive steps in the proof of Theorem \ref{thm:main} can be performed simultaneously for all ergodic measures in a Borel manner. Let's first inspect the prior results of this paper that are invoked for constructive purposes in the proof of Theorem \ref{thm:main}. Looking over the proof, we see that those results are Lemmas \ref{lem:enlarge}, \ref{lem:grabs}, and \ref{lem:zfolk} and Corollaries \ref{cor:sfolk} and \ref{cor:slice}. As stated, Lemmas \ref{lem:enlarge} and \ref{lem:grabs} and Corollary \ref{cor:slice} are immediately applicable to non-ergodic actions. Additionally, Lemma \ref{lem:zfolk} can be replaced in the non-ergodic setting by Theorem \ref{thm:kira}. Finally, Corollary \ref{cor:sfolk} as stated is applicable to non-ergodic actions but needs some explanation. In all cases where that corollary was applied in the proof of Theorem \ref{thm:main}, we had a partition $\xi$, a class-bijective $G$-invariant sub-$\sigma$-algebra $\cF$, and a $\cF$-expressible $T \in [E_G^X]$ as given by Corollary \ref{cor:slice} with the property that $\sinv_T \subseteq \cF$, and we applied Corollary \ref{cor:sfolk} to the $T$-action on the factor associated with $\salg_T(\xi) \vee \sP_\xi$ with $\Omega = \{\mu \res \salg_T(\xi) \vee \sP_\xi\}$ (for a fixed $\mu \in \E_G(X)$). In the non-ergodic case, $\cF$ will be countably generated. By Lemma \ref{lem:borelfactor} there is an aperiodic Borel bijection $T' : Y \rightarrow Y$ on a standard Borel space $Y$ and a Borel map $\phi : X \rightarrow Y$ with $\phi \circ T = T' \circ \phi$ and $\phi^{-1}(\Borel(Y)) = \salg_T(\xi) \vee \sP_\xi$. Since $\sinv_G' \subseteq \sinv_T' \subseteq \cF \subseteq \salg_T(\xi) \vee \sP_\xi$, Lemma \ref{lem:borelfactor} further asserts the existence of a Borel map $\hat{e} : Y \rightarrow \E_G(X)$ with $\hat{e} \circ \phi = e$. Consequently, the map $\nu \in \E_G(X) \mapsto \phi_*(\nu) \in \M_{T'}(Y)$ is injective and hence its image is Borel \cite[Cor. 15.2]{K95}. So in the non-ergodic case we apply Corollary \ref{cor:sfolk} to $T' : Y \rightarrow Y$ with $\Omega = \{\phi_*(\nu) : \nu \in \E_G(X)\} \subseteq \M_{T'}(Y)$ and with $m : y \in Y \mapsto \phi_*(\hat{e}(y)) \in \Omega$.

Finally, we turn our attention to the constructions that occur properly in the proof of Theorem \ref{thm:main}. Given our prior comments, there are only three instances to consider. The first is in the proof of Claim \ref{claim2}. There, under the assumption that $\sH(\pv) < \rh_G(X, \mu \given \Sigma) < \infty$ we find a partition $\xi$ with $\sH_\mu(\xi \given \Sigma) < \infty$ and $\sH(\pv) < \rh_{G,\mu}(\xi \given \Sigma)$. In the non-ergodic case, we assume that $\sH(\pv^\nu) < \rh_G(X, \nu \given \Sigma) < \infty$ for all $\nu \in \E_G(X)$. Fix a sequence of finite Borel partitions $\xi_n$ with $\xi_{n+1}$ finer than $\xi_n$ and $\salg(\cup_n \xi_n) = \Borel(X)$. By \cite[Thm. 6.3]{AS} we have $\lim_{n \rightarrow \infty} \rh_{G,\nu}(\xi_n \given \Sigma) = \rh_G(X, \nu \given \Sigma)$. So, up to dividing $X$ into a countable number of $G$-invariant Borel sets, we can set $\xi = \xi_n$ for some sufficiently large $n$.

The second instance occurs in the proof of Claim \ref{claim3}, specifically the attainment of $c$ and the sequence of partitions $\xi_n$. So assume that $\rh_G(X, \nu \given \Sigma) = \infty$ for all $\nu \in \E_G(X)$. Fix a sequence of finite Borel partitions $\gamma_n$ with $\gamma_{n+1}$ finer than $\gamma_n$, $\gamma_0 = \{X\}$, and $\salg(\cup_n \gamma_n) = \Borel(X)$. Define $c_\nu = \inf_n \sup_{m \geq n} \rh_{G,\nu}(\gamma_m \given \salg_G(\gamma_n) \vee \Sigma)$. By \cite[Cor. 6.5]{AS} the map $\nu \mapsto c_\nu$ is Borel and by \cite[Thm. 6.3]{AS} $c_\nu > 0$ for all $\nu$. Up to dividing $X$ into a countable number of $G$-invariant Borel sets, we can assume there is $c > 0$ with $c < c_\nu$ for all $\nu \in \E_G(X)$. Set $\xi_0 = \gamma_0 = \{X\}$. Now inductively assume that $\xi_n$ has been defined and that for each $\nu \in \E_G(X)$ $\xi_n$ coincides with some $\gamma_k$ mod $\nu$-null sets. Then $\sup_m \rh_{G,\nu}(\gamma_m \given \salg_G(\xi_n) \vee \Sigma) > c$ for all $\nu \in \E_G(X)$, so we may let $k(\nu) \in \N$ be least with $\rh_{G, \nu}(\gamma_{k(\nu)} \given \salg_G(\xi_n) \vee \Sigma) > c$. Now let $\xi_{n+1}$ be any countable partition that, for every $k$, coincides with $\gamma_k$ on $e^{-1}(\{\nu \in \E_G(X) : k(\nu) = k\})$. Continuing by induction, we obtain an increasing sequence of countable Borel partitions $\xi_n$ with $\sH_\nu(\xi_n) < \infty$ and $\rh_{G,\nu}(\xi_n \given \salg_G(\xi_{n-1}) \vee \Sigma) > c$ for all $n \geq 1$. The rest of the proof of Claim \ref{claim3} can proceed as before, but one must check that $\alpha \leq \bigvee_n \alpha_n$ can be chosen in a Borel manner to have the correct $\nu$-distribution for every $\nu \in \E_G(X)$. There is a simple trick for doing this, the basic idea of which is illustrated at the end of the proof of the next theorem.

The third and final instance to consider is the construction of the probability vectors $\pv^n$ in the proof of Claim \ref{claim5}. In the non-ergodic case we can fix an enumeration of the countable set of probability vectors having rational coordinates, and for each $n \in \N$ and $\nu \in \E_G(X)$ we can choose $\pv^{n,\nu}$ to be the least probability vector in the enumeration satisfying the constraints. Then for every $n$ the map $\nu \mapsto \pv^{n,\nu}$ will be Borel.
\end{proof}

\begin{thm}[The non-ergodic relative factor theorem] \label{thm:nefactor}
Let $X$ be a standard Borel space, let $G \acts X$ be an aperiodic Borel action, and let $\cF$ be a countably generated $G$-invariant sub-$\sigma$-algebra. If $\nu \in \E_G(X) \mapsto \lambda_\nu$ is a Borel map assigning to each $\nu \in \E_G(X)$ a Borel probability measure $\lambda_\nu$ on $[0, 1]$ satisfying $\rh_G(X, \nu \given \cF) \geq \sH([0, 1], \lambda_\nu)$, then there is a $G$-equivariant Borel map $\phi : X \rightarrow [0, 1]^G$ such that, for every $\nu \in \E_G(X)$, $\phi_*(\nu) = \lambda_\nu^{\Stab_*(\nu) \backslash G}$ and $\phi^{-1}(\Borel([0,1]^G))$ is $\nu$-independent with $\cF$ relative to $\Stab^{-1}(\Borel(\Sub(G)))$.
\end{thm}

\begin{proof}
By the ergodic decomposition theorem \cite{Far62,Var63}, there is a $G$-invariant Borel map $e : X \rightarrow \E_G(X)$ such that $\nu(e^{-1}(\nu)) = 1$ for every $\nu \in \E_G(X)$. Setting $X_{\textrm{fin}} = \{x \in X : \sH([0, 1], \lambda_{e(x)}) < \infty\}$ and $X_\infty = \{x \in X : \sH([0, 1], \lambda_{e(x)}) = \infty\}$, we obtain a $G$-invariant Borel partition $\{X_{\textrm{fin}}, X_\infty\}$ of $X$. It suffices to define $\phi$ separately on $X_{\textrm{fin}}$ and on $X_\infty$. So without loss of generality we may assume $X = X_{\textrm{fin}}$ or $X = X_\infty$.

First suppose that $\sH([0, 1], \lambda_\nu) < \infty$ for all $\nu \in \E_G(X)$. Then every $\lambda_\nu$ has countable support. Enumerate the atoms of $\lambda_\nu$ as $a_\nu^1, a_\nu^2, \ldots$ so that $\lambda_\nu(a_\nu^k) \geq \lambda_\nu(a_\nu^{k+1})$ and $a_\nu^k \leq a_\nu^{k+1}$ when $\lambda_\nu(a_\nu^k) = \lambda_\nu(a_\nu^{k+1})$. Note that the numbers $a_\nu^k \in [0, 1]$ can be determined in a Borel way from the $\lambda_\nu$-measures of all the intervals with rational endpoints. Therefore the map $\nu \in \E_G(X) \mapsto (a_\nu^k)_{k \in \N} \in [0, 1]^\N$ is Borel. Set $\pv^\nu = (\lambda_\nu(a_\nu^k))_{k \in \N}$. Since $\lambda_\nu(a_\nu^k)$ is the infimum of the $\lambda_\nu$-measure of the intervals with rational endpoints containing $a_\nu^k$, the map $\nu \mapsto \pv^\nu$ is Borel as well. By Theorem \ref{thm:nepfact} there is a Borel partition $\beta$ with $\beta \in \Bp{\nu}{\pv^\nu}{\cF}$ for all $\nu \in \E_G(X)$. Say $\beta = \{B_n : n \in \N\}$. Define $\phi : X \rightarrow [0, 1]^G$ by $\theta(x)(g) = a_{e(x)}^k$ if $g^{-1} \cdot x \in B_k$. Then Lemmas \ref{lem:stabequal} and \ref{lem:oneone} imply that $\phi$ has the desired properties for every $\nu \in \E_G(X)$.

Now suppose that $\sH([0, 1], \lambda_\nu) = \infty$ for all $\nu \in \E_G(X)$. Then $\rh_G(X, \nu \given \cF) = \infty$ for every $\nu \in \E_G(X)$. Set $\pv = (\frac{1}{2}, \frac{1}{2})$. Inductively for $n \in \N$ we build a sequence of Borel partitions $\beta^n \in \bigcap_{\nu \in \E_G(X)} \Bp{\nu}{\pv}{\cF \vee \bigvee_{k < n} \salg_G(\beta^k)}$. Theorem \ref{thm:nepfact} provides $\beta^0 \in \bigcap_{\nu \in \E_G(X)} \Bp{\nu}{\pv}{\cF}$. Now assume $\beta^0$ through $\beta^{n-1}$ have been constructed. Since $\rh_{G,\nu}(\bigvee_{k < n} \beta^k \given \cF) \leq n \cdot \log(2) < \infty$, by sub-additivity of entropy \cite[Cor. 4.6]{AS} we have $\rh_G(X, \nu \given \cF \vee \bigvee_{k < n} \salg_G(\beta^k)) = \infty$. Thus again Theorem \ref{thm:nepfact} produces $\beta^n$. This constructs $(\beta^n)_{n \in \N}$. Say $\beta^n = \{B_0^n, B_1^n\}$, write $2 = \{0, 1\}$, and define the $G$-equivariant Borel map $\phi_0 : X \rightarrow (2^\N)^G$ by $\phi_0(x)(g)(n) = i$ if and only if $g^{-1} \cdot x \in B_i^n$. Then by Lemmas \ref{lem:stabequal} and \ref{lem:oneone} $\phi_0$ pushes every $\nu$ forward to $(u_2^\N)^{\Stab_*(\nu) \backslash G}$, where $u_2$ is the normalized counting measure on $2$, $\phi_0$ is $\nu$-class-bijective, and $\phi_0^{-1}(\Borel((2^\N)^G)) = \bigvee_{n \in \N} \salg_G(\beta^n)$ is $\nu$-independent with $\cF$ relative to $\Stab^{-1}(\Borel(\Sub(G)))$ for every $\nu$. Now fix a Borel isomorphism $\pi$ from $2^\N$ to $[0, 1]$ which takes $u_2^\N$ to Lebesgue measure $\mathrm{Leb}$. Define the Borel map $(\nu, y) \in \E_G(X) \times [0, 1] \mapsto \psi_\nu(y) \in [0, 1]$ by $\psi_\nu(y) = \inf \{q \in \Q : \lambda_\nu([0, q]) \geq y\}$. Each $\psi_\nu$ pushes $\mathrm{Leb}$ forward to $\lambda_\nu$ since for $r \in [0, 1]$, $B = [0, r]$, and $y \in [0, 1]$ we have
$$\lambda_\nu(B) \geq y \Longleftrightarrow r \geq \psi_\nu(y) \Longleftrightarrow y \in \psi_\nu^{-1}(B) \Longleftrightarrow \mathrm{Leb}(\psi_\nu^{-1}(B)) \geq y.$$
Finally, we define $\phi : X \rightarrow [0, 1]^G$ by $\phi(x)(g) = \psi_{e(x)} \circ \pi \circ \phi_0(x)(g)$. This completes the proof.
\end{proof}

To close this section, we point out the peculiar fact that the residual factor theorem can fail for non-amenable groups. This fact was discovered by Lewis Bowen. We thank him for permission to include it here.

\begin{prop}[Lewis Bowen] \label{prop:resfail}
The residual factor theorem is false for non-amenable groups.
\end{prop}

\begin{proof}
Bowen's original argument was based on the f-invariant (an invariant introduced by Bowen in \cite{B10a}). His argument, in brief, was that if you fix a Bernoulli shift over the rank $2$ free group, then in the space of ergodic self-joinings the generic joining will have f-invariant minus infinity, but if the residual factor theorem were true the generic joining would have f-invariant equal to the entropy of the original Bernoulli shift.

Here we give a different argument that requires less preparation to explain in detail. Let $G = \langle a, b \rangle$ be the rank $2$ free group. Let $(X_i, \mu_i)$ be a copy of $(2^G, u_2^G)$ for $1 \leq i \leq 4$, and let $(X_L, \mu_L)$ be the product of $(X_1, \mu_1)$ with $(X_2, \mu_2)$ and let $(X_R, \mu_R)$ be the product of $(X_3, \mu_3)$ with $(X_4, \mu_4)$. 
Consider the Ornstein--Weiss factor map $\psi : X_3 \rightarrow X_L$ defined by
$$\psi(x_3)(g) = \Big( x_3(g) + x_3(g a), \ x_3(g) + x_3(g b) \Big) \mod 2.$$
It is well known, and a simple exercise to verify, that $\psi$ is a $2$-to-$1$ surjection that pushes the Bernoulli measure $\mu_3$ forward to $\mu_L$. Let $\pi_3 : X_R \rightarrow X_3$ be the projection map and define $\phi : (X_R, \mu_R) \rightarrow (X_L, \mu_L)$ by $\phi = \psi \circ \pi_3$. Now set $\lambda = (\phi \times \id)_*(\mu_R)$. Then $\lambda$ is a joining between the equal entropy Bernoulli shifts $(X_L, \mu_L)$ and $(X_R, \mu_R)$. For each $i$ view $\Borel(X_i) \subseteq \Borel(X_L \times X_R)$ in the natural way. By construction we have that $\cF = \Borel(X_1) \vee \Borel(X_2) \vee \Borel(X_3)$ coincides with $\Borel(X_3)$ mod $\lambda$-null sets. Thus
$$\rh_G(\cF, \lambda) = \rh_G(\Borel(X_3), \lambda) = \log(2) < \log(4).$$
Since $G$ is finitely generated, Rokhlin entropy is upper-semicontinuous in the weak$^*$-topology \cite[Cor. 5.6]{AS}. So there is an open neighborhood $U$ of $\lambda$ with $\rh_G(\cF, \nu) < \log(4)$ for all $\nu \in U$. However, if $f : (X_L, \mu_L) \rightarrow (X_R, \mu_R)$ is a factor map and $\nu = (\id \times f)_*(\mu_L)$ the factor joining, then $\Borel(X_L) = \cF = \Borel(X_L \times X_R)$ mod $\nu$-null sets and hence
$$\rh_G(\cF, \nu) = \rh_G(X_L \times X_R, \nu) = \rh_G(X_L, \mu_L) = \log(4).$$
Therefore there are no factor joinings from $(X_L, \mu_L)$ to $(X_R, \mu_R)$ lying in the open neighborhood $U$ of $\lambda$. Thus the factor joinings from the Bernoulli shift $(X_L, \mu_L)$ to the equal-entropy Bernoulli shift $(X_R, \mu_R)$ do not form a dense set in the space of all joinings of $\mu_L$ with $\mu_R$.
\end{proof}

\section{Further results and discussion} \label{sec:discuss}

A well-known property of classical entropy is that it is additive, meaning if $G$ is amenable, $G \acts (X, \mu)$ is a {\pmp} action, $\cF$ is a $G$-invariant sub-$\sigma$-algebra, and $G \acts (Y, \nu)$ is the factor associated to $\cF$, then
$$\ksh_G(X, \mu) = \ksh_G(Y, \nu) + \ksh_G(X, \mu \given \cF).$$
When the group is non-amenable and Kolmogorov--{\sinai} entropy is replaced with Rokhlin entropy, additivity in general fails and only sub-additivity remains true.

A natural question in regard to the factor theorem is to ask that the Bernoulli factors preserve additivity. In particular, if $G \acts (X, \mu)$ is free and ergodic and $(L, \lambda)$ is a probability space with $\sH(L, \lambda) = \rh_G(X, \mu)$, is there a factor map from $(X, \mu)$ onto $(L^G, \lambda^G)$ that has $0$ relative entropy? Our methods do not seem to answer either of these questions. However, we mention that the answers are positive if arbitrarily small errors are allowed.

\begin{thm} \label{thm:sinaiadd}
Let $G \acts (X, \mu)$ be an aperiodic ergodic {\pmp} action, let $\Sigma$ be a $G$-invariant sub-$\sigma$-algebra, and let $\pv$ be a finite probability vector with $\sH(\pv) \leq \rh_G(X, \mu \given \Sigma) < \infty$. For every $\epsilon > 0$ there is a partition $\alpha$ with $\dist(\alpha) = \pv$ that is $G$-Bernoulli over $\Sigma$ and satisfies
$$\sH(\alpha) + \rh_G(X, \mu \given \salg_G(\alpha) \vee \Sigma) < \rh_G(X, \mu \given \Sigma) + \epsilon.$$
\end{thm}

\begin{proof}
Recall the metric $\dR_\mu$ introduced in Section \ref{sec:nonerg}. On the space of partitions of cardinality $|\pv|$ the metrics $\dB_\mu$ and $\dR_\mu$ are uniformly equivalent \cite[Fact 1.7.7]{Do11}. So there is $\epsilon' > 0$ with $\dR_\mu(\alpha, \beta) < \epsilon / 2$ whenever $|\alpha| = |\beta| = |\pv|$ and $\dB_\mu(\alpha, \beta) < \epsilon'$.  Let $\delta : \R^2_+ \rightarrow \R$ be as in Theorem \ref{thm:main} and pick $0 < \kappa < \min(\delta(\sH(\pv), \epsilon'), \epsilon / 2)$. Pick a probability vector $\qv$ that refines $\pv$ and satisfies $\rh_G(X, \mu \given \Sigma) < \sH(\qv) < \rh_G(X, \mu \given \Sigma) + \kappa$. By the generalized Krieger generator theorem \cite{S1} we can find a partition $\cQ$ with $\salg_G(\cQ) \vee \Sigma = \Borel(X)$ and $\dist(\cQ) = \qv$. Let $\cP$ be the coarsening of $\cQ$ associated with the coarsening $\pv$ of $\qv$. Sub-additivity and the inequality $\rh_G(X, \mu \given \salg_G(\cP) \vee \Sigma) \leq \sH(\cQ \given \cP)$ give
$$\sH(\cP) + \sH(\cQ \given \cP) - \kappa = \sH(\cQ) - \kappa < \rh_G(X, \mu \given \Sigma) \leq \rh_G(\cP \given \Sigma) + \sH(\cQ \given \cP),$$
hence $\rh_G(\cP \given \Sigma) > \sH(\cP) - \kappa$. It follows that
$$|\dist(\cP) - \pv| + |\sH(\cP) - \sH(\pv)| + |\rh_G(\cP \given \Sigma) - \sH(\pv)| < \delta(\sH(\pv), \epsilon').$$
By our main theorem, there is a partition $\alpha$ with $\dist(\alpha) = \pv$ and $\dB_\mu(\alpha, \cP) < \epsilon'$ which is $G$-Bernoulli over $\Sigma$. Finally, by \cite[Lem. 5.2]{AS} we have
\begin{align*}
\rh_G(X, \mu \given \salg_G(\alpha) \vee \Sigma) & \leq \dR_\mu(\alpha, \cP) + \rh_G(X, \mu \given \salg_G(\cP) \vee \Sigma)\\
 & < \frac{\epsilon}{2} + \sH(\cQ \given \cP)\\
 & = \frac{\epsilon}{2} + \sH(\cQ) - \sH(\pv)\\
 & < \epsilon + \rh_G(X, \mu \given \Sigma) - \sH(\alpha).\qedhere
\end{align*}
\end{proof}

\begin{cor} \label{cor:fullfactor}
Let $G \acts (X, \mu)$ be an aperiodic ergodic {\pmp} action, and let $\Sigma$ be a $G$-invariant sub-$\sigma$-algebra. For every $\epsilon > 0$ there is a partition $\alpha$ that is $G$-Bernoulli over $\Sigma$ and satisfies $\rh_G(X, \mu \given \salg_G(\alpha) \vee \Sigma) = 0$ and $\sH(\alpha) \leq \rh_G(X, \mu \given \Sigma) + \epsilon$.
\end{cor}

\begin{proof}
First suppose that $\rh_G(X, \mu \given \Sigma) < \infty$. By Theorem \ref{thm:sinaiadd} there is a partition $\alpha_1$ with $\sH(\alpha_1) = \rh_G(X, \mu \given \Sigma)$ that is $G$-Bernoulli over $\Sigma$ and satisfies $\rh_G(X, \mu \given \salg_G(\alpha_1) \vee \Sigma) < \epsilon / 2$. Inductively assume that $\alpha_n$ has been constructed such that $\alpha_n$ is $G$-Bernoulli over $\Sigma$, $\rh_G(X, \mu \given \salg_G(\alpha_n) \vee \Sigma) < \epsilon / 2^n$, and
$$\sH(\alpha_n) \leq \rh_G(X, \mu \given \Sigma) + (1 - 2^{-n + 1}) \epsilon.$$
Apply Theorem \ref{thm:sinaiadd} to get $\beta$ with $\sH(\beta) = \rh_G(X, \mu \given \salg_G(\alpha_n) \vee \Sigma) < \epsilon / 2^n$ such that $\beta$ is $G$-Bernoulli over $\salg_G(\alpha_n) \vee \Sigma$ and satisfies $\rh_G(X, \mu \given \salg_G(\beta \vee \alpha_n) \vee \Sigma) < \epsilon / 2^{n+1}$. Now set $\alpha_{n+1} = \beta \vee \alpha_n$. This completes the construction of the $\alpha_n$'s. Setting $\alpha = \bigvee_{n \in \N} \alpha_n$ completes the proof in this case.

Now suppose that $\rh_G(X, \mu \given \Sigma) = \infty$. Fix an increasing sequence of finite partitions $(\xi_n)_{n \in \N}$ with $\xi_0 = \{X\}$ and $\bigvee_n \salg(\xi_n) = \Borel(X)$. We may use Lemma \ref{lem:enlarge} to choose $\xi_1$ so that $\salg_G(\xi_1)$ is class-bijective. Let $G \acts (X_n, \mu_n)$ be the factor of $(X, \mu)$ associated with $\salg_G(\xi_n) \vee \Sigma$, say via $\phi_n : (X, \mu) \rightarrow (X_n, \mu_n)$, and let $\cF_n$ be the image of $\salg_G(\xi_{n-1}) \vee \Sigma$ for $n \geq 1$. For each $n \geq 1$ we can apply the previous paragraph to obtain a partition $\beta_n$ of $X_n$ that is $G$-Bernoulli over $\cF_n$ and satisfies $\rh_G(X_n, \mu_n \given \salg_G(\beta_n) \vee \cF_n) = 0$. Then $\alpha = \bigvee_{n \geq	 1} \phi_n^{-1}(\beta_n)$ is $G$-Bernoulli over $\Sigma$ and
\begin{align*}
\rh_G(X, \mu \given \salg_G(\alpha) \vee \Sigma) & \leq \sum_{n \geq 1} \rh_G(\xi_n \given \salg_G(\alpha \vee \xi_{n-1}) \vee \Sigma)\\
 & \leq \sum_{n \geq 1} \rh_G(X_n, \mu_n \given \salg_G(\beta_n) \vee \cF_n) = 0.\qedhere
\end{align*}
\end{proof}

While in general Rokhlin entropy is not additive for factor maps (indeed factor actions can have greater entropy than their source), the above corollary does suggest an entropy-style structure theory of class-bijective factor maps. Specifically, for two actions $G \acts (X, \mu)$ and $G \acts (Y, \nu)$ and a factor map $\phi : (X, \mu) \rightarrow (Y, \nu)$, call $\phi$ \emph{entropy-additive}\footnote{Recall that entropy-additive maps must be class-bijective when $\rh_G(Y, \nu) > 0$ \cite[Thm. 9.2]{AS}.} if
$$\rh_G(X, \mu) = \rh_G(Y, \nu) + \rh_G(X, \mu \given \phi^{-1}(\Borel(Y))),$$
and say $\phi$ \emph{has $0$ relative entropy} if $\rh_G(X, \mu \given \phi^{-1}(\Borel(Y))) = 0$. The above corollary suggests that for every class-bijective factor $\phi : (X, \mu) \rightarrow (Y, \nu)$ there is an action $G \acts (Z, \eta)$ and maps $(X, \mu) \overset{r}{\rightarrow} (Z, \eta) \overset{a}{\rightarrow} (Y, \nu)$ such that $a \circ r = \phi$, $r$ has $0$ relative entropy, and $a$ is entropy-additive. This will be true provided that entropy is additive for class-bijective Bernoulli extensions.

\begin{cor}
Let $G \acts (X, \mu)$ be an aperiodic ergodic {\pmp} action with stabilizer type $\theta$ and let $G \acts (Y, \nu)$ be a class-bijective factor via $\phi : (X, \mu) \rightarrow (Y, \nu)$. Let $\nu = \int_{\Sub(G)} \nu_\Gamma \ d \theta(\Gamma)$ be the disintegration of $\nu$ over $\theta$. Then there exists a class-bijective Bernoulli extension of $(Y, \nu)$
$$G \acts (Z, \eta) = G \acts \Bigg(L^G \times Y, \int_{\Sub(G)} \lambda^{\Gamma \backslash G} \times \nu_\Gamma \ d \theta(\Gamma) \Bigg)$$
and maps $(X, \mu) \overset{r}{\rightarrow} (Z, \eta) \overset{a}{\rightarrow} (Y, \nu)$ such that $a \circ r = \phi$, $r$ has $0$ relative entropy, and $a : Z = L^G \times Y \rightarrow Y$ is the projection map.
\end{cor}

\begin{proof}
This is immediate from applying Corollary \ref{cor:fullfactor} with $\Sigma = \phi^{-1}(\Borel(Y))$ and noting Lemma \ref{lem:oneone}.
\end{proof}

A well-known component of classical Ornstein theory is that Bernoulli measures are precisely those measures which are finitely determined. As a consequence of the perturbative factor theorem, part of this statement extends to the non-amenable realm. Recall that for a finite set $K$ and two measures $\mu, \nu \in \M_G(K^G)$, the \emph{$\bar{d}$-distance} between $\mu$ and $\nu$, denoted $\bar{d}(\mu, \nu)$, is the infimum of
$$\lambda(\{(x, y) \in K^G \times K^G : x(1_G) \neq y(1_G)\})$$
as $\lambda \in \M_G(K^G \times K^G)$ varies over all joinings of $\mu$ with $\nu$. A measure $\mu \in \M_G(K^G)$ is \emph{finitely determined} if for every $\epsilon > 0$ there is $\delta > 0$ and a weak$^*$-open neighborhood $U$ of $\mu$ such that for every $\nu \in \M_G(K^G) \cap U$ with $|\rh_G(L^G, \nu) - \rh_G(L^G, \mu)| < \delta$ and $\Stab_*(\nu) = \Stab_*(\mu)$ we have $\bar{d}(\mu, \nu) < \epsilon$.

\begin{thm}
Let $G$ be a countably infinite group, let $(K, \kappa)$ be a finite probability space, and let $\theta$ be an IRS supported on infinite-index subgroups. If $\rh_G(K^G, \kappa^{\theta \backslash G}) = \sH(K, \kappa)$ then $\kappa^{\theta \backslash G}$ is finitely determined.
\end{thm}

\begin{proof}
Let $\delta : \R^2_+ \rightarrow \R$ be as in Theorem \ref{thm:main}. Identify $\Prob(K)$ with the set of $K$-indexed probability vectors, and fix a probability measure $\omega$ on $K$ satisfying $\sH(\omega) < \sH(\kappa)$ and
\begin{equation} \label{eqn:fd1}
|\omega - \kappa| + 2 |\sH(\omega) - \sH(\kappa)| < (1/2) \delta(\sH(\kappa), \epsilon / 2).
\end{equation}

Let $\beta = \{B_k : k \in K\}$ be the canonical partition of $K^G$, where $B_k = \{x \in K^G : x(1_G) = k\}$. Since $K$ is finite, there is a weak$^*$-open neighborhood $U$ of $\kappa^{\theta \backslash G}$ such that
\begin{equation} \label{eqn:fd2}
\forall \nu \in U \qquad |\dist_\nu(\beta) - \kappa| + |\sH_\nu(\beta) - \sH(\kappa)| < (1/3) \delta(\sH(\kappa), \epsilon / 2).
\end{equation}
Now consider a measure $\nu \in \M_G(K^G) \cap U$ satisfying $\Stab_*(\nu) = \theta$ and
\begin{equation} \label{eqn:fd3}
|\rh_G(K^G, \nu) - \rh_G(K^G, \kappa^{\theta \backslash G})| < \min(|\sH(\omega) - \sH(\kappa)|, (1/6) \delta(\sH(\kappa), \epsilon / 2)).
\end{equation}
Since $\rh_G(K^G, \kappa^{\theta \backslash G}) = \sH(\kappa)$ we have $\rh_G(\beta, \nu) = \rh_G(K^G, \nu) > \sH(\omega)$. Also (\ref{eqn:fd1}), (\ref{eqn:fd2}), (\ref{eqn:fd3}) imply
$$|\dist_\nu(\beta) - \omega| + |\sH_\nu(\beta) - \sH(\omega)| + |\rh_G(\beta, \nu) - \sH(\omega)| < \delta(\sH(\kappa), \epsilon / 2).$$
By Theorem \ref{thm:main} there is a partition $\beta' = \{B_k' : k \in K\}$ with $\dist(\beta') = \omega$ that is $(G, \nu)$-Bernoulli and satisfies $\dB_\nu(\beta, \beta') < \epsilon / 2$. Define $\phi : K^G \rightarrow K^G$ by $\phi(x)(g) = k \Leftrightarrow g^{-1} \cdot x \in B_k'$ and set $\lambda = (\id \times \phi)_*(\nu)$. Then $\lambda$ is a joining of $\nu$ with $\omega^{\theta \backslash G}$ and
\begin{align*}
\lambda(\{(x, y) \in K^G \times K^G : x(1_G) \neq y(1_G)\}) & = \nu(\{x \in K^G : x(1_G) \neq \phi(x)(1_G)\})\\
 & = \dB_\nu(\beta, \beta') < \epsilon/2.
\end{align*}
Thus $\bar{d}(\nu, \omega^{\theta \backslash G}) < \epsilon / 2$. Of course, $\kappa^{\theta \backslash G}$ satisfies all of the assumptions on $\nu$, so we similarly have $\bar{d}(\kappa^{\theta \backslash G}, \omega^{\theta \backslash G}) < \epsilon / 2$. Now its a simple exercise in measure theory to see that $\bar{d}$ is indeed a metric and satisfies the triangle identity. Hence $\bar{d}(\nu, \kappa^{\theta \backslash G}) < \epsilon$.
\end{proof}

Finally, we mention that many delicate constructions are performed on Bernoulli shifts due to the concrete, combinatorial nature of these actions. By our theorem, these constructed objects can be pulled back to actions of positive entropy. We mention one explicit example here that we think is quite intriguing. Consider an ergodic {\pmp} action $G \acts (X, \mu)$ with the property that the induced orbit equivalence relation $E_G^X$ is not $\mu$-hyperfinite. The \emph{measurable von Neumann--Day conjecture} then asserts that there exists a {\pmp} free action $F_2 \acts (X, \mu)$, where $F_2$ is the rank $2$ free group, such that almost-every $F_2$-orbit is contained in a $G$-orbit. This conjecture was raised by Gaboriau in \cite[p. 24, Question 5.16]{G02} (and again raised at the end of \cite{GL09}). In \cite{GL09} Gaboriau and Lyons proved that every non-amenable group admits a (free) Bernoulli shift action satisfying the conjecture, and in \cite{B17} Bowen proves that all (free) Bernoulli shifts over non-amenable groups satisfy the conjecture. For non-free actions, Bowen, Hoff, and Ioana have proved that the type-$\theta$ Bernoulli shift $G \acts ([0,1]^G, \lambda^{\theta \backslash G})$, where $\lambda$ is Lebesgue measure, satisfies the conjecture whenever the induced orbit equivalence relation is not almost-everywhere hyperfinite \cite{BHI}.

\begin{cor} \label{cor:vn}
Let $G \acts (X, \mu)$ be an ergodic {\pmp} action such that the induced orbit equivalence relation $E_G^X$ is not $\mu$-hyperfinite. Assume either: (i) the action is free and $\rh_G(X, \mu) > 0$, or (ii) $\rh_G(X, \mu) = \infty$. Then $G \acts (X, \mu)$ satisfies the measurable von Neumann--Day conjecture.
\end{cor}

\end{document}